%% file: D-GE.tex
\documentclass[10pt]{amsart}

\usepackage{etex}

\usepackage[latin1]{inputenc}  
\usepackage{lmodern}  
\usepackage{amsmath}
\usepackage{amssymb}
\usepackage{commath}

\newenvironment{psmallmatrix}
  {\left(\begin{smallmatrix}}
  {\end{smallmatrix}\right)}

\DeclareMathAlphabet{\mathpzc}{OT1}{pzc}{m}{it}
\usepackage[all]{xy}
\usepackage{latexsym}
\usepackage{graphicx}
\usepackage{tikz}

\usepackage{graphics}

\usepackage[left=2.6cm,right=2.6cm,top=2.8cm,bottom=3cm]{geometry}
\usepackage{amssymb}
\usepackage{etex} 
\usepackage{tikz-cd}
\usepackage{mathtools}
\usepackage{xifthen}
\usepackage{stmaryrd}  
\usepackage{enumitem}
\usepackage{multicol}
\usepackage[all]{xy}
\usepackage{xcolor}
\usepackage{color}
\usepackage{float}
\usepackage{epsfig}  
\usepackage{amsfonts}
\usepackage{amsfonts}
\usepackage{subfigure}
\usepackage{ dsfont }

\usetikzlibrary{matrix,arrows,backgrounds,shapes.misc,shapes.geometric,patterns,calc,positioning}
\usetikzlibrary{calc,shapes}
\usetikzlibrary{decorations.pathmorphing}

\usepackage[colorlinks,citecolor=purple,urlcolor=blue,bookmarks=false,hypertexnames=true]{hyperref}

\theoremstyle{theorem}
\newtheorem{theorem}{Theorem}[section]         
\newtheorem{lemma}[theorem]{Lemma}

\newtheorem*{Teo}{Theorem} 
\newtheorem*{Pro}{Proposition}
\newtheorem{prop}[theorem]{Proposition}

\theoremstyle{remark}
\newtheorem{remark}[theorem]{Remark}

\newtheorem{ex}[theorem]{Example}

\theoremstyle{definition}
\newtheorem{defi}[theorem]{Definition} 

\usepackage{color}

\newcommand{\Dd}{\mathcal{D}}
\newcommand{\modu}{\mathrm{mod}\hspace{1pt}}
\newcommand{\Cc}{\mathcal{C}}

\newcommand{\cA}{\mathcal{A}}
\newcommand{\Tt}{\mathcal{T}}
\newcommand{\x}{\mathbf{x}}

\newcommand{\ra}{\rightarrow}

\newcommand{\Ext}{\mathrm{Ext}\hspace{.5pt}}
\newcommand{\rep}{\mathrm{rep}\hspace{1.2pt}}
\newcommand{\mmod}{\mathrm{mod}\hspace{1pt}}
\newcommand{\Hom}{\mathrm{Hom}\hspace{1pt}}
\newcommand{\End}{\mathrm{End}\hspace{1pt}}
\newcommand{\add}{\mathrm{add}\hspace{1pt}}

\newcommand{\udim}{\underline{\mathrm{dim}}\hspace{1pt}}

\newcommand{\ttop}{\mathrm{top}\hspace{1pt}}
\newcommand{\soc}{\mathrm{soc}\hspace{1pt}}

\title{Extensions in Jacobian algebras via punctured skein relations}
\author{Salom\'on Dom\'inguez - Ana Garc\'ia Elsener}

\thanks{
\\
The first author works for the University An\'auhac and has received travel grants from said institution during this project. The second author worked for the University of Glasgow while writing this article, and thanks professor Karin Baur and the grant Dimer algebras on surfaces P 30549 FWF 2017-2020 for their support on this project.}

\AtEndDocument{\bigskip{\footnotesize%

  \textsc{Economics and Business School, Universidad An\'ahuac M\'exico. Av. Universidad An\'ahuac 46, Col. Lomas An\'ahuac Huixquilucan, Estado de M\'exico, C.P. 52786.} \par
  \textit{E-mail address} \texttt{salomon.dominguez@anahuac.mx} \par
  \vspace{10pt}
  
  \textsc{Facultad de Cs. Exactas y Naturales, Universidad Nacional de Mar del Plata, Dean Funes 3350, 7600, Bs. As. Argentina} \par
  \textit{E-mail address} \texttt{anaelsener@gmail.com}

}}

\begin{document}

\maketitle 
\begin{center}
{\it Dedicated to the memory of Andrzej Skowro\'nski.}
\end{center}

\begin{abstract}
Given a Jacobian algebra arising from the punctured disk, we show that all non-split extensions can be found using the tagged arcs and skein relations previously developed in cluster algebras theory. Our geometric interpretation can be used to find non-split extensions over other Jacobian algebras arising from surfaces with punctures. We show examples in type $D$ and in other punctured surfaces.
\end{abstract}

\section{Introduction}

Jacobian algebras arising from surfaces were defined by Labardini-Fragoso in \cite{L} building in the works \cite{DWZ} and \cite{FST}. The reader can also find an overview of Labardini's work on Jacobian algebras here \cite{labardini2016}. Given a compact Riemann surface with some disk removed, in order to create boundaries, and adding marked points in each boundary component and in the interior of the surface, we can define a triangulation. A triangulation is a finite set of curves (up to isotopy) splitting the surface into (ideal) triangles. A marked point in the interior of the surface is called a puncture. The simplest punctured surface is the once-punctured disk. This surface is associated to the cluster category of type $D$, see \cite{S}. A triangulation of the punctured disk defines a cluster-tilted algebra of type $D$ \cite{BMRRT,bmr}.  More generally, a punctured surface defines a cluster category \cite{QZ}, and a triangulation of the surface defines a Jacobian algebra in the sense of \cite{L}.

\smallskip

Br{\"u}stle-Zhang \cite{BZ} introduced a geometric interpretation of Amiot's generalized cluster categories \cite{Am}.
For the unpunctured version, Jacobian algebras arising from surface triangulations are gentle algebras and were studied by Assem et. all in \cite{ABCP}, where string modules are defined by arcs that do not belong to the triangulation. Gentle algebras and their module categories are very well understood. Building on these works, and on their knowledge on gentle algebras and cluster algebras arising from unpunctured surfaces, Canakci and Schroll \cite{CSh} study non-split extensions in module categories of gentle algebras arising from surface triangulations. They show how to find non-split extensions and, as a consequence, they find non-split triangles in the cluster category. The non-split extensions and triangles arise from skein relations over the cluster algebra, this is viewed as a geometric operation that smooths arc crossings in the interior of the surface. 



Skein relations were used to study bases of cluster algebras arising from surfaces (and to prove the Fomin-Zelevinsky positivity conjecture over said algebras). See \cite{MW,MSW}. In the trivial coefficients case, skein relations are well understood, and formulas on cluster algebras arising from surfaces with punctures are known. If no frozen variables exist, by the decorated Teichm{\"u}ller spaces associated to surfaces 
(even in the presence of punctures), Ptolemy relations hold true by construction.




In this work we use skein relations, that appear as cluster algebra identities over a once punctured surface, to find all non-split extensions over a particular cluster tilted algebra of type $D$. Then we find triangles in the cluster category and we characterize when these triangles induce non-split extensions over Jacobian algebras. 

\smallskip

The structure of the paper goes as follows. After finding all non split extensions for a particular triangulation of the punctured disk in Section \ref{section type-D-cluster}, we recall skein relations and other cluster algebra formulas we will need in Section \ref{sec:skein-rel}. We see how the extensions in Section \ref{section type-D-cluster} and skein relations are related in Section \ref{sec:extensions}. Then we present our main results in Section \ref{resultados-principales}. 

\begin{Teo} (Theorem \ref{theorem-all-triangles})
All non-trivial triangles with indecomposable extreme terms in the cluster category of type $D$ can be obtained via (punctured) skein relations.
The triangles have the form
\[ X_\alpha \to X_C \to X_\beta \to X_{\alpha} [1], \]
where $C$ is a multicurve in $\alpha^+ \beta$.
\end{Teo}

We will use this theorem to find explicit non-split extensions. In the case of Jacobian algebras arising from the punctured disk, in the next proposition we find all the explicit non-split extensions with indecomposable extreme terms. The next result holds in general for cluster categories arising from surfaces for an adequate $C$ such that the triangle $ X_\alpha \to X_C \to X_\beta \to X_{\alpha} [1] $ exists. We include the proof, based on Lemma \ref{lemma-existe-sec}, since were not able to find it in the literature. 

\begin{Pro} (Proposition \ref{prop:preserva-dimension}) Let $\alpha$ and $\beta$ be arcs on the (punctured) disk such that $e(\alpha, \beta) \geq 1$. Let $ X_\alpha \to X_C \to X_\beta \to X_{\alpha} [1] $ be a non-split triangle with indecomposable extremes $X_\alpha$ and $X_\beta$ in the cluster category, where $C \in \alpha^+ \beta$, and let $\mathcal{T}$ be a triangulation. 
Then there is a non-split extension 
\[ 0 \to M_{\alpha} \to M_{C} \to M_{\beta} \to 0 \]
if and only if $d(C) = d(\{ \alpha, \beta\})$,  where $d$ denotes the total number of  crossings of the multicurve with the triangulation.\end{Pro}

We conclude with a list of examples over Jacobian algebras arising from the punctured disk and over other surfaces in Section \ref{EXAMPLES}.

\section{Triangles in the cluster category of type $D$}\label{section type-D-cluster}

In this section we find all the non-split extensions for a particular orientation of a quiver $Q$ of type $D$. That is, all the non split extensions with indecomposable extremes in the category of representations $\rep Q$. Then we use those extensions to find all the non-split triangles in the cluster category of type $D$. The reader is referred to \cite{ASS,S14} for basic definitions that will be used in this section.

\subsection{Non-split extensions over the hereditary algebra}\label{hereditary}

Let $H$ be the hereditary algebra $ K Q$ where $Q$ is some orientation of the $D_n$ quiver ($n \geq 4$)

\begin{center}
\begin{tikzcd}n \arrow[dash]{rd} \\  
& n-2 \arrow[dash]{r} & \cdots \arrow[dash]{r} & 1 \\
n-1 \arrow[dash]{ru}  \end{tikzcd}
\end{center}

The algebra $H$ is hereditary and representation-finite. A basis for all non-split extensions between indecomposable modules, when $H$ is as above or it is of type $A$ (i.e. $Q$ is a linear graph) was given in \cite{Br}. Moreover almost all non-split extensions were described in the article. 
In the next lemma we find another non-split extension.

From now on let $Q$ be the type $D_n$ quiver, taking an orientation for the graph above where \emph{all the arrows point right}, and consider $K$ an algebraically closed field of $\mathrm{char} K \neq 2$.

\begin{lemma}\label{new-sequences} There exist a non-split extension between indecomposable representations in $\mmod H$ given in Figure \ref{sequence}.
\end{lemma}

\begin{proof} The non-split extension of indecomposable representations is presented in the next figure, where $1 < r < s < i < n-1$, $I_n$ is the $n \times n$ identity matrix, and the subscripts for each $K^j$ indicate the vertex in the oriented quiver.

\begin{figure}[H]
{\hspace*{-3em}{\begin{tikzpicture}[scale=0.9]
\draw (-1.5,3.5)node{$0$};
\draw (-1.5,6.5)node{$0$};
\draw (0,5)node{$0$};
\draw (1.7,5)node{$0$};
\draw (2.25,5)node{$\cdots$};
\draw (2.8,5)node{$0$};
\draw (4.25,5)node{$K_{i}$};
\draw (5.75,5)node{$K$};
\draw (6.3,5)node{$\cdots$};
\draw (6.9,5)node{$K_{s}$};
\draw (8.4,5)node{$K$};
\draw (9.75,5)node{$K$};
\draw (10.25,5)node{$\cdots$};
\draw (11.05,5)node{$K_{r+1}$};
\draw (12.65,5)node{$K_{r}$};
\draw (13.25,5)node{$\cdots$};
\draw (13.65,5)node{$K_{2}$};
\draw (15,5)node{$K_{1}$};
\draw[->][line width=1pt] (-1.25,3.75) -- (-0.4,4.75); 
\draw[->][line width=1pt] (-1.4,6.25) -- (-0.4,5.25);
\draw[->][line width=1pt] (0.4,5) -- (1.25,5); 
 \draw[->][line width=1pt] (3.05,5) -- (4,5); 
 \draw[->][line width=1pt] (4.55,5) -- (5.5,5)
node[pos=0.5,above] {$1$};      
 \draw[->][line width=1pt] (7.3,5) -- (8.15,5)
node[pos=0.5,above] {$1$};     
 \draw[->][line width=1pt] (8.7,5) -- (9.54,5)
node[pos=0.5,above] {$1$};     
 \draw[->][line width=1pt] (11.45,5) -- (12.25,5)
node[pos=0.5,above] {$1$};     
 \draw[->][line width=1pt] (13.95,5) -- (14.8,5)
node[pos=0.5,above] {$1$};

\draw (-1.5,-1.5)node{$K$};
\draw (-1.5,1.5)node{$K$};
\draw (0,0)node{$K^{2}$};
\draw (1.5,0)node{$K^{2}$};
\draw (2.05,0)node{$\cdots$};
\draw (2.65,0)node{$K^{2}_{p}$};
\draw (4.25,0)node{$K^{3}_{i}$};
\draw (5.6,0)node{$K^{3}$};
\draw (6.2,0)node{$\cdots$};
\draw (6.7,0)node{$K^{3}_{s}$};
\draw (8.45,0)node{$K^{2}$};
\draw (9.85,0)node{$K^{2}$};
\draw (10.5,0)node{$\cdots$};
\draw (11.15,0)node{$K^{2}_{r+1}$};
\draw (12.65,0)node{$K_{r}$};
\draw (13.25,0)node{$\cdots$};
\draw (13.65,0)node{$K_{2}$};
\draw (15,0)node{$K_{1}$};
\draw[->][line width=1pt] (-1.25,-1.25) -- (-0.4,-0.25)
node[pos=0.5,below] {{\tiny$\begin{bmatrix}
1\\
1
\end{bmatrix}$}};
\draw[->][line width=1pt] (-1.4,1.25) -- (-0.4,0.25)
node[pos=0.5,above] {{\tiny$\begin{bmatrix}
1\\
0
\end{bmatrix}$}};
\draw[->][line width=1pt] (0.3,0) -- (1.25,0)
node[pos=0.5,above] {$I_{2}$}; 
 \draw[->][line width=1pt] (3,0) -- (4,0)
node[pos=0.5,above] {{\tiny$\begin{bmatrix}
0&0\\
1&0\\
0&1
\end{bmatrix}$}};
 \draw[->][line width=1pt] (4.5,0) -- (5.35,0)
node[pos=0.5,above] {$I_{3}$}; 
 \draw[->][line width=1pt] (7.05,0) -- (8.2,0)
node[pos=0.5,above] {{\tiny$\begin{bmatrix}
0&1&0\\
0&0&1
\end{bmatrix}$}};
 \draw[->][line width=1pt] (8.75,0) -- (9.54,0)
node[pos=0.5,above] {$I_{2}$};  
 \draw[->][line width=1pt] (11.6,0) -- (12.35,0)
node[pos=0.5,above] {{\tiny$\begin{bmatrix}
1&0\\

\end{bmatrix}$}};  
 \draw[->][line width=1pt] (13.95,0) -- (14.8,0)
node[pos=0.5,above] {$1$};

\draw (-1.5,-6.5)node{$K$};
\draw (-1.5,-3.5)node{$K$};
\draw (0,-5)node{$K^{2}$};
\draw (1.7,-5)node{$K^{2}$};
\draw (2.45,-5)node{$\cdots$};
\draw (3,-5)node{$K^{2}_{p}$};
\draw (4.3,-5)node{$K^{2}_{i}$};
\draw (5.8,-5)node{$K^{2}$};
\draw (6.5,-5)node{$\cdots$};
\draw (7.1,-5)node{$K^{2}_{s}$};
\draw (8.6,-5)node{$K$};
\draw (10.2,-5)node{$K$};
\draw (10.85,-5)node{$\cdots$};
\draw (11.25,-5)node{$0$};
\draw (12.65,-5)node{$0$};
\draw (13.25,-5)node{$\cdots$};
\draw (13.65,-5)node{$0$};
\draw (15,-5)node{$0$};

\draw[->][line width=1pt] (-1.25,-6.25) -- (-0.4,-5.25)
node[pos=0.5,below] {{\tiny$\begin{bmatrix}
1\\
1
\end{bmatrix}$}};
\draw[->][line width=1pt] (-1.4,-3.75) -- (-0.4,-4.75)
node[pos=0.5,above] {{\tiny$\begin{bmatrix}
1\\
0
\end{bmatrix}$}};
\draw[->][line width=1pt] (0.4,-5) -- (1.25,-5)
node[pos=0.5,above] {$I_{2}$}; 
 \draw[->][line width=1pt] (3.25,-5) -- (4,-5)
node[pos=0.5,above] {$I_{2}$};  
 \draw[->][line width=1pt] (4.6,-5) -- (5.45,-5)
node[pos=0.5,above] {$I_{2}$}; 
 \draw[->][line width=1pt] (7.5,-5) -- (8.35,-5)
node[pos=0.5,below] {{\tiny$\begin{bmatrix}
1&0
\end{bmatrix}$}};
 \draw[->][line width=1pt] (8.9,-5) -- (9.74,-5)
node[pos=0.5,above] {$1$};   
 \draw[->][line width=1pt] (11.45,-5) -- (12.25,-5)
node[pos=0.5,above] {$0$};     
 \draw[->][line width=1pt] (13.95,-5) -- (14.8,-5)
node[pos=0.5,above] {$0$};

\draw[->][line width=1pt] (6.7,4.5) -- (6.7,0.5)
node[pos=0.4,right] {\scriptsize{$f_{k}= \left\{ \begin{array}{lcc}
             1 &   if  & 1 \leq k \leq r  \\
             \\ \begin{bmatrix}
1\\
\frac{1}{2}
\end{bmatrix} &  if &  r+1 \leq k \leq s \\
             \\ \begin{bmatrix}
\frac{1}{2}\\
1\\
\frac{1}{2}
\end{bmatrix} &  if  & s+1 \leq k \leq i\\
	 &     &    \\
	0 &   if  &  i+1\leq k \leq n
             \end{array}
   \right.$}};

\draw[->>][line width=1pt] (6.7,-0.5) -- (6.7,-4.5)
node[pos=0.5,right] {\scriptsize{$g_{k}= \left\{ \begin{array}{lcc}
             0 &   if  & 1 \leq k \leq r  \\
             \\ \begin{bmatrix}
1&-2
\end{bmatrix} &  if &  r+1 \leq k \leq s \\
             \\ \begin{bmatrix}
0&1&-2\\
1&0&-1
\end{bmatrix} &  if  & s+1 \leq k \leq i\\
\\ \begin{bmatrix}
1&-2\\
0&-1
\end{bmatrix} &  if  & i+1 \leq k \leq n-2\\
	&    &      \\
	 1&  if    & k=n-1    \\
	-1&   if  &  k=n
             \end{array}
   \right.$}};
\end{tikzpicture}}}
\caption{Non-split extension Lemma \ref{new-sequences}}
\label{sequence}
\end{figure}
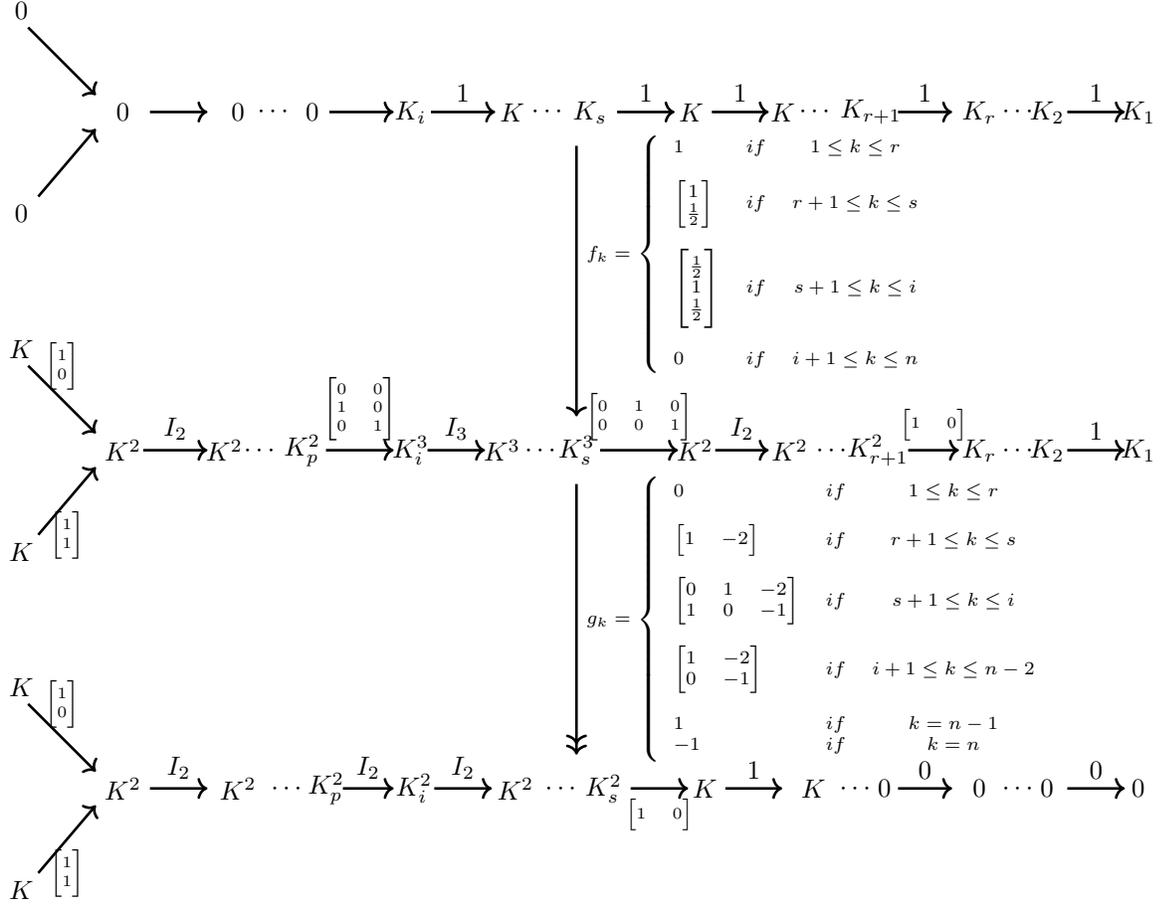
\end{proof}

In \cite{Br} the author finds a representative for all one dimensional $\Ext$ spaces and two linearly independent representatives when $\Ext_H(N,M)$ is of dimension two, plus another extension that is linearly dependent with them. 

While the middle term is already settled when the dimension of $\Ext_H(N,M)$ is one, we see that there are four possible middle terms for a non-split extension when $\Ext_H(N,M)$ is of dimension two. The extension not mentioned in \cite[Theorem 2.2 (b)]{Br}, appears in Figure \ref{sequence}, Lemma \ref{new-sequences}.

For the next lemma the following notions are important. Let $M,Y$ be indecomposable modules. It is known that $H$ is of finite representation type and the Auslander--Reiten quiver $\Gamma(H)$ is a standard component, so a morphism $M \to Y$ is a composition of paths in $\Gamma (H)$ modulo mesh relations.

\begin{lemma}\label{middle-terms} Let $M_i$ and $M_\beta$ be indecomposable $H$-modules such that $\dim_K \Ext_H(M_\beta, M_i) =2$, the middle term in a non-split sequence
\[ 0 \to M_i \to Y_j \to M_\beta \to 0 \] has to be one of the following, where the middle term summands appear in the Auslander--Reiten quiver in Figure \ref{m-t}.
\begin{enumerate}
\item[s1)] $Y_1 = A_1 \oplus B_1 \oplus X_3$ (or $Y_1 = A_1 \oplus B_2 \oplus X_3$)
\item[s2)] $Y_2 = A_2 \oplus B_2 \oplus X_3$ (or $Y_2 = A_2 \oplus B_1 \oplus X_3$)
\item[s3)] $Y_3 = X_1 \oplus X_2$
\item[s4)] $Y_4 = X \oplus X_3$
\end{enumerate}
\end{lemma}

\begin{remark} This result is also illustrated in an example for type $D_5$ in \cite{S14}, Section 3.3.4.2.
\end{remark}

\begin{proof}
Consider $\Gamma(H)$ the Auslander--Reiten quiver. Let $M_i$ be the projective $P(i)$, this choice makes the argument easier but $M_i$ can be any indecomposable module such that the two dimensional extension exists as we are using compositions of irreducible morphisms. 

The extensions (s1)-(s3) were mentioned in \cite{Br}, and (s4) appears in Lemma \ref{new-sequences}. We want to show that those four are the only possible non-split extensions. Let $M_\beta$ be an indecomposable such that $\dim_K \Ext_H(M_\beta, M_i) =2$, and let $0 \to M_i \to Y \to M_\beta \to 0 $ be a non-split extension. See the relative positions of $M_i$ and $M_\beta$ in Figure \ref{m-t} (a). Observe that depending on the positions of $M_i$ and $M_\beta$, the sectional paths $M_i \to X_3$ and $X_3 \to M_\beta$ may not exist, in that case set $X_3 \simeq 0$.

\begin{figure}[h!]
\centering
\def\svgwidth{5.5in}
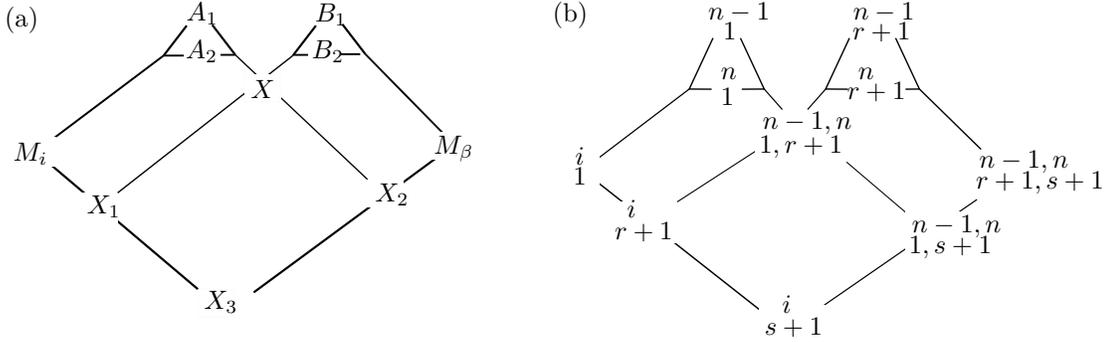
\caption{Possible middle terms for the extension. Lemma \ref{middle-terms}.}
\label{m-t}
\end{figure}

\noindent
In Figure \ref{m-t}
 (b) we display for each slot the indices corresponding to the projective cover and injective envelope. For instance the label $\begin{psmallmatrix} n-1 & n \\ r+1 & s+1\end{psmallmatrix}$ in the extreme right of Figure \ref{m-t} (b) indicates that $P(n-1) \oplus P(n)$ is the (minimal) projective cover for $M_{\beta}$ and $I(r+1) \oplus I(s+1)$ is the injective envelope. This information can be deduced from the position of $M_\beta$ in $\Gamma(H)$.
 
We denote the direct sum of indecomposable injective modules $I(a_1) \oplus \cdots \oplus I(a_m)$ by $I(a_1,\ldots,a_m)$, similarly denote $P(a_1, \ldots, a_m)$ the direct sum of indecomposable projective modules. Consider the projective cover and injective envelope of $M_i$ and of $M_\beta$. 

While $M_i$ is isomorphic to its projective cover, $P(n-1,n)$ is the projective cover for $M_\beta$.  By the properties of projective modules and using snake lemma, we can build a commutative diagram where the vertical lines are epimorphisms. 
\begin{center}
\begin{tikzcd}
0 \arrow{r} & P(i) \arrow[twoheadrightarrow]{d} \arrow{r} & P(i,n-1,n) \arrow[twoheadrightarrow]{d}\arrow{r} & P(n-1,n)\arrow[twoheadrightarrow]{d} \arrow{r}
& 0 \\
0 \arrow{r} & M_i  \arrow{r} & Y  \arrow{r} & M_\beta \arrow{r}  & 0  
\end{tikzcd}
\end{center}

Similarly, by properties of injective modules and using snake lemma, we can build a commutative diagram where the vertical lines are monomorphisms. 

\begin{center}
\begin{tikzcd}
0 \arrow{r} & M_i \arrow[hookrightarrow]{d}  \arrow{r} & Y \arrow[hookrightarrow]{d} \arrow{r} & M_\beta \arrow{r} \arrow[hookrightarrow]{d} & 0  \\
0 \arrow{r} & I(1) \arrow{r} & I(1,r+1,s+1) \arrow{r} & I(r+1,s+1) \arrow{r}
& 0
\end{tikzcd}
\end{center}

Hence, since there is an epimorphism from $P(i,n-1,n)$ to $Y$ and there is a monomorphism from $Y$ to $I(r+1,s+1)$, the middle term has  to satisfy:
\begin{enumerate}
\item[i.] $\soc Y $ is a direct summand of $S(1)\oplus S(r+1) \oplus S(s+1)$, and $\ttop Y$ is a direct summand of  $S(i) \oplus S(n-1) \oplus S(n)$, 

\item[ii.] $\Hom (M_i,Y)\neq 0$ and $\Hom (Y,M_\beta)\neq 0$, and
\item[iii.] $\udim Y = \udim M_i + \udim M_\beta $.

\end{enumerate} 
If we write the semisimple $\ttop P(M) $ and $\soc I(M)$ for each indecomposable $M$ in $\Gamma (H)$, we obtain a pattern depicted partially in Figure \ref{m-t} (b) and we see that the $\ttop$ and $\soc$ conditions (i) are satisfied only by the indecomposables appearing in Figure \ref{m-t} (a). We have $\udim Y$ and $\udim M_i + \udim M_\beta$. It is easy to see that the sequences (s1)-(s4) are the only ones whose middle terms satisfy this Thus, these are the only non-split short exact sequences having extremes $M_i, M_{\beta}$.\end{proof}

\subsection{Cluster category}\label{sec:cluster triangles}

The cluster category of type $D_n$ was defined as  $ \mathcal{C} =\Dd^b ( \modu k Q) / \tau^{-1} [1]$ where $Q$ is an oriented quiver of type $D_n$ in \cite{BMRRT}. The cluster category is a particular case of a cluster category from a quiver with potential \cite{DWZ}, studied by several authors and introduced in \cite{Am}. The hereditary algebra $H$ in the previous section is a particular case of cluster-tilted algebra as in \cite{bmr}, and it is also a Jacobian algebra arising from a surface triangulation \cite{L}.

\smallskip

In \cite[Lemma 3.2]{P} it is proved that every extension in $\mmod H$ can be lifted to a triangle in $\Cc$. Applying the Auslander--Reiten translation, denoted by $\tau$, to all the triangles induced by the non-split extensions (always considered up to isomorphism) in Section \ref{hereditary}, we will obtain all the possible middle terms for triangles with indecomposable extremes in $\Cc$.

\begin{figure}[h!]
\centering
\def\svgwidth{5.6in}
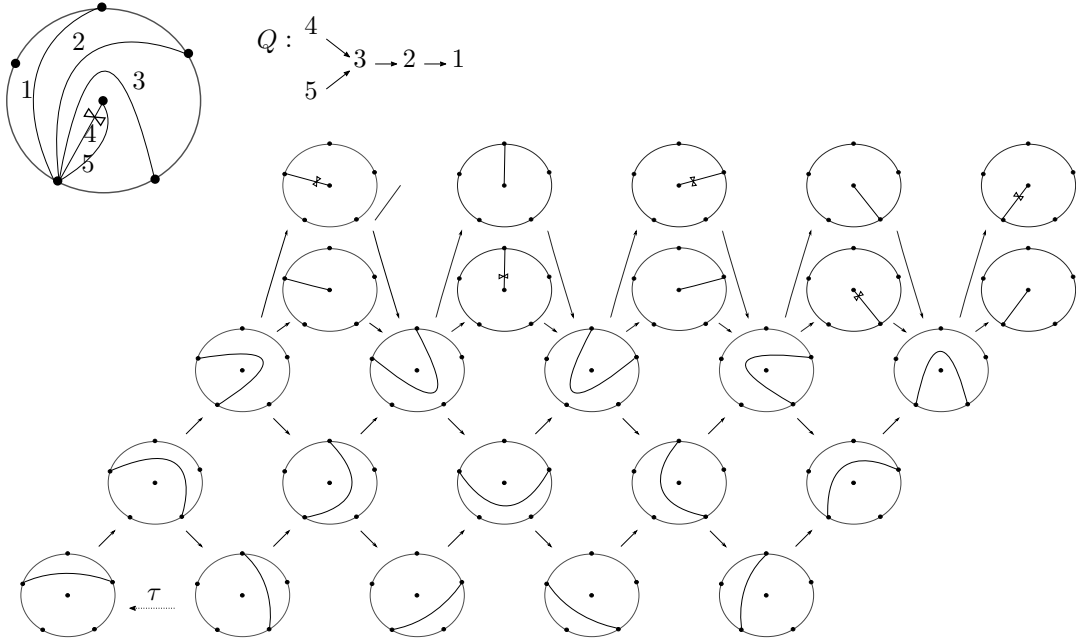
\caption{Category of tagged arcs on the punctured disk. For each arc $\gamma$, $\tau(\gamma)$ is the arc immediate on the left.}
\label{categoria}
\end{figure}

\begin{lemma}\label{lema-pla} Let $\Cc$ be a cluster category of Dynkin type. Every non-split triangle with indecomposable extremes is obtained lifting a non-split extension of the form $0\to P_i \to Y \to M \to 0$ on $\mmod H$ for a certain hereditary cluster-tilted algebra $H$.
\end{lemma}
\begin{proof}
Let $\epsilon \colon A \to B \to C \to A[1]$ be a non split triangle,
where $A$ and $C$ are indecomposable objects in $\Cc$. The category is standard so all morphisms are determined by the Auslander--Reiten quiver and the mesh relations. There is a cluster tilting object $T$ defined as a slice containing $A$ and with the shape of the opposite quiver $Q^{op}$ that gives an equivalence of categories \cite[Theorem 2.2]{bmr}: \[\Hom(T, -)\colon \Cc / (T[1]) \to \mmod H,\] 
where $H = KQ$. Then $\Hom_{\Cc}(T, \epsilon)$ is the required non-split extension.
\end{proof}

The cluster category $\Cc$ of type $D$ was realized geometrically in \cite[Theorem 4.3]{S}. The author establishes a bijection between indecomposable objects in $\Cc$ and tagged arcs (the tag being notched or plain) on the punctured disk with $n$ marked points on the boundary, and the irreducible morphisms are expressed as geometric moves and change of tags. See Figure \ref{categoria}. In this category, the crossing number (see Definition \ref{def:crossing}) between two tagged arcs $\alpha$ and $\beta$ is 0,1 or 2  \cite[Section 3.2]{S}. The crossing number equals the dimension of $\Ext_{\Cc} (X_\alpha ,X_\beta )$ as a $K$-vector space.

\section{Cluster algebras from surfaces and skein relations}\label{sec:skein-rel}

Here we summarize some definitions and results on cluster algebras arising from surfaces. In particular, we want to recover some equations that can be shown in pictures over the punctured disk and that will relate to non-split extensions.

\smallskip

Let $S$ be a connected oriented 2-dimensional Riemann surface with non empty boundary $\partial S$. Fix a non-empty finite set $M$ of marked points. There is at least one marked point on each connected boundary component. The marked points in the interior of $S$ are called punctures. The pair $(S,M)$ is called a \emph{bordered marked surface} or, in a shortened way, a \emph{surface}. For more details, see \cite{FST}.

Consider the next definition for plain (i.e. non-tagged) arcs and generalized arcs.

\begin{defi} 

\begin{enumerate}
\item A \emph{generalized arc} $\alpha$ is a curve in $S$ (up to isotopy) such that: its endpoints are in $M$, except for its endpoints $\alpha$ is disjoint from $\partial S$, $\alpha$ does not cut out an unpunctured monogon or an unpunctured bigon.

\item A generalized arc is an \emph{arc} if moreover it does not cross itself except when its endpoints coincide. 

\item A \emph{closed loop} is a closed curve $\zeta$ which is disjoint from $M$ and $\partial S$.

\item For any two arcs $\alpha$ and $\beta$, the \emph{plain crossing number} $e^\bullet(\alpha, \beta)$ is the minimal number of crossings between curves $\tilde{\alpha}$ and $\tilde{\beta}$ in their respective isotopy classes.

\item  Two arcs $\alpha$ and $\beta$ are \emph{compatible} if $e^\bullet(\alpha,\beta)=0$. 
\end{enumerate}  
\end{defi}

We will use later the plain crossing number to define the crossing number in the tagged setting.

Generalized arcs and loops are allowed to self-cross a finite number of times. An \emph{ideal triangulation} is a maximal collection $\Tt$ of pairwise compatible arcs. It is possible for an ideal triangle to have only two distinct edges, then it is called \emph{self-folded} triangle. An example is shown in Figure \ref{loop}.

\begin{figure}[h!]
\centering
\def\svgwidth{3.7in}
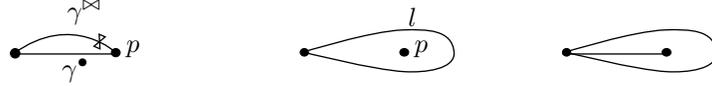
\label{triangulacion}
\caption{Tagged arcs at a puncture and noose. Self-folded triangle (right).}
\label{loop}
\end{figure}

In \cite[Section 4]{FST} each ideal triangulation $\Tt$ is associated to a skew-symmetric matrix $B_{\Tt}$ in a puzzle-like game. This matrix can be used as the only piece of information needed to define a cluster algebra.
The \emph{cluster algebra associated to the  surface} $(S,M)$ is the cluster algebra $\cA (S,M)$ defined from the seed $(\x,B_{\Tt})$, where the initial variables $x_1, \ldots, x_n$ are in  correspondence with the arcs in an ideal triangulation $\Tt$ of $(S,M)$. We will always consider the cluster algebra with \emph{ trivial coefficients}. We will not give a detailed definition of cluster algebra, we refer the reader to the very nice introduction given in \cite[Section 2]{MSW}.

\smallskip

Ideal triangulations are connected by flips. A flip is a local exchange that happens on a square within the triangulation. All four sides are either arcs in the triangulation or boundaries, and one of the diagonals belongs to $\Tt$. A flip exchanges this diagonal for the other one.

When $\Tt$ has self-folded triangles the internal edge of that triangle cannot be flipped. To solve this problem Fomin--Shapiro--Thurston \cite[Section 7]{FST} defined \emph{tagged arcs}. When one of the endpoints of the arc is a puncture then we have the option of tagging near this endpoint either \emph{notched} $\bowtie$ or \emph{plain} $\bullet$ (in the plain case we may not put a mark in our figures). The tagged arc denoted by $\gamma^{\Join}$ in Figure \ref{loop} replaces the arc $l$ enveloping the puncture in the model. An arc $\alpha$ not at the puncture is equal to its plain version. We call the arc $l$ a \emph{noose} or a loop.

\smallskip

When the surface is the disk with one puncture there can only be notched tagged arcs at the puncture. Denote by $\gamma^{\bowtie}$ the notched tagged arc at $p$ and by $\gamma^{\bullet}$ the plain version, see Figure \ref{loop}.

While one can define the crossing number for tagged arcs in general, we will give a definition only for the punctured disk. 

\begin{defi}\label{def:crossing} Let $\alpha$ and $\beta$ be two tagged arcs on the punctured disk. Then the \emph{crossing number} $e(\alpha, \beta)$ is computed as follows
\begin{enumerate}
\item if one of them is not an arc at the puncture then $e(\alpha, \beta) = e^\bullet(\alpha^\bullet,\beta^\bullet)$ 
\item if both are arcs at the puncture and they have the same tag then $e(\alpha, \beta) = 0$, if they have different tag then the notched arc, say $\alpha$, should be replaced with the corresponding noose $l$: $e(\alpha,\beta) = e^\bullet (l,\beta )$. 
\end{enumerate}
\end{defi}

A maximal set of non-crossing tagged arcs (that is two by two compatible: $e(\alpha, \beta) = 0$) is a \emph{triangulation}. Every ideal triangulation not containing self-folded triangles can be associated to a tagged triangulation in a trivial way, where all arcs at the puncture are tagged plain.

\subsection{Skein relations}

Given an ideal triangulation $\Tt$ without self-folded triangles, generalized arcs and loops define cluster algebra elements. We will recall some identities that \emph{hold for cluster algebras arising from surfaces possibly with punctures, in particular they hold for the punctured disk}, and follow from the interpretation of cluster variables as lambda lengths and the
Ptolemy relations for lambda lengths \cite{FT}. The identities that we will mention also arise from combinatorial formulas that appear in several articles as \cite{MW,MSW}.

\smallskip

Each finite set $C$ of generalized arcs or closed loops is called a {\it multicurve}, such set is associated to a cluster algebra element defined by the product of all elements $x_\alpha$ with $\alpha\in C$.  A multicurve $C$ is represented drawing each curve $\alpha \in C$ over $(S,M)$.

Let $\gamma^{\bowtie} $ and $\gamma^{\bullet}$ be as in Figure \ref{loop}, the elements in $\cA (S,M)$ satisfy $x_{\gamma^{\bowtie}} x_{\gamma^{\bullet}} = x_{l}$. This definition is compatible with Ptolemy relations for lambda lengths and exchange relations of cluster variables, see \cite[Lemma 8.2]{FT}, and agrees with the combinatorial definition in \cite[Section 4]{MSW}. When $\alpha$ is a tagged arc, $x_\alpha$ is a cluster variable. If $\epsilon$ is a curve isotopic to a boundary segment, we set  $x_{\epsilon} = 1$.


Let $\alpha$, $\beta$ and $\delta$ be generalized arcs, tagged plain (this includes the noose), or closed loops. Let $\alpha$ be a generalized arc crossing $\beta$ at a point $x$, and let $\delta$ be such that it has a self-intersection at a point $x$.

\begin{defi}\label{def:smoothing} 
For the crossings as those in Figure \ref{skeinrel} (1) and (2),

\begin{figure}[h!]
\centering
\def\svgwidth{5in}
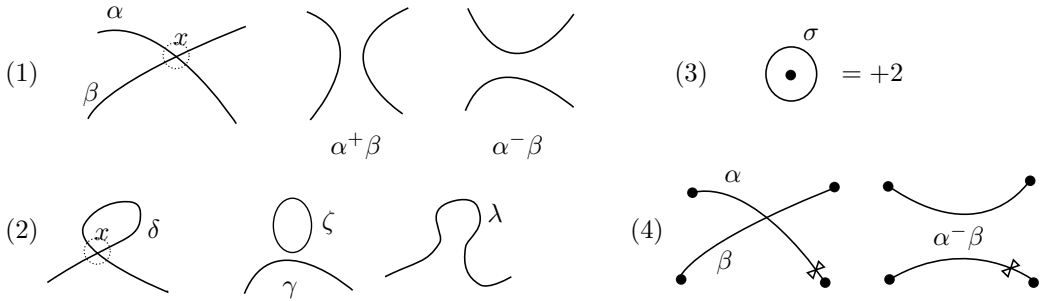
\caption{Smoothing arcs.}
\label{skeinrel}
\end{figure}

\begin{enumerate}

\item the smoothing of $\{ \alpha, \beta\}$ is the pair of multicurves $\alpha^+ \beta$, obtained walking along $\alpha$, stopping just before the point $x$ and turning right following $\beta$, and $\alpha^- \beta$, obtained in a similar way but turning left at $x$.

\item  the smoothing of the self intersecting curve $\{ \delta \}$ is the pair of multicurves $\{ \zeta, \gamma \}$,  obtained walking along $\delta$ and avoiding the point $x$ by turning right, and $\{ \lambda \}$, obtained walking along $\delta$ and avoiding the point $x$ by turning left.
\end{enumerate}  

\end{defi}

We have identities for the elements represented by the multicurves in the cluster algebra $\cA (S,M)$:
\[ (1) \ \mathrm{and} \ (4) \hspace{5pt} x_\alpha  x_\beta = x_{\alpha^+ \beta} + x_{\alpha^- \beta} \hspace{15pt} ; \hspace{15pt} (2) \hspace{5pt}
x_\delta = x_{\gamma, \zeta} + x_\lambda.\]

A closed loop $\sigma $ that is contractible to a puncture gives us the element $\mathrm{(3) \ \ x_{\sigma} = + 2}$\footnote{This is since we evaluate all $y$ variables to 1, see \cite[Section 8.3]{MSW}}.

If $\alpha$ and $\beta$ are tagged arcs crossing at a point $x$ and one of them, say $\alpha$, has a puncture as an endpoint and it is tagged notched there, we can proceed as in Figure \ref{skeinrel} (1), preserving the notched end in the multicurves $\alpha^+ \beta$ and $\alpha^- \beta$, as the reader can see in the example in Figure \ref{skeinrel} (4). The equation $(1)$ still holds and can be obtained from the plain version applying the automorphism \emph{change of tag} from \cite[Section 4.4]{ASSh}.

\begin{prop}[Punctured skein relations]\label{Propo1}
Let $\mathcal{A}(S,M)$ be the cluster algebra arising from the punctured disk. Let the arcs in the next figures represent the corresponding elements in $\cA (S,M)$.
The following identities take place in the cluster algebra.

\begin{figure}[h!]
\centering
\def\svgwidth{6in}
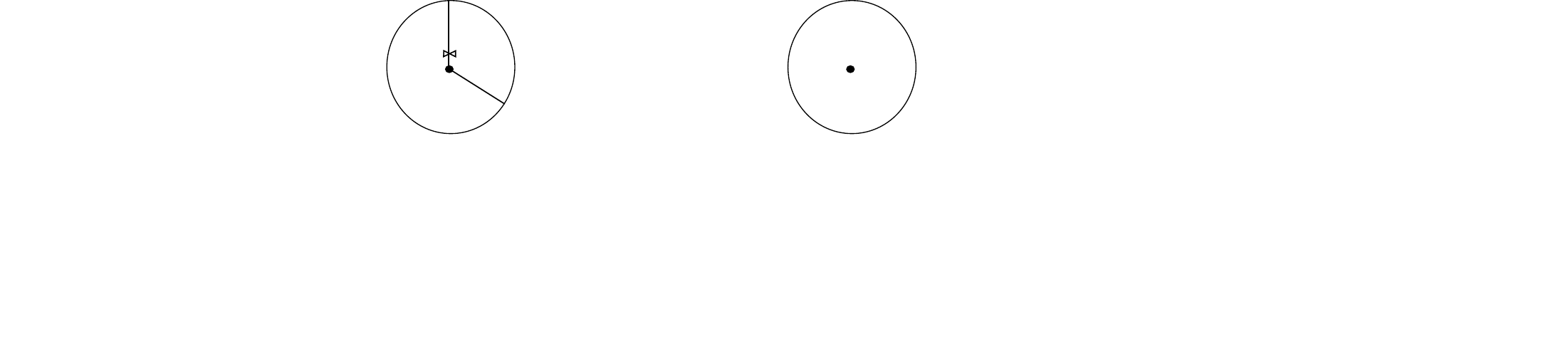
\label{prop1}
\end{figure}

\end{prop}

\begin{proof}
Part (a) is easily obtained after multiplying the expression by $x_{\gamma^\bullet}$, where $\gamma^{\Join} $ is the notched arc in the figure. Then the skein relation $(1)$ can be applied and $x_{\gamma^\bullet}$ can be extracted as a common factor. This gives formula (a). Observe that we can consider that the puncture is the intersection point $x$ and the intersection is determined by the curves having different tag. The smoothing is obtained avoiding the puncture and by turning right and left. In this case we will use the notation $\alpha^+ \beta , \alpha^- \beta$.

We will obtain the first half of (b), as the rest is symmetric. First observe the following identities in Figure \ref{new}, where we use (a), and (2) and (3) from Figure \ref{skeinrel}.

\begin{figure}[h!]
\centering
\def\svgwidth{4in}
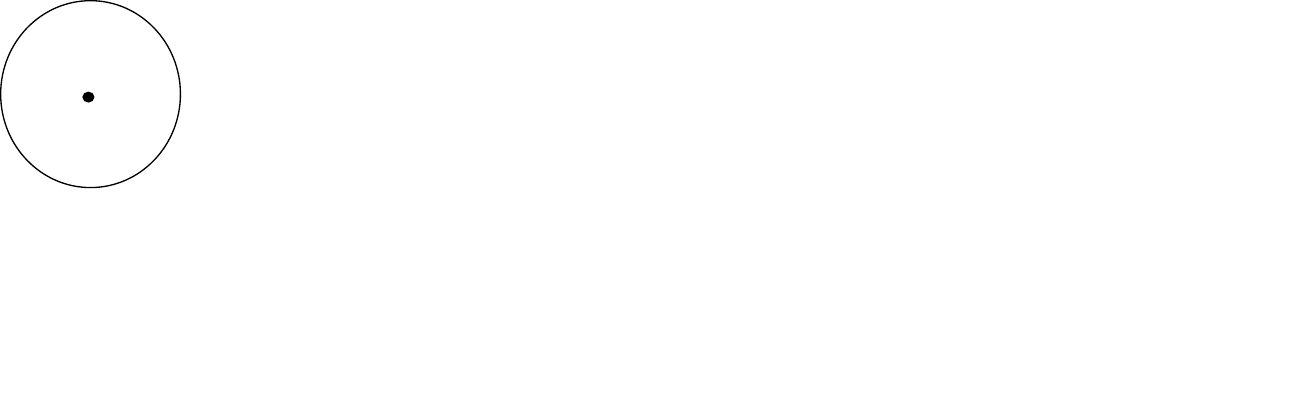
\caption{Proof of Proposition \ref{Propo1}.}
\label{new}
\end{figure}

The formulas in Figure \ref{new} imply the identity in Figure \ref{New2}. Now we can apply skein relations for a double intersection. We have the identity in Figure \ref{new4}, and a symmetrical formula. Thus we obtain part (b).\end{proof}

\begin{figure}[h!]
\centering
\def\svgwidth{3in}
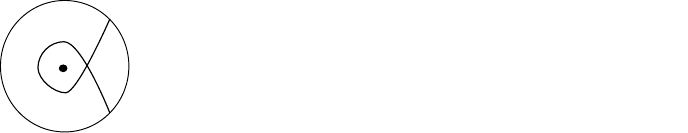
\caption{Proof of Proposition \ref{Propo1}.}
\label{New2}
\end{figure}

\begin{figure}[h!]
\centering
\def\svgwidth{3,6in}
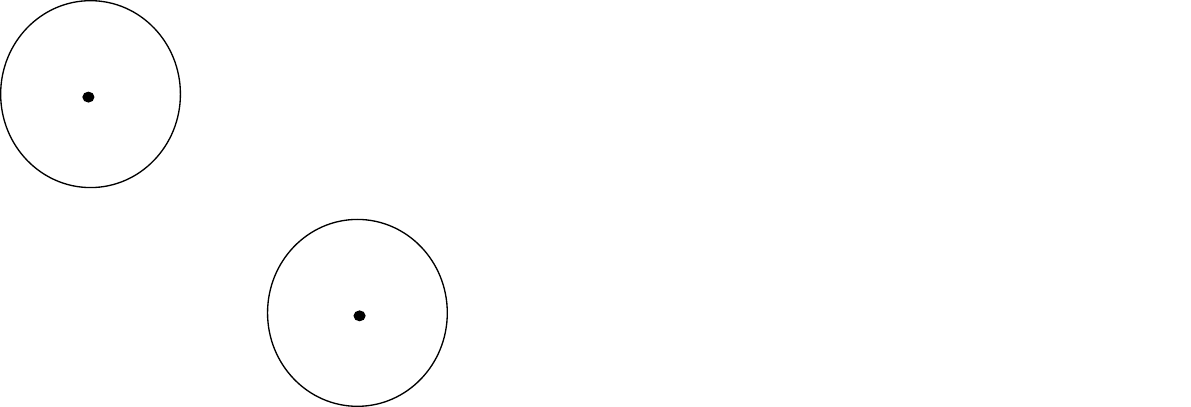
\caption{Proof of Proposition \ref{Propo1}.}\label{new4}
\end{figure}

In the next section we will interpret Proposition \ref{prop1} (b) as a generalization of
\[ (1) \ \ \   x_\alpha  x_\beta = x_{\alpha^+ \beta} + x_{\alpha^- \beta}\] 
which is the \emph{cluster multiplication formula} (see \cite{CK,Pa}) for string objects $X_\alpha$ and $X_\beta$ such that $\dim_K (X_\alpha, X_\beta) = 1$. The reader can find the definition of string objects and cluster categories arising from unpunctured surfaces in \cite{BZ}. The geometric relation between the one dimensional extension space and the equation (1) was proven in \cite{CSh}. This work shows how the middle terms arise from (unpunctured) skein relations.

\begin{remark}\label{Obs: otros resultados} We are going to show how non-split triangles with indecomposable extremes arise from punctured skein relations. The path we follow is the inverse to the technique in \cite{CSh}. In that article, the authors start with a pair of (generalized) arcs $\alpha$ and $\beta$ that cross in the interior of an unpunctured surface. Then, depending on three possibilities for the relative position of this pair of arcs, the authors define a triangulation, hence an associated gentle Jacobian algebra $J_{\alpha,\beta}$, such that there is a non-split extension in $\mathrm{Ext}_{J_{\alpha,\beta}} (M_\alpha, M_\beta)$. From these extensions, they can deduce the existence of triangles in the generalized cluster category. 

In our case, we find all the non-split triangles with indecomposable extremes in the cluster category $\mathcal{C}$ since we can take a single 'canonical triangulation' that defines the hereditary algebra $H$ and we can use Lemma \ref{lema-pla}. Then we will deduce the existence of each non-split extension of $J$-modules for any cluster tilted algebra $J$ of type $D$ from these triangles, for this we will use Lemma \ref{lemma-existe-sec}. 
\end{remark}


\section{Extensions from punctured skein relations}\label{sec:extensions}

We know from Section \ref{sec:cluster triangles} that non-split triangles in the cluster category $\Cc$ of type $D$ can be found using non-split extensions in $\mmod H$. We will use the non-split extensions mentioned in Section \ref{hereditary} to determine all possible short exact sequences over a Jacobian algebra arising from the punctured disk. 

\smallskip

Take the algebra $KQ=H$ defined by the triangulation in Figure \ref{Big} (that is again the one defined by the oriented quiver in Section \ref{hereditary}). Said algebra is defined as $\End^{op}_{\mathcal{C}}(T)$, where we denote by $T_i[1] $ the shift of each cluster tilting object summand and use the short notation $i$ for the arc in $ \mathcal{T}$. 

Recall that the Auslander-Reiten functor $\tau$ acts on arcs by rotating the endpoints counterclockwise and changing tag when the arc is at the puncture while the inverse $\tau^{-1}$ rotates endpoints in the opposite direction, see Figure \ref{categoria} in Section \ref{sec:cluster triangles}. Thus the projective module $P(i)$ over $\mathrm{mod} \, H$ corresponds to the arc $\tau^{-1}(i)$ in the geometric realization. 

\smallskip 

From now on it is important to remember that each tagged arc (not in $\mathcal{T}$) corresponds to an indecomposable module in $\mathrm{mod} \, H$ as there is a correspondence between arcs not in $\mathcal{T}$ and modules, see \cite[Section 6.4]{S}.

\begin{figure}[h!]
\centering
\scalebox{0.77}{\def\svgwidth{2,1in}
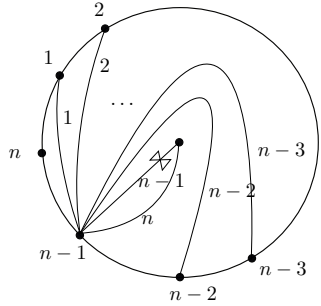}
\caption{Triangulation $\mathcal{T}$ that defines the Jacobian algebra $H=KQ$.}
\label{Big}
\end{figure}

Until the end of this section we denote an arc of the form $\tau^{-1}(i)$ by $\alpha$, and the module $P(i)$ by $M_{\alpha}$. 


Denote each boundary marked point $a_1, \ldots, a_n$ increasing clockwise as in Figure \ref{Big}, with this notation $a_n$ is one of the endpoints of $\alpha$. 

\smallskip

For a given $\alpha$, let $\beta \notin \mathcal{T}$ be an arc such that $e(\alpha, \beta) > 0$, then denote by $M_{\beta}$ the corresponding $H$-module. 

 The triangulation $\mathcal{T}$ (i.e. the position of the chosen cluster tilting object that defines an hereditary algebra) is such that for $\alpha$ and $\beta$ we have $\dim_K \Ext^1 (M_{\beta},M_{\alpha}) = e(\alpha, \beta) $ since the morphisms involved in the sequence do not factor through the ideal of morphisms $ ( \add (\oplus_{i=1}^n T_i [1]) ) = (T[1])$ in $\Cc$.

\subsection{Cases where $e(\alpha, \beta) =1$}

In this section we obtain each one of the non-split extensions in which $ \dim_K \Ext^1(M_{\beta},M_{\alpha}) = 1$. They are exactly the non-split extensions described in \cite[Theorem 2.2, 1-5]{Br} and the reader can compare the Figure \ref{categoria} and the Auslander--Reiten quiver in the after mentioned article. Now we can interpret each middle term as a multicurve $\alpha^+ \beta$ obtained by smoothing the crossing between $\alpha$ and $\beta$. 

\smallskip

Note that a multicurve defines a module as follows: if $\alpha^+ \beta = \{ \gamma_1, \ldots, \gamma_m \}$, then $M_{\alpha^+ \beta} = \oplus_{j=1}^m M_{\gamma_j}$. When $\gamma_j$ is isotopic to a boundary segment $M_{\gamma_j} = 0$.

\begin{figure}[h!]
\centering
\scalebox{0.88}{\def\svgwidth{6,6in}
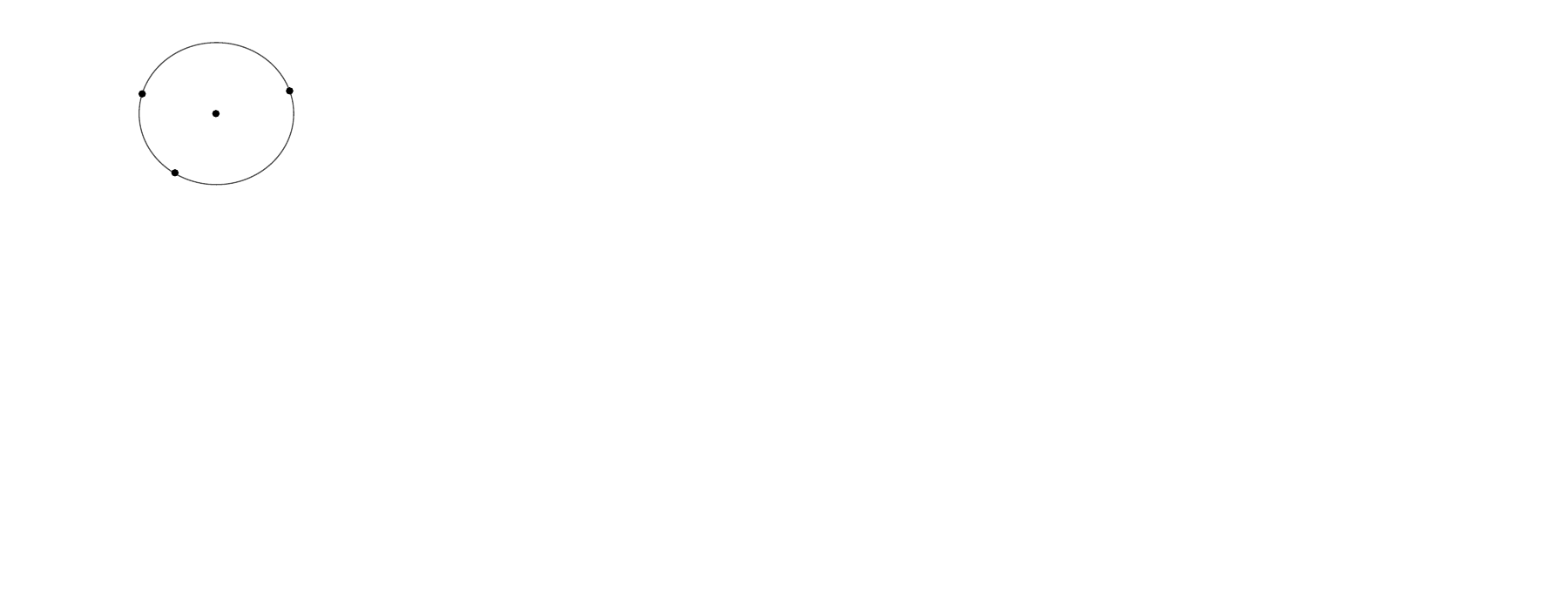}
\caption{Non-split extensions in $ \Ext^1(M_{\beta},M_{\alpha})$. The middle term corresponds to the module associated to the multicurve $\alpha^+ \beta$.}
\label{sec1}
\end{figure}

Observe the middle terms in the non-split sequences (1), (2) and (3) in Figure \ref{sec1}. For the case (1) we need the identity in the cluster algebra $x_l = x_{\gamma^{\bowtie}} x_{\gamma^\bullet}$. In case (2) the intersection is at the puncture, so the smoothing $\alpha^+ \beta$ is obtained as in Proposition \ref{Propo1} (a). 

See the middle terms in the non-split sequences (4) and (5) in Figure \ref{sec2}.
In the case of sequence (5), the smoothing carries a tag as explained in Figure \ref{skeinrel} (4).

\begin{figure}[h!]
\centering
\scalebox{0.88}{\def\svgwidth{5,7in}
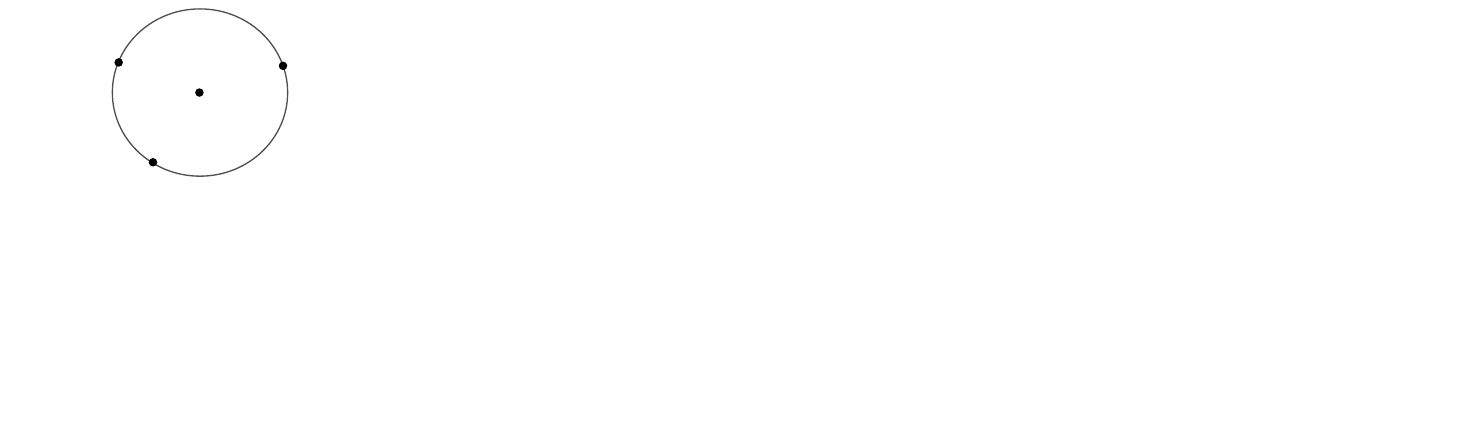}
\caption{Non-split extensions in $ \Ext^1(M_{\beta},M_{\alpha})$. The middle term is obtained as $\alpha^+ \beta$.}
\label{sec2}
\end{figure}

\begin{remark}
 Note that when $e(\alpha,\beta) = 1$, then $\alpha^+ \beta$ is a single multicurve. When $e(\alpha, \beta) = 2$ the set $\alpha^+ \beta $ will contain four multicurves an they appear in Figure \ref{localsmoo}. 
\end{remark}

\subsection{Cases where $e(\alpha, \beta) =2$}

We obtain now the non-split extensions appearing in Lemma \ref{middle-terms}. The reader can see the summands of $x_{\alpha^+ \beta}$ in Figure \ref{localsmoo}. They correspond to the four rightmost terms enclosed by a rectangle in Proposition \ref{Propo1} (b).
\[ 2 x_{\alpha} x_{\beta} = (x_{\alpha^- \beta}) + (x_{\alpha^+ \beta}) = (x_{\alpha^- \beta,x} + x_{\alpha^- \beta,y}) + (x_{\alpha^+ \beta,x} + x_{\alpha^+ \beta,y}),\]
When $\alpha$ and $\beta$ cross twice, we denote the local smoothing of $\alpha$ and $\beta$ at the interior point $z$ by $\alpha^* \beta, z$. This multicurve is obtained as in Definition \ref{def:smoothing}.

\begin{figure}[h]
\centering
\scalebox{0.9}{\def\svgwidth{5.2in}
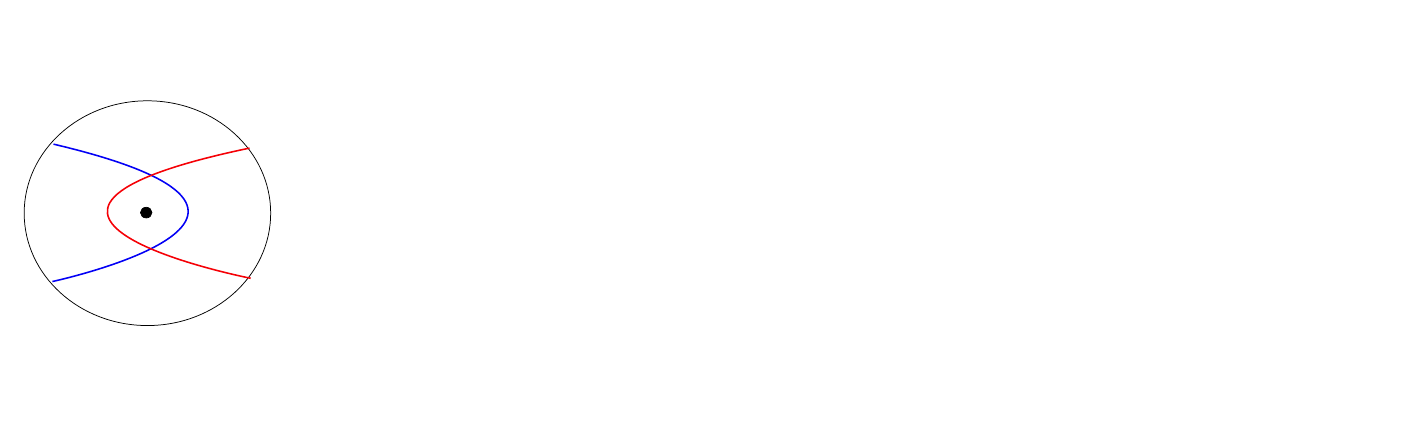}
\caption{The smoothing $\alpha^+ \beta$ at points $x$ and $y$ in the interior of the disk. The element corresponding to $\alpha^+ \beta,y$ is replaced by the sum in the rectangle since the identity holds in the cluster algebra.}
\label{localsmoo}
\end{figure}

\begin{figure}[h!]
\centering
\scalebox{0.9}{\def\svgwidth{3,8in}
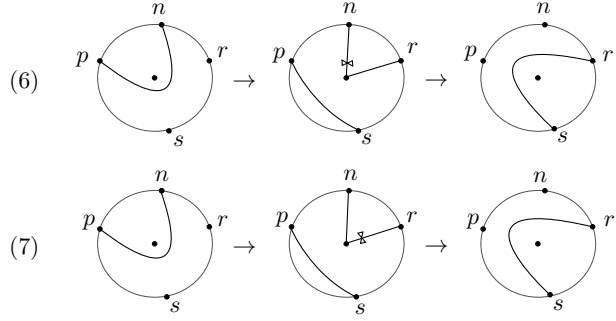}
\caption{Two non-split extensions that define a basis of $ \Ext^1(M_\beta,M_\alpha)$.}
\label{sec3}
\end{figure}

\begin{figure}[h]
\centering
\scalebox{0.9}{\def\svgwidth{3,6in}
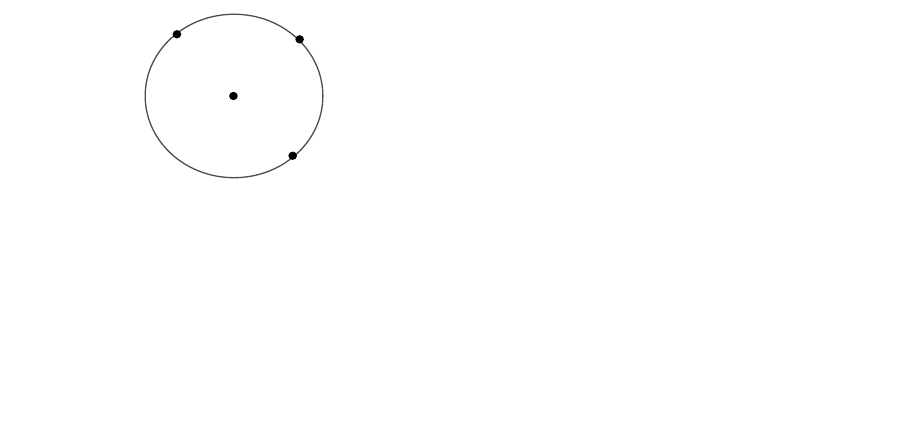}
\caption{The remaining non-split extensions in $ \Ext^1(M_\beta,M_\alpha)$.}
\label{sec4}
\end{figure}


\section{Non-split extensions of type $D$}\label{resultados-principales}

During this section we will need the cluster tilting equivalence. A cluster tilting object $T$ is a maximal basic object such that $\Hom(T,T[1]) = \Ext (T,T)=0$. In the geometric models, said object $T$ is represented by a triangulation $\mathcal{T}$, as $e(\alpha, \beta)$ represents the dimension of the corresponding $\Ext(X_\alpha, X_\beta)$ as $K$ vector space. Given a cluster tilting object $T$, the $\Hom$ functor produces a celebrated equivalence \cite[Theorem 2.2]{bmr}:
\[\Hom(T, -)\colon \Cc / (T[1]) \to \mmod B\] 
for $B= \End_{\mathcal{C}} (T)^{op}$. We have the next result.

\begin{theorem}\label{theorem-all-triangles}
All non-trivial triangles with indecomposable extreme terms in the cluster category of type $D$ can be obtained via (punctured) skein relations.
The triangles have the form
\[ X_\alpha \to X_C \to X_\beta \to X_{\alpha} [1], \]
where $C$ is a multicurve in $\alpha^+ \beta$.
\end{theorem}

\begin{proof}
From Lemma \ref{lema-pla} the extensions (1)-(9) lift to triangles in $\Cc$.  When we consider the category $\Cc$ the marked point $n$ on the boundary is not distinguished anymore  and the geometric figures (1)-(9) represent triangles independently of the arcs endpoints. The arc crossings in (1)-(9) are all the possible crossings of two arcs over the punctured disk, hence we find all the possible non-split extensions with indecomposable extremes.
\end{proof}

\begin{remark} We found the possible middle terms of \[ X_\alpha \to X_C \to X_\beta \to X_{\alpha} [1] \] performing a smoothing $\alpha^+ \beta$. On the other hand, we can do $\alpha^- \beta = \beta^+ \alpha$, in that case we obtain the middle terms for  \[ X_\beta \to X_E \to X_\alpha \to X_{\beta} [1]. \]
\end{remark}

Now that we have identified all the triangles in $\mathcal{C}$ we ask, given a cluster tilting object $T$, when a triangle gives rise to a short exact sequence in $\mmod B$. We will use the next lemma.

\begin{lemma}\label{lemma-existe-sec} Let $T$ be a cluster tilting object and $B = \End (T)^{op}$. 
Let $A \to E \to C \to A [1] $ be a triangle in a cluster category such that the identity 
\begin{equation*}
\dim_K \Hom(T,A) + \dim_K \Hom(T,C)  = \dim_K \Hom(T,E),
\end{equation*}
holds. Then 
\begin{equation*}
0 \to  \Hom(T,A) \to  \Hom(T,E)  \to  \Hom(T,C) \to 0
\end{equation*}
is a short exact sequence in $\mmod B$.
\end{lemma}
\begin{proof}
We start with the triangle $A \ra E \ra C \ra A[1]$ and apply $\Hom (T,-)$. Since it is a cohomological functor, this yields a long exact sequence in $\mmod B$:
\[
\cdots \ra \Hom(T,C[-1]) \ra \Hom(T,A) \xrightarrow{f}  \Hom(T,E)  \xrightarrow{g}  \Hom(T,C) \ra \Hom(T,A[1]) \ra \cdots
\]

Now we can truncate this long exact sequence and obtain the (shorter) exact sequence
\[
0 \ra ker f \ra \Hom(T,A) \xrightarrow{f}  \Hom(T,E)  \xrightarrow{g}  \Hom(T,C) \ra coker g \ra 0
\]

Basic linear algebra shows that
\[
\dim_K ker f - \dim_K \Hom(T,A) + \dim_K  \Hom(T,E)  - \dim_K \Hom(T,C) + \dim coker g = 0.
\]

From the hypothesis, we get $\dim_K ker f + \dim_K coker g = 0$. Hence $\dim_K ker f = \dim_K coker g = 0$ and the sequence
\[
0 \to  \Hom(T,A) \ra  \Hom(T,E)  \ra  \Hom(T,C) \to 0
\]
is exact.\end{proof}

\begin{remark}\label{rem-vectordim} Let $M_\gamma$ be a $B$ module for $B=KQ/I = \End_\Cc^{op}(T)$ a cluster tilted algebra of type $D$ (or more generally for a $2$-Calabi--Yau tilted algebra from a cluster category arising from a surface). Then $M_\gamma$ is such that as a $(Q,I) $ representation $\dim_K (M_\gamma)_i = e(i, \gamma)$, where we denote by $i$ the arc corresponding to $T_i [1]$ and also we may denote by $i$ the associated vertex of the quiver $Q$. This is because 
\[ (M_\gamma)_i = \Hom_B(P_i,M_\gamma) \simeq \dfrac{\Hom_{\mathcal{C}}(T_i, X_\gamma)} {(T[1])} \simeq \Hom_{\mathcal{C}} (T_i, X_\gamma)\simeq \Hom_{\mathcal{C}} (T_i[1], X_\gamma[1])=\Ext_{\mathcal{C}} (T_i [1], X_\gamma)\]
where the third $K$-isomorphism holds as $\Hom_{\mathcal{C}}(T_i,T[1]) = 0$. 
\end{remark}

\begin{defi} Let $\mathcal{T}$ be a triangulation of the punctured disk.
Let $C = \{ \gamma_i , \ldots , \gamma_t \}$ be a multicurve. We call the {\it total dimension of $C$}, and denote by $d(C)$, the total number of crossings between the multicurve and $\mathcal{T}$.
\end{defi}

Let $T$ be a cluster tilting object in $\mathcal{C}$ and $\mathcal{T}$ the corresponding triangulation on the punctured disk. Let $\alpha$ and $\beta$ be two arcs such that $\alpha, \beta \notin \mathcal{T}$ and $e(\alpha, \beta) \geq 1$. For the multicurve $\{ \alpha, \beta \}$ we have $d(\{\alpha, \beta \})= \sum_j e(j, \alpha) +  e(j, \beta)$ where $j$ runs over all the arcs in $\mathcal{T}$. That is, from Remark \ref{rem-vectordim}, $d(\{\alpha, \beta \})$ is the dimension of $M_\alpha \oplus M_\beta$ as a $K$ vector space.



\begin{prop}\label{prop:preserva-dimension} Let $\alpha$ and $\beta$ be arcs in the disk such that $e(\alpha, \beta) \geq 1$. Let $ X_\alpha \to X_C \to X_\beta \to X_{\alpha} [1] $ and $ X_\beta \to X_E \to X_\alpha \to X_{\beta} [1] $ be non-split triangles with indecomposable extremes in the cluster category of type $D$, where $C \in \alpha^+ \beta$ and $E \in \alpha^- \beta$, and let $\mathcal{T}$ be a triangulation.
Then there is a non-split extension 
\[ 0 \to M_{\alpha} \to M_{C} \to M_{\beta} \to 0 \]
if and only if $d(C) = d(\{ \alpha, \beta\})$, and there is a non-split extension
\[ 0 \to M_{\beta} \to M_{E} \to M_{\alpha} \to 0 \]
if and only if $d(E) = d(\{ \alpha, \beta\})$\end{prop}
\begin{proof}
This is a direct consequence or Theorem \ref{theorem-all-triangles}, Lemma \ref{lemma-existe-sec} and Remark \ref{rem-vectordim}.
\end{proof}

Observe that this result looks similar to \cite[Theorem 3.7]{CSh}. If a pair of string modules $M_\alpha$ and $M_\beta$ in the category $\mmod B$, where $B$ is a Jacobian algebra arising from a surface without punctures, cross in a 3-cycle then $M_C$ does not preserve total dimension. This type of crossing is the only one, among the three presented by the authors, where the dimension is not preserved. 

However, the path to obtain our last result is different, see Remark \ref{Obs: otros resultados}.


\section{Appendix: Examples in type $D$ and other surfaces}\label{EXAMPLES}

In this section, we will compute some examples of extensions given by punctured skein relations as shown in Figures \ref{sec1}, \ref{sec2}, \ref{sec3} and \ref{sec4}. The examples are not exclusively in type $D$, as the results in the previous section apply to  geometric configurations that resemble the punctured disk. The notation and results used in this section follow \cite{D,DWZ,L}, but we briefly recall them now.

Let ($Q(\tau)$, $P(\tau)$) be a quiver with potential for a  triangulation $\tau$ of the surface and $P(\tau)=\displaystyle{ \sum_{\triangle}S^{\triangle}} +S^{p}$  the potential defined by Labardini-Fragoso \cite{L}, where $\triangle$ runs over the set of the interior triangles in $\tau$ oriented in clockwise and $S^{p}$ is the cycle around the punctured $p$ oriented in counterclockwise.

A representation of a quiver with potential ($Q$,$P$) is a pair $(M,P)$, where $P$ is a potential and $M$ consists of the two following families.

\begin{itemize}
	\item[1)] A family $(M_{i})_{i\in Q_{0}}$ of  finite dimensional $K$ vector spaces.
	
	\item[2)] A family $(\varphi_{\alpha}:M_{s(\alpha)}\longrightarrow {}{M_{t(\alpha)}})_{\alpha \in Q_1}$ of $K$-linear transformations such that  $\partial_{\alpha}(P)=0$,   for all $\alpha \in Q_{1} $. Where $\partial_{a}$ is the cyclic derivative defined for all arrows $a\in Q_{1}$ and each cycle $a_{1}...a_{d}$ in $Q$:
\[\partial_{a}(a_{1}...a_{d})=\displaystyle{\sum_{i=1}^{d}}\delta_{a,a_{i}}a_{i+1}...a_{d}a_{1}...a_{i-1}\]
where $\delta_{a,a_{i}}$ is the Kronecker delta and extending linearly and continuously to obtain a morphism $\partial_{a}:\langle \langle R \rangle \rangle_{cyc} \rightarrow \langle \langle R \rangle \rangle$, for more details see \cite[Definition 3.1]{DWZ}. 

\end{itemize}

For an arc $j$  which does not belong to $\tau$, let $\gamma_{q_{0},q_{1}}=[q_{0},q_{1}]_{j}$ be the segment of  $j$ walking from $q_{0}$ to $q_{1}$. The segment $\gamma_{q_{0},q_{1}}$ {\it surrounds the punture $p$ counterclockwise}  if locally $\gamma_{q_{0},q_{1}}$ is the segment shown in  Figure  \ref{desviacion}. A {\it detour curve} $d_{q_{0},q_{1}}^{\Delta }$ is drawn inside $\Delta$  if  $\gamma_{q_{0},q_{1}}=[q_{0},q_{1}]_{j}$ surrounds the puncture $p$, as is the Figure \ref{desviacion}.

\vspace{-3mm}
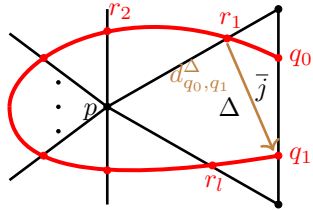
\begin{figure}[H]
\centering				

\subfigure{\begin{tikzpicture}[scale=0.65]

\filldraw [black] (-5,0)  circle (2pt)
			(-5,-4)  circle (2pt)
			(-8.5,-2)  circle (2pt)
			(-9.5,-2)  circle (1pt)
			(-9.5,-2.5)  circle (1pt)
			(-9.5,-1.5)  circle (1pt);
			
\draw  (-6,-2)node{$\Delta$};
 \draw[line width=1pt] (-5, 0) -- (-5, -4)
	node[pos=0.4,left] {$\overline{j}$};
\draw[line width=1pt] (-5, 0) -- (-8.5, -2);
\draw[line width=1pt] (-5, -4) -- (-8.5, -2);
\draw[line width=1pt] (-8.5, -2) -- (-8.5,0);
\draw[line width=1pt] (-8.5, -2) -- (-10.5,-0.5);
\draw[line width=1pt] (-8.5, -2) -- (-10.5,-3.5);
\draw[line width=1pt] (-8.5, -2) -- (-8.5,-4)
node[pos=0.05,left] {$p$};

\filldraw [red] (-5,-1)  circle (2pt)
		      (-5,-3)  circle (2pt)
		      (-6.35,-3.2)  circle (2pt)
		      (-8.5,-3.3)  circle (2pt)
		      (-9.8,-3)  circle (2pt)
		      (-9.8,-1)  circle (2pt)
		      (-8.5,-0.45)  circle (2pt)
		      (-6.05,-0.6)  circle (2pt);

\draw[color=red][line width=1.5pt] (-5,-1)  .. controls(-8, 0.5) and (-10.5,-1) ..  (-10.5,-2)
			node[pos=0,right] {$q_{0}$}
			node[pos=0.11,above] {$r_{1}$}
			node[pos=0.4,above] {$r_{2}$};
			
\draw[color=red][line width=1.5pt] (-5,-3)  .. controls(-8, -3.5) and (-10.5,-3.5) ..  (-10.5,-2)
			node[pos=0,right] {$q_{1}$}
			node[pos=0.15,below] {$r_{l}$};
		
\draw[->][color=brown][line width=1pt] (-6.05,-0.7)  --   (-5.1,-2.9)
			node[pos=0.3,left] {$d_{q_{0},q_{1}}^{\Delta }$};

	 \end{tikzpicture}
 		
}							

\caption{Detour curve.}\label{desviacion}
						
 \end{figure}	

Let $j$ be an arc that does not belong to $\tau$, an arc representation is defined. In case of a punctured disk the arc representation $M_j =(M,P)$ is defined easily following \cite[Definition 34]{D}.

\smallskip

For every vertex $i\in Q_{0}$ the $K$ vector space $M_i$ is defined as:

\begin{itemize}
	\item[1)] If the arc $j$ is not the arc incident at $p$ labeled $\bowtie$ then $M_{i}=K^{{\mathds A}}$, where ${\mathds A}$ is the number of intersection points of $j$ with $i$.

	\item[2)] Let $j$ be the arc incident at $p$ labeled $\bowtie$ and $q$ the other endpoint of $j$,  the arc $j$  is replaced by the loop $j^{o}$  based at $q$ which surrounds the puncture $p$. 

		\begin{itemize}
			\item[$\bullet$] If the arc $j'$ with endpoints $q$ and $p$ labeled plain belongs to $\tau$ then $M_{i}=K^{{\mathds A}}$, where ${{\mathds A}}$ is the number of  intersection points of $j^{o}$ with $i$.

			\item[$\bullet$]  If the arc $j'$ with endpoints $q$ and $p$ labeled plain  does not belong to $\tau$ then there is an arc $\overline{j}\in \tau$ not incident at $p$, such that $j^{o}=[q,q_{0}]_{j^{o}}\cup \gamma_{q_{0},q_{1}}\cup[q_{1},q]_{j^{o}}$ and $\gamma_{q_{0},q_{1}}$ surrounds the puncture $p$ counterclockwise and $q_{o},q_{1}\in \overline{j}$. The vector space is $M_{i}=K^{\mathds A}$ if $i\neq \overline{j}$, on the other hand,  $M_{i}=\frac{K^{A}}{K_{q_{0}}}$ for $i=\overline{j}$, where ${\mathds A}$ is the number of intersection points of $\gamma_{q_{0},q_{1}}\cup[q_{1},q]_{j^{o}}$  with the arc $i$ and $K_{q_{0}}$ is the copy of the field corresponding to the intersection point of $j^{o}$ with $\overline{j}$.

		\end{itemize}

\end{itemize}

For every arrow $\alpha: i\longrightarrow k$ in $Q_{1}$ the linear transformation $\varphi_{\alpha}$ is defined in two steps. First we define a linear transformation $\overline{\varphi}_{\alpha}$ and then $\varphi_{\alpha}$ it is defined modifying $\overline{\varphi}_{\alpha}$.

If $j$ is an arc with labeled plain and $b_{1},...,b_{m}$ and $c_{1},...,c_{t}$ is an enumeration of the intersection points of $j$ with $i$ and $k$ respectively. The entry $m_{c_{h},b_{l}}$ of the matrix $\overline{\varphi}_{\alpha}$ is defined as follows:

$$m_{c_{h},b_{l}}= \left\{ \begin{array}{lcc}
             1 &     \mbox{if the interior of $[b_{l},c_{h}]_{j}$ does not intersect any arc of $\tau$ }.\\
	  1  &  \mbox{if there is a segment $\gamma_{c_{h'},c_{h}}$ of $j$ which sourrounds the puncture $p$}\\
	    &  \mbox{counterclockwise and $r_{1}=b_{l}$}.\\
             0   & \mbox{else}

             \end{array}
   \right.$$

If $j$ is the arc incident at $p$ labeled  $\bowtie$ and $q$ is the other end of $j$ then;

\begin{itemize}
	
	\item[$\bullet$] The arc $j$ is replaced by the loop $j^{o}$ based on $q$ in the definition of $m_{c_{h},b_{l}}$ if the arc $j'$ with endpoints $q$ and $p$ labeled plain belongs to $\tau$.

	\item[$\bullet$] The arc $j$ is replaced by $\gamma_{q_{0},q_{1}}\cup[q_{1},q]_{j^{o}}$  in the definition of $m_{c_{h},b_{l}}$ if the arc $j'$ with endpoints $q$ and $p$ labeled plain does not belong to $\tau$.
	
\end{itemize}
 
Now, we will define the linear transformation $\varphi_{\alpha}$. If  $j$ is the arc incident at $p$  labeled  $\bowtie$ and $q$ the other endpoint of $j$  then;

\begin{itemize}
	
	\item[$\bullet$]$\varphi_{\alpha}=\overline{\varphi_{\alpha}}$ if the arc $j'$ with endpoints $q$ and $p$ labeled plain belongs to $\tau$.

	\item[$\bullet$] 

$$\varphi_{\alpha}= \left\{ \begin{array}{lcc}
             \overline{\varphi}_{\alpha} &     \mbox{if $i\neq \overline{j}\neq {k}$}\\
	   \overline{\varphi}^{q_{0}}_{\alpha}  &  \mbox{if $i=\overline{j}$}\\
	   ^{q_{0}}\overline{\varphi}_{\alpha}  &  \mbox{if $k=\overline{j}$}
             \end{array}
   \right.$$

if the arc $j'$ with endpoints $q$ and $p$  labeled plain does not belong to $\tau$. Where, $\overline{\varphi}^{q_{0}}_{\alpha}$ is the matrix obtained from  $\overline{\varphi}_{\alpha}$ by removing the column that corresponds to the intersection point $q_{0}$ of $\gamma_{q_{0},q_{1}}\cup[q_{1},q]_{j^{o}}$ with $i$, and $^{q_{0}}\overline{\varphi}_{\alpha}$ is the matrix obtained from  $\overline{\varphi}_{\alpha}$ by removing the arrow  which corresponds to the intersection point $q_{0}$ of $\gamma_{q_{0},q_{1}}\cup[q_{1},q]_{j^{o}}$ with $k$.
	
\end{itemize}

On the other hand, if $j$ is an arc labeled plain in both endpoints then $\varphi_{\alpha}=\overline{\varphi}_{\alpha}$.


\subsection{Examples on the disk}

\begin{ex} See the next arcs crossing over a triangulated punctured disk.

\vspace{-4mm}
\begin{figure}[H]
\begin{tabular}{ p{50mm}  p{60mm}}
\begin{tikzpicture}[scale=0.8]


\draw (0,0) circle (2cm);    

\filldraw [black] 	(0,0)  circle (2pt)
			(1.7,1.05)  circle (2pt)
			(-1.5,1.35)  circle (2pt)
			(-2,0)  circle (2pt)
			(1.5,-1.35)  circle (2pt)
			(-1.5,-1.35)  circle (2pt)
			(0,-2)  circle (2pt)
			(0,2)  circle (2pt);

\draw[color=blue][line width=1pt] (-1.5,1.35) -- (-1.5,-1.35)
node[pos=0.25,left] {{\tiny1}};	
\draw[color=blue][line width=1pt] (0,2) -- (-1.5,-1.35)
node[pos=0.75,left] {{\tiny2}};
\draw[color=blue][line width=1pt] (0,2) -- (0,0)
node[pos=0.5,left] {{\tiny4}};
\draw[color=blue][line width=1pt] (0,2) -- (1.5,-1.35)
node[pos=0.75,right] {{\tiny6}};
\draw[color=blue][line width=1pt] (0,0) -- (-1.5,-1.35)
node[pos=0.5,right] {{\tiny3}};
\draw[color=blue][line width=1pt] (0,0) -- (1.5,-1.35)
node[pos=0.6,right] {{\tiny5}};
\draw[color=blue][line width=1pt] (-1.5,-1.35) -- (1.5,-1.35)
node[pos=0.35,below] {{\tiny7}};

\draw[color=red][line width=1pt] (-2,0) .. controls(0, 1) and (1.5,0.5) ..  (0,-2)
node[pos=0.5,above] {$\alpha$};

\draw[color=green][line width=1pt] (-1.5,1.35) .. controls(-1, -1.2) and (1.25,-0.65) ..  (1.7,1.05)
node[pos=0.5,below] {$\beta$};

 \end{tikzpicture}

&
\begin{tikzpicture}[scale=0.9]


\draw  (-4,0)node{$1$};
\draw  (-2.5,0)node{$2$};
\draw  (-1,1.15)node{$4$};
\draw  (-1,-1.15)node{$3$};
\draw  (0.5,0)node{$5$};
\draw  (2,1.15)node{$6$};
\draw  (2,-1.15)node{$7$};

\draw[->][line width=0.5pt] (-2.65,0) -- (-3.85,0)
node[pos=0.5,above] {$\alpha_{1}$};	

\draw[->][line width=0.5pt] (-2.35,0.1) -- (-1.15,1)
node[pos=0.5,above] {$\alpha_{2}$};	

\draw[->][line width=0.5pt] (-1.15,-1.1) -- (-2.35,-0.1)
node[pos=0.5,below] {$\alpha_{3}$};	

\draw[->][line width=0.5pt] (-1,0.95) -- (-1,-0.95)
node[pos=0.5,left] {$\alpha_{4}$};

\draw[->][line width=0.5pt] (0.35,0.1) -- (-0.85,1)
node[pos=0.5,below] {$\alpha_{5}$};	

\draw[->][line width=0.5pt] (-0.855,-1.05) -- (0.35,-0.1)
node[pos=0.5,above] {$\alpha_{6}$};	

\draw[->][line width=0.5pt] (-0.85,1.15) -- (1.85,1.15)
node[pos=0.5,above] {$\alpha_{7}$};	

\draw[<-][line width=0.5pt] (-0.85,-1.15) -- (1.85,-1.15)
node[pos=0.5,below] {$\alpha_{8}$};	

\draw[->][line width=0.5pt] (1.85,1) -- (0.65,0.1)
node[pos=0.5,above] {$\alpha_{9}$};	

\draw[->][line width=0.5pt]  (0.65,-0.1) -- (1.85,-1)
node[pos=0.5,above] {$\alpha_{10}$};	

 \end{tikzpicture}




\end{tabular}

\caption{Left: triangulation $\tau$ of the disk and arcs $\alpha$ and $\beta$. Right: quiver $Q(\tau)$}\label{arcos}
\end{figure}
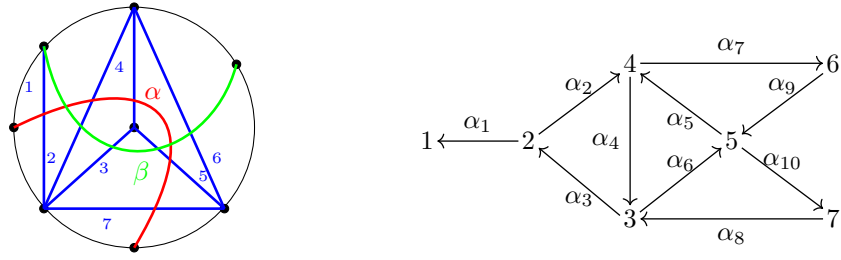

See Figure \ref{arcos}. The arc representations  $M_\alpha$ and $M_\beta$ are the following,

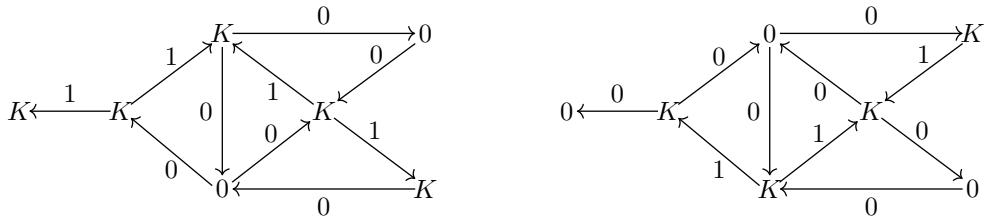
\begin{figure}[H]
\begin{tabular}{ p{70mm}  p{70mm}}
\begin{tikzpicture}[scale=0.9 ]


\draw  (-4,0)node{$K$};
\draw  (-2.5,0)node{$K$};
\draw  (-1,1.15)node{$K$};
\draw  (-1,-1.15)node{$0$};
\draw  (0.5,0)node{$K$};
\draw  (2,1.15)node{$0$};
\draw  (2,-1.15)node{$K$};

\draw[->][line width=0.5pt] (-2.65,0) -- (-3.85,0)
node[pos=0.5,above] {$1$};	

\draw[->][line width=0.5pt] (-2.35,0.1) -- (-1.15,1)
node[pos=0.5,above] {$1$};	

\draw[->][line width=0.5pt] (-1.15,-1.1) -- (-2.35,-0.1)
node[pos=0.5,below] {$0$};	

\draw[->][line width=0.5pt] (-1,0.95) -- (-1,-0.95)
node[pos=0.5,left] {$0$};	

\draw[->][line width=0.5pt] (0.35,0.1) -- (-0.85,1)
node[pos=0.5,below] {$1$};	

\draw[->][line width=0.5pt](-0.855,-1.05) -- (0.3,-0.15)
node[pos=0.5,above] {$0$};	

\draw[->][line width=0.5pt] (-0.85,1.15) -- (1.85,1.15)
node[pos=0.5,above] {$0$};	

\draw[<-][line width=0.5pt] (-0.85,-1.15) -- (1.85,-1.15)
node[pos=0.5,below] {$0$};	

\draw[->][line width=0.5pt] (1.85,1) -- (0.7,0.15)
node[pos=0.5,above] {$0$};	

\draw[->][line width=0.5pt]  (0.65,-0.1) -- (1.85,-1)
node[pos=0.5,above] {$1$};	


 \end{tikzpicture}

&
\begin{tikzpicture}[scale=0.9]


\draw  (-4,0)node{$0$};
\draw  (-2.5,0)node{$K$};
\draw  (-1,1.15)node{$0$};
\draw  (-1,-1.15)node{$K$};
\draw  (0.5,0)node{$K$};
\draw  (2,1.15)node{$K$};
\draw  (2,-1.15)node{$0$};

\draw[->][line width=0.5pt] (-2.65,0) -- (-3.85,0)
node[pos=0.5,above] {$0$};	

\draw[->][line width=0.5pt] (-2.35,0.1) -- (-1.15,1)
node[pos=0.5,above] {$0$};	

\draw[->][line width=0.5pt] (-1.15,-1.1) -- (-2.35,-0.1)
node[pos=0.5,below] {$1$};	

\draw[->][line width=0.5pt] (-1,0.95) -- (-1,-0.95)
node[pos=0.5,left] {$0$};	

\draw[->][line width=0.5pt] (0.35,0.1) -- (-0.85,1)
node[pos=0.5,below] {$0$};	

\draw[->][line width=0.5pt] (-0.855,-1.05) -- (0.3,-0.15)
node[pos=0.5,above] {$1$};	

\draw[->][line width=0.5pt] (-0.85,1.15) -- (1.85,1.15)
node[pos=0.5,above] {$0$};	

\draw[<-][line width=0.5pt] (-0.85,-1.15) -- (1.85,-1.15)
node[pos=0.5,below] {$0$};	

\draw[->][line width=0.5pt] (1.85,1) -- (0.7,0.15)
node[pos=0.5,above] {$1$};

\draw[->][line width=0.5pt]  (0.65,-0.1) -- (1.85,-1)
node[pos=0.5,above] {$0$};	


 \end{tikzpicture}




\end{tabular}

\caption{Arc representations  $M_{\alpha}$ (left) and $M_{\beta}$ (right).}\label{rep alpha beta}
\end{figure}

The possible four middle terms are defined by four multicurves (a)-(d). By Proposition \ref{prop:preserva-dimension}, we see that in the four cases the number of crossings of each multicurve with the triangulation $\tau$ is preserved so each one of the multicurves in Figure \ref{terminos medios 1} (a)-(d) defines a middle term in an extension.  

\begin{figure}[H]
\begin{tabular}{ p{60mm}  p{50mm}}
\begin{tikzpicture}[scale=0.8]


\draw (0,0) circle (2cm);    
\draw (0,0) circle (2cm);    

\filldraw [black] 	(0,0)  circle (2pt)
			(1.7,1.05)  circle (2pt)
			(-1.5,1.35)  circle (2pt)
			(-2,0)  circle (2pt)
			(1.5,-1.35)  circle (2pt)
			(-1.5,-1.35)  circle (2pt)
			(0,-2)  circle (2pt)
			(0,2)  circle (2pt);

\draw[color=blue][line width=1pt] (-1.5,1.35) -- (-1.5,-1.35)
node[pos=0.25,left] {{\tiny1}};	
\draw[color=blue][line width=1pt] (0,2) -- (-1.5,-1.35)
node[pos=0.75,left] {{\tiny2}};
\draw[color=blue][line width=1pt] (0,2) -- (0,0)
node[pos=0.5,left] {{\tiny4}};
\draw[color=blue][line width=1pt] (0,2) -- (1.5,-1.35)
node[pos=0.75,right] {{\tiny6}};
\draw[color=blue][line width=1pt] (0,0) -- (-1.5,-1.35)
node[pos=0.5,right] {{\tiny3}};
\draw[color=blue][line width=1pt] (0,0) -- (1.5,-1.35)
node[pos=0.6,right] {{\tiny5}};
\draw[color=blue][line width=1pt] (-1.5,-1.35) -- (1.5,-1.35)
node[pos=0.35,below] {{\tiny7}};

\draw(0,-2.3)node{$p$};
\draw(-2.3,0)node{$n$};
\draw(-1.55,1.6) node{$r$};
\draw(1.8,1.3)node{$s$};

\draw[color=red][line width=1pt] (-2,0) --  (0,0)
node[pos=0.4,above] {{\tiny $\gamma_{1}$}};

\draw  [color=red](-0.23,0)node{\rotatebox{90}{$\bowtie$}};

\draw[color=red][line width=1pt] (-1.5,1.35) --  (0,0)
node[pos=0.4,above] {{\tiny $\gamma_{2}$}};

\draw[color=red][line width=1pt] (1.7,1.05) -- (0,-2)
node[pos=0.2,left] {{\tiny $\gamma_{3}$}};

 \end{tikzpicture}

&
\begin{tikzpicture}[scale=0.8]


\draw (0,0) circle (2cm);    
\draw (0,0) circle (2cm);    

\filldraw [black] 	(0,0)  circle (2pt)
			(1.7,1.05)  circle (2pt)
			(-1.5,1.35)  circle (2pt)
			(-2,0)  circle (2pt)
			(1.5,-1.35)  circle (2pt)
			(-1.5,-1.35)  circle (2pt)
			(0,-2)  circle (2pt)
			(0,2)  circle (2pt);

\draw[color=blue][line width=1pt] (-1.5,1.35) -- (-1.5,-1.35)
node[pos=0.25,left] {{\tiny1}};	
\draw[color=blue][line width=1pt] (0,2) -- (-1.5,-1.35)
node[pos=0.75,left] {{\tiny2}};
\draw[color=blue][line width=1pt] (0,2) -- (0,0)
node[pos=0.5,left] {{\tiny4}};
\draw[color=blue][line width=1pt] (0,2) -- (1.5,-1.35)
node[pos=0.75,right] {{\tiny6}};
\draw[color=blue][line width=1pt] (0,0) -- (-1.5,-1.35)
node[pos=0.5,right] {{\tiny3}};
\draw[color=blue][line width=1pt] (0,0) -- (1.5,-1.35)
node[pos=0.6,right] {{\tiny5}};
\draw[color=blue][line width=1pt] (-1.5,-1.35) -- (1.5,-1.35)
node[pos=0.35,below] {{\tiny7}};

\draw(0,-2.3)node{$p$};
\draw(-2.3,0)node{$n$};
\draw(-1.55,1.6) node{$r$};
\draw(1.8,1.3)node{$s$};

\draw[color=red][line width=1pt] (-2,0) --  (0,0)
node[pos=0.4,above] {{\tiny $\gamma_{1}$}};

\draw  [color=red](-0.18,0.18)node{\rotatebox{55}{$\bowtie$}};
\draw[color=red][line width=1pt] (-1.5,1.35) --  (0,0)
node[pos=0.4,above] {{\tiny $\gamma_{2}$}};

\draw[color=red][line width=1pt] (1.7,1.05) -- (0,-2)
node[pos=0.2,left] {{\tiny $\gamma_{3}$}};

 \end{tikzpicture}




\end{tabular}

\end{figure}

\vspace{-14mm}

\begin{figure}[H]
\begin{tabular}{ p{60mm}  p{50mm}}
\begin{tikzpicture}[scale=0.8]


\draw (0,0) circle (2cm);    

\filldraw [black] 	(0,0)  circle (2pt)
			(1.7,1.05)  circle (2pt)
			(-1.5,1.35)  circle (2pt)
			(-2,0)  circle (2pt)
			(1.5,-1.35)  circle (2pt)
			(-1.5,-1.35)  circle (2pt)
			(0,-2)  circle (2pt)
			(0,2)  circle (2pt);

\draw[color=blue][line width=1pt] (-1.5,1.35) -- (-1.5,-1.35)
node[pos=0.25,left] {{\tiny1}};	
\draw[color=blue][line width=1pt] (0,2) -- (-1.5,-1.35)
node[pos=0.75,left] {{\tiny2}};
\draw[color=blue][line width=1pt] (0,2) -- (0,0)
node[pos=0.5,left] {{\tiny4}};
\draw[color=blue][line width=1pt] (0,2) -- (1.5,-1.35)
node[pos=0.75,right] {{\tiny6}};
\draw[color=blue][line width=1pt] (0,0) -- (-1.5,-1.35)
node[pos=0.5,left] {{\tiny3}};
\draw[color=blue][line width=1pt] (0,0) -- (1.5,-1.35)
node[pos=0.6,right] {{\tiny5}};
\draw[color=blue][line width=1pt] (-1.5,-1.35) -- (1.5,-1.35)
node[pos=0.35,below] {{\tiny7}};

\draw(0,-2.3)node{$p$};
\draw(-2.3,0)node{$n$};
\draw(-1.55,1.6) node{$r$};
\draw(1.8,1.3)node{$s$};

\draw[color=red][line width=1pt] (-2,0) .. controls(-0.5, -1) and (1,-0.5) ..  (1.7,1.05)
node[pos=0.4,below] {{\tiny $\gamma_{1}$}};

\draw[color=red][line width=1pt] (-1.5,1.35) .. controls(-0.5, 1) and (1.5,0.5) ..  (0,-2)
node[pos=0.5,right] {{\tiny $\gamma_{2}$}};

 \end{tikzpicture}

&

\begin{tikzpicture}[scale=0.8]


\draw (0,0) circle (2cm);    
\draw (0,0) circle (2cm);    

\filldraw [black] 	(0,0)  circle (2pt)
			(1.7,1.05)  circle (2pt)
			(-1.5,1.35)  circle (2pt)
			(-2,0)  circle (2pt)
			(1.5,-1.35)  circle (2pt)
			(-1.5,-1.35)  circle (2pt)
			(0,-2)  circle (2pt)
			(0,2)  circle (2pt);

\draw[color=blue][line width=1pt] (-1.5,1.35) -- (-1.5,-1.35)
node[pos=0.25,left] {{\tiny1}};	
\draw[color=blue][line width=1pt] (0,2) -- (-1.5,-1.35)
node[pos=0.75,left] {{\tiny2}};
\draw[color=blue][line width=1pt] (0,2) -- (0,0)
node[pos=0.5,left] {{\tiny4}};
\draw[color=blue][line width=1pt] (0,2) -- (1.5,-1.35)
node[pos=0.75,right] {{\tiny6}};
\draw[color=blue][line width=1pt] (0,0) -- (-1.5,-1.35)
node[pos=0.5,right] {{\tiny3}};
\draw[color=blue][line width=1pt] (0,0) -- (1.5,-1.35)
node[pos=0.6,right] {{\tiny5}};
\draw[color=blue][line width=1pt] (-1.5,-1.35) -- (1.5,-1.35)
node[pos=0.35,below] {{\tiny7}};

\draw(0,-2.3)node{$p$};
\draw(-2.3,0)node{$n$};
\draw(-1.55,1.6) node{$r$};
\draw(1.8,1.3)node{$s$};

\draw[color=red][line width=1pt] (-2,0) .. controls(1.5, -1.25) and (1,0.75) ..  (-1.5,1.35)
node[pos=0.75,right] {{\tiny $\gamma_{1}$}};

\draw[color=red][line width=1pt] (0,-2) --  (1.7,1.05)
node[pos=0.75,right] {{\tiny $\gamma_{2}$}};

 \end{tikzpicture}

\end{tabular}

\caption{Multicurves (a), (b), (c) and (d) -in lexicographic order- corresponding to items (6), (7), (8) and (9) of Figures \ref{sec3} and \ref{sec4}.}\label{terminos medios 1}
\end{figure}
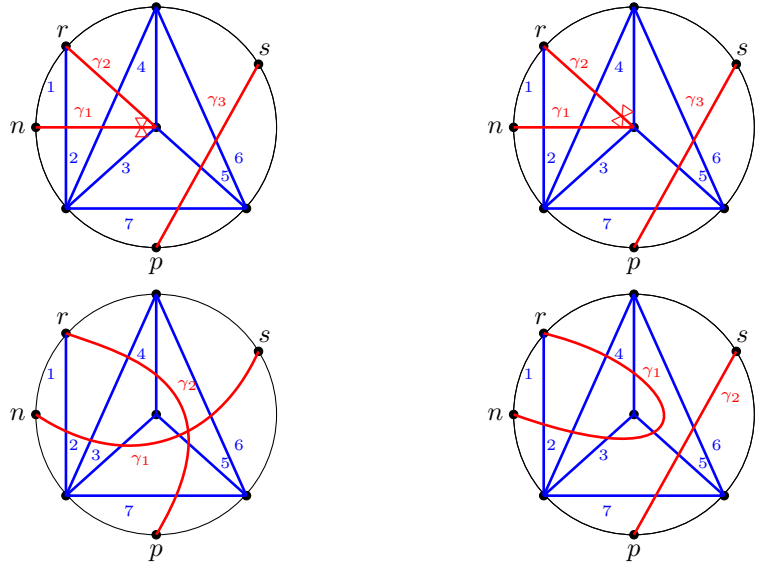

To compute $M_{\gamma_{1}}\oplus M_{\gamma_{2}}\oplus M_{\gamma_{3}}$ in (a) Figure \ref{terminos medios 1}, the arc representations $M_{\gamma_{2}}$ and $M_{\gamma_{3}}$ are computed easily, while to compute $M_{\gamma_{1}}$, the arc $\gamma_{1}$  is replaced by the curve $\gamma_{q_{0},q_{1}}\cup[q_{1},n]_{j^{o}}$ and the detour curve is drawn. The arc representation computation for (b) is very similar to (a) and it is omitted.
See the arc representations in Figure \ref{arc rep p} and the extensions in Figure \ref{sucesiones exactas}.

\vspace{-2mm}

 \begin{figure}[H]
\centering				

\subfigure{\begin{tikzpicture}[scale=0.9]


\draw (0,0) circle (2cm);    
\draw (0,0) circle (2cm);    

\filldraw [black] 	(0,0)  circle (2pt)
			(1.7,1.05)  circle (2pt)
			(-1.5,1.35)  circle (2pt)
			(-2,0)  circle (2pt)
			(1.5,-1.35)  circle (2pt)
			(-1.5,-1.35)  circle (2pt)
			(0,-2)  circle (2pt)
			(0,2)  circle (2pt);

\draw[color=blue][line width=1pt] (-1.5,1.35) -- (-1.5,-1.35)
node[pos=0.25,left] {{\tiny1}};	
\draw[color=blue][line width=1pt] (0,2) -- (-1.5,-1.35)
node[pos=0.25,left] {{\tiny2}};
\draw[color=blue][line width=1pt] (0,2) -- (0,0)
node[pos=0.5,left] {{\tiny4}};
\draw[color=blue][line width=1pt] (0,2) -- (1.5,-1.35)
node[pos=0.75,right] {{\tiny6}};
\draw[color=blue][line width=1pt] (0,0) -- (-1.5,-1.35)
node[pos=0.35,right] {{\tiny3}};
\draw[color=blue][line width=1pt] (0,0) -- (1.5,-1.35)
node[pos=0.6,right] {{\tiny5}};
\draw[color=blue][line width=1pt] (-1.5,-1.35) -- (1.5,-1.35)
node[pos=0.35,below] {{\tiny7}};

\draw(-2.3,0)node{$n$};

\draw[color=red][line width=0.5pt] (-2,0) --  (0,0);
\draw  [color=red](-0.23,0)node{\rotatebox{90}{$\bowtie$}};
\draw[->][color=brown][line width=0.75pt] (-0.85,-0.7) --  (-0.6,0.4)
node[pos=0.4,right] {{\tiny $d$}};
\draw[color=magenta][line width=1pt] (-2,0) .. controls(1, 1.5) and (1,-1) .. (-1.25,-0.75);
\draw[color=magenta](-1.35,-0.7)node{{\tiny $q_{0}$}};	
\draw[color=magenta](-0.9,0.6)node{{\tiny $q_{1}$}};	
	 \end{tikzpicture}
 	
}

\caption{ Curve $\gamma_{q_{0},q_{1}}\cup[q_{1},n]_{j^{o}}$.}
						
 \end{figure}
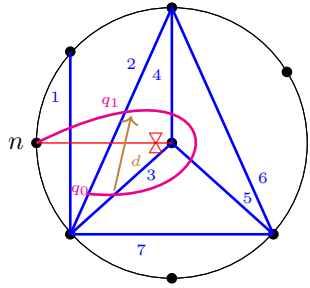

Then, the arc representations of the middle terms (a), (c) and (d) -in lexicographic order- from Figure \ref{terminos medios 1} are the following

\begin{figure}[H]
\begin{tikzpicture}[scale=0.85]

\draw  (-4,0)node{$K$};
\draw  (-2.5,0)node{$K^{2}$};
\draw  (-1,1.15)node{$K$};
\draw  (-1,-1.15)node{$K$};
\draw  (0.5,0)node{$K^{2}$};
\draw  (2,1.15)node{$K$};
\draw  (2,-1.15)node{$K$};

\draw[->][line width=0.5pt] (-2.72,0) -- (-3.85,0)
node[pos=0.5,above] {{\tiny $ \begin{bmatrix}
1&0\\
\end{bmatrix} $}};	

\draw[->][line width=0.5pt] (-2.27,0.18) -- (-1.15,1)
node[pos=0.5,left] {{\tiny $ \begin{bmatrix}
1&0\\
\end{bmatrix} $}};

\draw[->][line width=0.5pt] (-1.15,-1.1) -- (-2.35,-0.1)
node[pos=0.5,below] {{\tiny $ \begin{bmatrix}
1\\
0
\end{bmatrix} $}};

\draw[->][line width=0.5pt] (-1,0.95) -- (-1,-0.95)
node[pos=0.5,left] {$0$};

\draw[->][line width=0.5pt] (0.3,0.15) -- (-0.85,1)
node[pos=0.5,right] {{\tiny $ \begin{bmatrix}
1&0\\
\end{bmatrix} $}};

\draw[->][line width=0.5pt] (-0.855,-1.05) -- (0.27,-0.15)
node[pos=0.5,above] {{\tiny $ \begin{bmatrix}
1\\
0
\end{bmatrix} $}};	

\draw[->][line width=0.5pt] (-0.85,1.15) -- (1.85,1.15)
node[pos=0.5,above] {$0$};	

\draw[<-][line width=0.5pt] (-0.85,-1.15) -- (1.85,-1.15)
node[pos=0.5,below] {$0$};	

\draw[->][line width=0.5pt] (1.85,1) -- (0.72,0.17)
node[pos=0.35,below]  {{\tiny $ \begin{bmatrix}
0\\
1
\end{bmatrix} $}};

\draw[->][line width=0.5pt]  (0.65,-0.1) -- (1.85,-1)
node[pos=0.65,right]{{\tiny $ \begin{bmatrix}
0&1\\
\end{bmatrix} $}};	

\end{tikzpicture}
\end{figure}



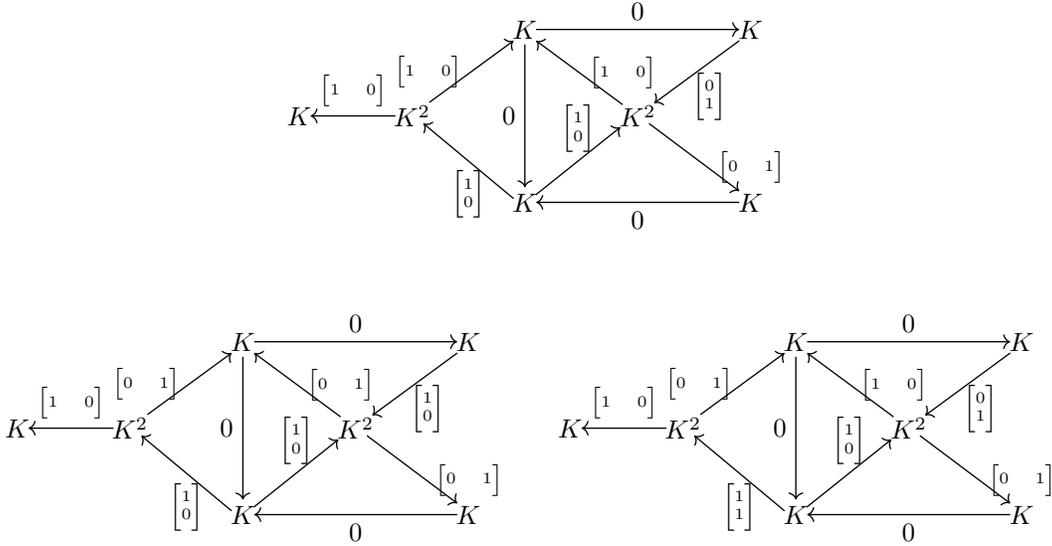
\begin{figure}[H]
\begin{tabular}{ p{70mm}  p{70mm}}
\begin{tikzpicture}[scale=0.85]

\draw  (-4,0)node{$K$};
\draw  (-2.5,0)node{$K^{2}$};
\draw  (-1,1.15)node{$K$};
\draw  (-1,-1.15)node{$K$};
\draw  (0.5,0)node{$K^{2}$};
\draw  (2,1.15)node{$K$};
\draw  (2,-1.15)node{$K$};

\draw[->][line width=0.5pt] (-2.72,0) -- (-3.85,0)
node[pos=0.5,above] {{\tiny $ \begin{bmatrix}
1&0\\
\end{bmatrix} $}};	

\draw[->][line width=0.5pt] (-2.27,0.18) -- (-1.15,1)
node[pos=0.5,left] {{\tiny $ \begin{bmatrix}
0&1\\
\end{bmatrix} $}};

\draw[->][line width=0.5pt] (-1.15,-1.1) -- (-2.35,-0.1)
node[pos=0.5,below] {{\tiny $ \begin{bmatrix}
1\\
0
\end{bmatrix} $}};

\draw[->][line width=0.5pt] (-1,0.95) -- (-1,-0.95)
node[pos=0.5,left] {$0$};

\draw[->][line width=0.5pt] (0.3,0.15) -- (-0.85,1)
node[pos=0.5,right] {{\tiny $ \begin{bmatrix}
0&1\\
\end{bmatrix} $}};

\draw[->][line width=0.5pt] (-0.855,-1.05) -- (0.27,-0.15)
node[pos=0.5,above] {{\tiny $ \begin{bmatrix}
1\\
0
\end{bmatrix} $}};	

\draw[->][line width=0.5pt] (-0.85,1.15) -- (1.85,1.15)
node[pos=0.5,above] {$0$};	

\draw[<-][line width=0.5pt] (-0.85,-1.15) -- (1.85,-1.15)
node[pos=0.5,below] {$0$};	

\draw[->][line width=0.5pt] (1.85,1) -- (0.72,0.17)
node[pos=0.35,below]  {{\tiny $ \begin{bmatrix}
1\\
0
\end{bmatrix} $}};

\draw[->][line width=0.5pt]  (0.65,-0.1) -- (1.85,-1)
node[pos=0.65,right]{{\tiny $ \begin{bmatrix}
0&1\\
\end{bmatrix} $}};	

 \end{tikzpicture}

&
\begin{tikzpicture}[scale=0.85]


\draw  (-4,0)node{$K$};
\draw  (-2.5,0)node{$K^{2}$};
\draw  (-1,1.15)node{$K$};
\draw  (-1,-1.15)node{$K$};
\draw  (0.5,0)node{$K^{2}$};
\draw  (2,1.15)node{$K$};
\draw  (2,-1.15)node{$K$};

\draw[->][line width=0.5pt] (-2.72,0) -- (-3.85,0)
node[pos=0.5,above] {{\tiny $ \begin{bmatrix}
1&0\\
\end{bmatrix} $}};	

\draw[->][line width=0.5pt] (-2.27,0.18) -- (-1.15,1)
node[pos=0.5,left] {{\tiny $ \begin{bmatrix}
0&1\\
\end{bmatrix} $}};

\draw[->][line width=0.5pt] (-1.15,-1.1) -- (-2.35,-0.1)
node[pos=0.5,below] {{\tiny $ \begin{bmatrix}
1\\
1
\end{bmatrix} $}};

\draw[->][line width=0.5pt] (-1,0.95) -- (-1,-0.95)
node[pos=0.5,left] {$0$};

\draw[->][line width=0.5pt] (0.3,0.15) -- (-0.85,1)
node[pos=0.5,right] {{\tiny $ \begin{bmatrix}
1&0\\
\end{bmatrix} $}};

\draw[->][line width=0.5pt] (-0.855,-1.05) -- (0.27,-0.15)
node[pos=0.5,above] {{\tiny $ \begin{bmatrix}
1\\
0
\end{bmatrix} $}};	

\draw[->][line width=0.5pt] (-0.85,1.15) -- (1.85,1.15)
node[pos=0.5,above] {$0$};	

\draw[<-][line width=0.5pt] (-0.85,-1.15) -- (1.85,-1.15)
node[pos=0.5,below] {$0$};	

\draw[->][line width=0.5pt] (1.85,1) -- (0.72,0.17)
node[pos=0.35,below]  {{\tiny $ \begin{bmatrix}
0\\
1
\end{bmatrix} $}};

\draw[->][line width=0.5pt]  (0.65,-0.1) -- (1.85,-1)
node[pos=0.65,right]{{\tiny $ \begin{bmatrix}
0&1\\
\end{bmatrix} $}};

 \end{tikzpicture}




\end{tabular}

\caption{Arc representations corresponding to middle terms (a), (c) and (d).}\label{arc rep p}
\end{figure}

The extensions (a), (c) and (d) for our example are the following;

\begin{figure}[H]
{
\begin{tikzpicture}[scale=1]

\draw  (-3.5,0)node{$0$};
\draw  (-2.5,0)node{$M_{\alpha}$};
\draw  (-0.5,0)node{{\tiny $M_{\gamma_{1}}\oplus M_{\gamma_{2}}\oplus M_{\gamma_{3}}$}};
\draw  (1.75,0)node{$M_{\beta}$};
\draw  (2.75,0)node{$0$};

\draw[->][line width=0.5pt] (-3.25,0) -- (-2.85,0);
\draw[->][line width=0.5pt] (-2.2,0) -- (-1.75,0);
\draw[->][line width=0.5pt] (0.6,0) -- (1.45,0);
\draw[->][line width=0.5pt] (2.05,0) -- (2.6,0);

\draw  (-1.95,0.25)node{$f$};
\draw  (-1.95,0.75)node{$||$};
\draw  (-1.75,2)node{{\tiny $ \left\{ \begin{array}{lc}
             1 &     \mbox{if $i=1,4,7$}\\
              &     \\
	   \begin{bmatrix}
1 \\
1
\end{bmatrix}  &  \mbox{if $i=2,5$}\\
              &     \\
	   0  &  \mbox{if $i=3,6$}
             \end{array}
   \right.$}};

\draw  (1.25,0.25)node{$g$};
\draw  (1.25,0.75)node{$||$};
\draw  (1.5,2)node{{\tiny $ \left\{ \begin{array}{lc}
             1 &     \mbox{if $i=6$}\\
	 -1 &     \mbox{if $i=3$}\\
              &     \\
	   \begin{bmatrix}
-1&1
\end{bmatrix}  &  \mbox{if $i=2,5$}\\
              &     \\
	   0  &  \mbox{if $i=1,4,7$}
             \end{array}
   \right.$}};

 \end{tikzpicture}


}
\end{figure}

\begin{figure}[H]
\begin{tabular}{ p{80mm}  p{80mm}}
\begin{tikzpicture}[scale=1]

\draw  (-3.5,0)node{$0$};
\draw  (-2.5,0)node{$M_{\alpha}$};
\draw  (-0.5,0)node{ $M_{\gamma_{1}}\oplus M_{\gamma_{2}}$};
\draw  (1.75,0)node{$M_{\beta}$};
\draw  (2.75,0)node{$0$};

\draw[->][line width=0.5pt] (-3.25,0) -- (-2.85,0);
\draw[->][line width=0.5pt] (-2.2,0) -- (-1.5,0);
\draw[->][line width=0.5pt] (0.4,0) -- (1.45,0);
\draw[->][line width=0.5pt] (2.05,0) -- (2.6,0);

\draw  (-1.95,0.25)node{$f$};
\draw  (-1.95,0.75)node{$||$};
\draw  (-1.75,2.5)node{{\tiny $ \left\{ \begin{array}{lc}
             1 &     \mbox{if $i=1,4,7$}\\
              &     \\
	   \begin{bmatrix}
1 \\
1
\end{bmatrix}  &  \mbox{if $i=2$}\\
              &     \\

\begin{bmatrix}
0 \\
1
\end{bmatrix}  &  \mbox{if $i=5$}\\
              &     \\
	   0  &  \mbox{if $i=3,6$}
             \end{array}
   \right.$}};

\draw  (1.25,0.25)node{$g$};
\draw  (1.25,0.75)node{$||$};
\draw  (1.5,2.5)node{{\tiny $ \left\{ \begin{array}{lc}
             1 &     \mbox{if $i=6,3$}\\
	 \begin{bmatrix}
1&0
\end{bmatrix} &     \mbox{if $i=5$}\\
              &     \\
	   \begin{bmatrix}
1&-1
\end{bmatrix}  &  \mbox{if $i=2$}\\
              &     \\
	   0  &  \mbox{if $i=1,4,7$}
             \end{array}
   \right.$}};

 \end{tikzpicture}

&
\begin{tikzpicture}[scale=1]

\draw  (-3.5,0)node{$0$};
\draw  (-2.5,0)node{$M_{\alpha}$};
\draw  (-0.5,0)node{$M_{\gamma_{1}}\oplus M_{\gamma_{2}}$};
\draw  (1.75,0)node{$M_{\beta}$};
\draw  (2.75,0)node{$0$};

\draw[->][line width=0.5pt] (-3.25,0) -- (-2.85,0);
\draw[->][line width=0.5pt] (-2.2,0) -- (-1.5,0);
\draw[->][line width=0.5pt] (0.4,0) -- (1.45,0);
\draw[->][line width=0.5pt] (2.05,0) -- (2.6,0);

\draw  (-1.95,0.25)node{$f$};
\draw  (-1.95,0.75)node{$||$};
\draw  (-1.75,2.5)node{{\tiny $ \left\{ \begin{array}{lc}
             1 &     \mbox{if $i=4,7$}\\
             2 &     \mbox{if $i=1$}\\
              &     \\
	   \begin{bmatrix}
1 \\
1
\end{bmatrix}  &  \mbox{if $i=5$}\\
             &     \\
	   \begin{bmatrix}
2 \\
1
\end{bmatrix}  &  \mbox{if $i=2$}\\
              &     \\
	   0  &  \mbox{if $i=3,6$}
             \end{array}
   \right.$}};

\draw  (1.25,0.25)node{$g$};
\draw  (1.25,0.75)node{$||$};
\draw  (1.5,2.5)node{{\tiny $ \left\{ \begin{array}{lc}
             1 &     \mbox{if $i=6$}\\
	 -1 &     \mbox{if $i=3$}\\
              &     \\
	   \begin{bmatrix}
-1&1
\end{bmatrix}  &  \mbox{if $i=5$}\\
              &     \\
	   \begin{bmatrix}
1&-2
\end{bmatrix}  &  \mbox{if $i=2$}\\

              &     \\
	   0  &  \mbox{if $i=1,4,7$}
             \end{array}
   \right.$}};

 \end{tikzpicture}




\end{tabular}

\caption{Extensions (a), (c) and (d).}\label{sucesiones exactas}
\end{figure}

\end{ex}


\begin{ex} See the arcs crossing on the triangulated surface.

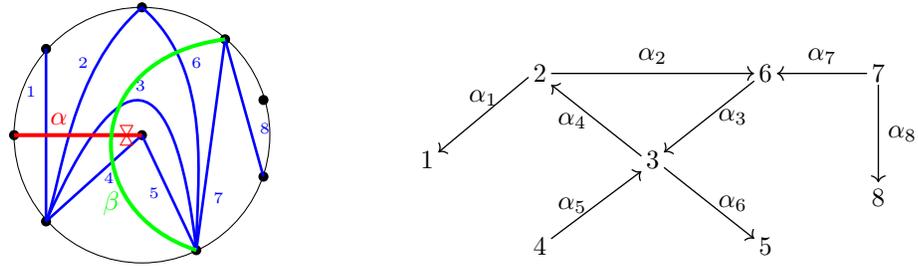
\begin{figure}[H]
\begin{tabular}{ p{50mm}  p{60mm}}
\begin{tikzpicture}[scale=0.85]


\draw (0,0) circle (2cm);    

\filldraw [black] 	(0,0)  circle (2pt)
			(1.9,0.55)  circle (2pt)
			(-1.5,1.35)  circle (2pt)
			(-2,0)  circle (2pt)
			(0.85,-1.8)  circle (2pt)
			(-1.5,-1.35)  circle (2pt)

			(0,2)  circle (2pt)

			(1.3,1.5)  circle (2pt)
			(1.9,-0.65)  circle (2pt);

\draw[color=blue][line width=1pt] (-1.5,1.35) -- (-1.5,-1.35)
node[pos=0.25,left] {{\tiny1}};	
\draw[color=blue][line width=1pt] (0,2) .. controls(-1.15, 1) and (-1.25,-0.75) .. (-1.5,-1.35)
node[pos=0.25,left] {{\tiny2}};
\draw[color=blue][line width=1pt] (-1.5,-1.35)  .. controls(-0.35, 1.25) and (0.5,1.25) .. (0.85,-1.8)
node[pos=0.5,above] {{\tiny 3}};
\draw[color=blue][line width=1pt] (0,2) .. controls(1.15, 1) and (0.85,-0.75) .. (0.85,-1.8)
node[pos=0.25,right] {{\tiny6}};
\draw[color=blue][line width=1pt] (0,0) -- (0.85,-1.8)
node[pos=0.5,left] {{\tiny5}};
\draw[color=blue][line width=1pt] (0,0) -- (-1.5,-1.35)
node[pos=0.5,right] {{\tiny4}};

\draw[color=blue][line width=1pt] (0.85,-1.8) -- (1.3,1.5) 
node[pos=0.25,right] {{\tiny7}};

\draw[color=blue][line width=1pt]  (1.3,1.5)  -- (1.9,-0.65) 
node[pos=0.65,right] {{\tiny8}};


\draw  [color=red](-0.23,0)node{\rotatebox{90}{$\bowtie$}};
\draw[color=red][line width=1.5pt] (-2,0) --  (0,0)
node[pos=0.35,above] {$\alpha$};

\draw[color=green][line width=1.5pt] (0.85,-1.8) .. controls(-1, -1.2) and (-1,1.2) ..  (1.3,1.5) 
node[pos=0.25,left] {$\beta$};

 \end{tikzpicture}

&
\begin{tikzpicture}[scale=1 ]


\draw  (-3,0)node{$1$};
\draw  (-1.5,1.15)node{$2$};
\draw  (0,0)node{$3$};
\draw  (1.5,1.15)node{$6$};
\draw  (-1.5,-1.15)node{$4$};
\draw  (1.5,-1.15)node{$5$};
\draw  (3,1.15)node{$7$};
\draw  (3,-0.5)node{$8$};

\draw[->][line width=0.5pt] (-1.65,1.1) -- (-2.85,0.1)
node[pos=0.5,above] {$\alpha_{1}$};	

\draw[->][line width=0.5pt] (-1.35,1.15) -- (1.35,1.15)
node[pos=0.5,above] {$\alpha_{2}$};	

\draw[->][line width=0.5pt] (1.35,1.05) -- (0.15,0.1)
node[pos=0.5,right] {$\alpha_{3}$};	
	
\draw[->][line width=0.5pt] (-0.15,0.05) -- (-1.35,1)
node[pos=0.5,left] {$\alpha_{4}$};

\draw[->][line width=0.5pt] (-1.35,-1.05) -- (-0.15,-0.15)
node[pos=0.5,left] {$\alpha_{5}$};	

\draw[->][line width=0.5pt] (0.15,-0.1) -- (1.35,-1.05)
node[pos=0.5,right] {$\alpha_{6}$};

\draw[->][line width=0.5pt] (2.85,1.15) -- (1.65,1.15)
node[pos=0.5,above] {$\alpha_{7}$};	

\draw[->][line width=0.5pt] (3,1) -- (3,-0.3)
node[pos=0.5,right] {$\alpha_{8}$};	

 \end{tikzpicture}




\end{tabular}
\caption{Left: triangulation $\tau$ and arcs $\alpha$ and $\beta$. Right: quiver $Q(\tau)$}\label{arcos 2}
\end{figure}

See the triangulation and the arcs $\alpha$ and $\beta$ in Figure \ref{arcos 2}. The arc representations $M_\alpha$ and $M_\beta$ are:

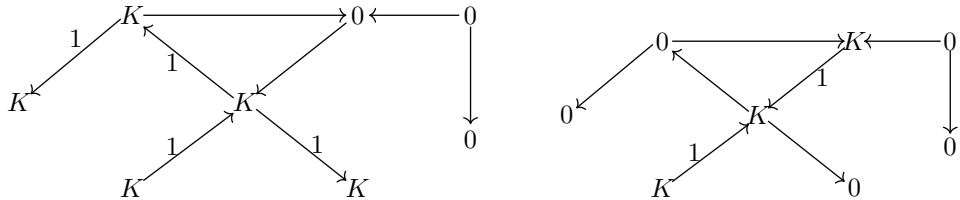
\begin{figure}[H]
\begin{tabular}{ p{70mm}  p{70mm}}
\begin{tikzpicture}[scale=0.8 ]

\draw  (-3,0)node{$K$};
\draw  (-1.5,1.15)node{$K$};
\draw  (0,0)node{$K$};
\draw  (1.5,1.15)node{$0$};
\draw  (-1.5,-1.15)node{$K$};
\draw  (1.5,-1.15)node{$K$};
\draw  (3,1.15)node{$0$};
\draw  (3,-0.5)node{$0$};

\draw[->][line width=0.5pt] (-1.65,1.1) -- (-2.85,0.1)
node[pos=0.5,above] {$1$};	

\draw[->][line width=0.5pt] (-1.35,1.15) -- (1.35,1.15);

\draw[->][line width=0.5pt] (1.35,1.05) -- (0.15,0.1);

\draw[->][line width=0.5pt] (-0.15,0.05) -- (-1.35,1)
node[pos=0.5,left] {$1$};

\draw[->][line width=0.5pt] (-1.35,-1.05) -- (-0.15,-0.15)
node[pos=0.5,left] {$1$};	

\draw[->][line width=0.5pt] (0.15,-0.1) -- (1.35,-1.05)
node[pos=0.5,right] {$1$};

\draw[->][line width=0.5pt] (2.85,1.15) -- (1.65,1.15);

\draw[->][line width=0.5pt] (3,1) -- (3,-0.3);


 \end{tikzpicture}

&
\begin{tikzpicture}[scale=0.8]


\draw  (-3,0)node{$0$};
\draw  (-1.5,1.15)node{$0$};
\draw  (0,0)node{$K$};
\draw  (1.5,1.15)node{$K$};
\draw  (-1.5,-1.15)node{$K$};
\draw  (1.5,-1.15)node{$0$};
\draw  (3,1.15)node{$0$};
\draw  (3,-0.5)node{$0$};

\draw[->][line width=0.5pt] (-1.65,1.1) -- (-2.85,0.1);

\draw[->][line width=0.5pt] (-1.35,1.15) -- (1.35,1.15);

\draw[->][line width=0.5pt] (1.35,1.05) -- (0.15,0.1)
node[pos=0.5,right] {$1$};	
	
\draw[->][line width=0.5pt] (-0.15,0.05) -- (-1.35,1);

\draw[->][line width=0.5pt] (-1.35,-1.05) -- (-0.15,-0.15)
node[pos=0.5,left] {$1$};	

\draw[->][line width=0.5pt] (0.15,-0.1) -- (1.35,-1.05);

\draw[->][line width=0.5pt] (2.85,1.15) -- (1.65,1.15);

\draw[->][line width=0.5pt] (3,1) -- (3,-0.3);


 \end{tikzpicture}




\end{tabular}

\caption{Arc representations  $M_{\alpha}$ and $M_{\beta}$ form Figure \ref{arcos 2}}\label{rep alpha beta 2}
\end{figure}

In this example there is one possible middle term as it is shown in Figure \ref{sec2}. And from Proposition \ref{prop:preserva-dimension} we know that it will be a middle term in a non-split extension.  

\vspace{-3mm}
\begin{figure}[H]
\centering				

\subfigure{\begin{tikzpicture}[scale=0.9]


\draw (0,0) circle (2cm);    

\filldraw [black] 	(0,0)  circle (2pt)
			(1.9,0.55)  circle (2pt)
			(-1.5,1.35)  circle (2pt)
			(-2,0)  circle (2pt)
			(0.85,-1.8)  circle (2pt)
			(-1.5,-1.35)  circle (2pt)

			(0,2)  circle (2pt)

			(1.3,1.5)  circle (2pt)
			(1.9,-0.65)  circle (2pt);

\draw[color=blue][line width=1pt] (-1.5,1.35) -- (-1.5,-1.35)
node[pos=0.25,left] {{\tiny1}};	
\draw[color=blue][line width=1pt] (0,2) .. controls(-1.15, 1) and (-1.25,-0.75) .. (-1.5,-1.35)
node[pos=0.25,left] {{\tiny2}};
\draw[color=blue][line width=1pt] (-1.5,-1.35)  .. controls(-0.35, 1.25) and (0.5,1.25) .. (0.85,-1.8)
node[pos=0.5,above] {{\tiny 3}};
\draw[color=blue][line width=1pt] (0,2) .. controls(1.15, 1) and (0.85,-0.75) .. (0.85,-1.8)
node[pos=0.25,right] {{\tiny6}};
\draw[color=blue][line width=1pt] (0,0) -- (0.85,-1.8)
node[pos=0.5,left] {{\tiny5}};
\draw[color=blue][line width=1pt] (0,0) -- (-1.5,-1.35)
node[pos=0.5,right] {{\tiny4}};

\draw[color=blue][line width=1pt] (0.85,-1.8) -- (1.3,1.5) 
node[pos=0.25,right] {{\tiny7}};

\draw[color=blue][line width=1pt]  (1.3,1.5)  -- (1.9,-0.65) 
node[pos=0.65,right] {{\tiny8}};


\draw[color=red][line width=1pt] (-2,0) --  (0.85,-1.8)
node[pos=0.65,left] {$\gamma_{1}$};

\draw  [color=red](0.18,0.2)node{\rotatebox{120}{$\bowtie$}};
\draw[color=red][line width=1pt] (0,0) --  (1.3,1.5) 
node[pos=0.5,left] {$\gamma_{2}$};

 \end{tikzpicture}
 		
}		

\caption{Arcs that define a middle term $M_{\gamma_1} \oplus M_{\gamma_2}$ in a non-split extension.}\label{terminos medios 2}
						
 \end{figure}
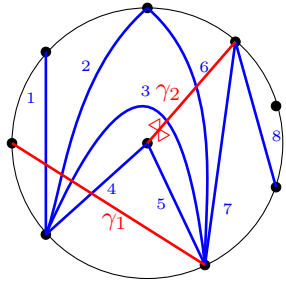

We compute the arc representation $M_{\gamma_{1}}\oplus M_{\gamma_{2}}$ from Figure \ref{terminos medios 2}. The arc representation $M_{\gamma_{1}}$ is computed easily, while to compute $M_{\gamma_{2}}$, the arc $\gamma_{2}$ is replaced by the curve $\gamma_{q_{0},q_{1}}\cup[q_{1},q]_{j^{o}}$ and the detour curve is drawn. Then the arc representation is the following;

\begin{figure}[H]
\centering				
\subfigure{\begin{tikzpicture}[scale=0.95]


\draw  (-3,0)node{$K$};
\draw  (-1.5,1.15)node{$K$};
\draw  (0,0)node{$K^{2}$};
\draw  (1.5,1.15)node{$K$};
\draw  (-1.5,-1.15)node{$K^{2}$};
\draw  (1.5,-1.15)node{$K$};
\draw  (3,1.15)node{$0$};
\draw  (3,-0.5)node{$0$};

\draw[->][line width=0.5pt] (-1.65,1.1) -- (-2.85,0.1)
node[pos=0.5,above] {$1$};	

\draw[->][line width=0.5pt] (-1.35,1.15) -- (1.35,1.15)
node[pos=0.5,above] {$0$};	

\draw[->][line width=0.5pt] (1.35,1.05) -- (0.23,0.18)
node[pos=0.25,below] {{\tiny $ \begin{bmatrix}
0\\
1
\end{bmatrix} $}};	
	
\draw[->][line width=0.5pt] (-0.15,0.05) -- (-1.35,1)
node[pos=0.5,left] {{\tiny $ \begin{bmatrix}
1&0
\end{bmatrix} $}};

\draw[->][line width=0.5pt] (-1.25,-1) -- (-0.2,-0.2)
node[pos=0.5,left] {$I_{2}$};	

\draw[->][line width=0.5pt] (0.15,-0.1) -- (1.35,-1.05)
node[pos=0.6,right] {{\tiny $ \begin{bmatrix}
1&0
\end{bmatrix} $}};

\draw[->][line width=0.5pt] (2.85,1.15) -- (1.65,1.15);

\draw[->][line width=0.5pt] (3,1) -- (3,-0.3);

 \end{tikzpicture}
 		
}		

\caption{Arc representation $M_{\gamma_{1}}\oplus M_{\gamma_{2}}$}\label{middle term 2}
						
 \end{figure}
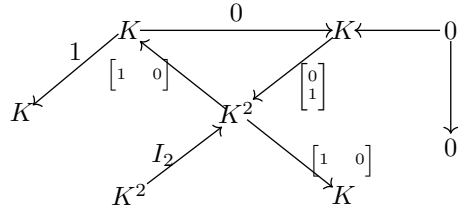	

The extension is shown next.

\begin{figure}[H]
\centering				

\subfigure{\begin{tikzpicture}[scale=1]

\draw  (-3.5,0)node{$0$};
\draw  (-2.5,0)node{$M_{\alpha}$};
\draw  (-0.5,0)node{$M_{\gamma_{1}}\oplus M_{\gamma_{2}}$};
\draw  (1.75,0)node{$M_{\beta}$};
\draw  (2.75,0)node{$0$};

\draw[->][line width=0.5pt] (-3.25,0) -- (-2.85,0);
\draw[->][line width=0.5pt] (-2.2,0) -- (-1.5,0);
\draw[->][line width=0.5pt] (0.4,0) -- (1.45,0);
\draw[->][line width=0.5pt] (2.05,0) -- (2.6,0);

\draw  (-1.95,0.25)node{$f$};
\draw  (-1.95,0.75)node{$||$};
\draw  (-1.75,2)node{{\tiny $ \left\{ \begin{array}{lc}
             1 &     \mbox{if $i=1,2,5$}\\
              &     \\
	   \begin{bmatrix}
1 \\
1
\end{bmatrix}  &  \mbox{if $i=3,4$}\\
              &     \\
	   0  &  \mbox{if $i=6,7,8$}
             \end{array}
   \right.$}};

\draw  (1.25,0.25)node{$g$};
\draw  (1.25,0.75)node{$||$};
\draw  (1.5,2)node{{\tiny $ \left\{ \begin{array}{lc}
             1 &     \mbox{if $i=6$}\\
              &     \\
	   \begin{bmatrix}
-1&1
\end{bmatrix}  &  \mbox{if $i=4,3$}\\
              &     \\
	   0  &  \mbox{if $i=1,2,5,7,8$}
             \end{array}
   \right.$}};

 \end{tikzpicture}
}		

\caption{Non-split extension from Figure \ref{terminos medios 2}}\label{exact sequence 2}
						
 \end{figure}	
\end{ex}

\subsection{Extensions in other surfaces}

In this section we apply the results in Section \ref{resultados-principales} to other surfaces.

The cluster category of a punctured surface $\mathcal{C}(S,M)$ was defined in \cite{QZ}. The authors establish a bijection between tagged curves and string objects in  $\mathcal{C}(S,M)$. Also in \cite[Section 5.2]{QZ} they show how Iyama-Yoshino reduction (see \cite[Section 1.3]{MP}) can be used in these categories. If we define a quiver with potential from a triangulation $\tau$ of $(S,M)$, such that $\gamma_0 \in \tau$, as we see in the Figure \ref{superficie-ejemplo3} then the category $\mathcal{C}(S,M)/\gamma_0$ (see \cite[Proposition 5]{MP}) is obtained as a cluster category from a quiver with potential $(Q(\tau),P(\tau))$ by removing the vertex corresponding to $\gamma_0$ in $\tau$ and also removing the terms containing arrows having that vertex as a source or target in $P(\tau)$.   


\begin{figure}[h!]
\scalebox{.75}{\centering
\def\svgwidth{1.9in}
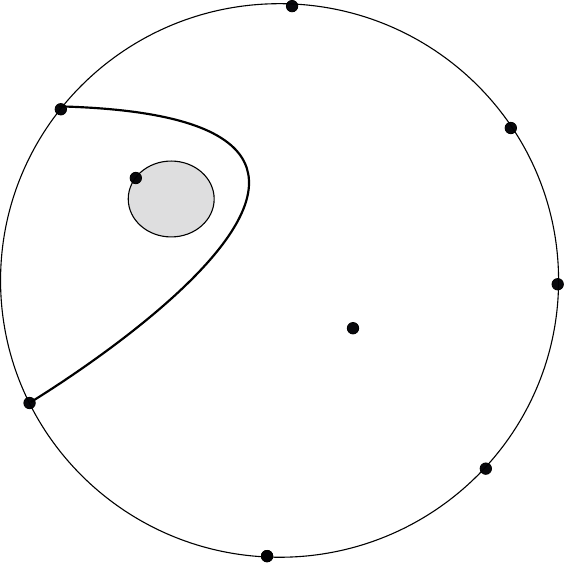}
\caption{Surface $(S,M)$ and arc $\gamma_0$ that splits the surface into two surfaces.}
\label{superficie-ejemplo3}
\end{figure}

Cutting along $\gamma_0$ splits the surface in Figure \ref{superficie-ejemplo3}. On one side we have an annulus of type $A_{2,1}$, and on the other side we have a punctured disk, where we see $\alpha$ and $\beta$ crossing. So the arcs $\alpha$ and $\beta$ lie on a disk and are such that $e(\alpha, \beta)  =2$. We can find a triangulation $\tau$ of $(S,M)$ containing $\gamma_0$ such that the portion of the triangulation on the disk is as in Figure \ref{Big}. Then we can define the associated quiver $Q(\tau)$ and potential $P(\tau)$ such that each one of the four multicurves in $\alpha^+ \beta$ correspond to middle terms of non-split extensions. These extensions lift to non-split triangles in $\mathcal{C}(S,M)$.

\begin{figure}[h!]
\begin{tabular}{ p{80mm}  p{70mm}}

\begin{tikzpicture}[scale=0.75]


\draw (0,0) circle (3.5cm);    
\draw[blue,fill=blue!25] (-1.5,0.8) circle (0.5cm);

\filldraw [black] 	(0,0)  circle (2pt)
			(0,3.5)  circle (2pt)
			(3.5,0)  circle (2pt)
			(-2.95,1.9)  circle (2pt)
			(2.45,2.5)  circle (2pt)
			(2.45,-2.5)  circle (2pt)
			(-2.85,-2)  circle (2pt)
			(-1.9,1.1)  circle (2pt)
			(0,-3.5)  circle (2pt);

\draw[color=blue][line width=1pt](-1.9,1.1) .. controls(-2.65, 0.35) and (-1.5,0) .. (0,0)
node[pos=0.25,below] {{\tiny1}};	
\draw[color=blue][line width=1pt](-1.9,1.1) .. controls(-1.7, 2.05) and (-0.85,1) .. (0,0)
node[pos=0.25,above] {{\tiny2}};
\draw[color=blue][line width=1pt] (0,0)  .. controls(-4.5, -1.75) and (-1.75,5) .. (0,0)
node[pos=0.5,left] {{\tiny 10}};
\draw[color=blue][line width=1pt] (-2.95,1.9)  .. controls(-3.5, -1.5) and (-1,-0.5) .. (0,0)
node[pos=0.25,below] {{\tiny8}};
\draw[color=blue][line width=1pt] (-2.95,1.9)  .. controls(-1.5, 3) and (-0.25,2.5) .. (0,0)
node[pos=0.45,above] {{\tiny4}};
\draw[color=blue][line width=1pt] (-2.95,1.9)  .. controls(-4.5, -2) and (0,-3.5) .. (0,-3.5)
node[pos=0.75,above] {{\tiny9}};
\draw[color=blue][line width=1pt] (-2.95,1.9)  .. controls(-1.7,3.5) and (1.7,3) .. (2.45,2.5)
node[pos=0.25,above] {{\tiny3}};
\draw[color=blue][line width=1pt] (0,0) -- (2.45,2.5)
node[pos=0.75,left] {{\tiny6}};
\draw[color=blue][line width=1pt] (0,0) -- (0,-3.5)
node[pos=0.75,left] {{\tiny7}};
\draw[color=blue][line width=1pt] (0,-3.5) -- (2.45,2.5)
node[pos=0.5,left] {{\tiny5}};
\draw[color=blue][line width=1pt] (0,-3.5) --(3.5,0)
node[pos=0.5,left] {{\tiny11}};

\draw[color=red][line width=1pt] (0,3.5) .. controls(-2,-2) and (0.75,-1) .. (3.5,0)
node[pos=0.5,left] {{\tiny $\alpha$}};

\draw[color=green][line width=1pt] (-2.85,-2) .. controls(0.5,2.5) and (1,1.25) ..(2.45,-2.5)
node[pos=0.5,above] {{\tiny $\beta$}};

\draw  (0,3.7)node{$p$};
\draw  (3.7,0)node{$n$};
\draw  (-3,-2.2)node{$s$};
\draw  (2.55,-2.7)node{$r$};

 \end{tikzpicture}

&

\begin{tikzpicture}[scale=0.78 ]


\draw  (-3.1,3.6)node{$2$};
\draw  (-3.1,1.8)node{$1$};
\draw  (-3.1,0)node{$10$};
\draw  (-0.5,1.15)node{$4$};
\draw  (2,0)node{$6$};
\draw  (1.4,1.5)node{$3$};
\draw  (-1.5,-1.15)node{$8$};
\draw  (0.5,-1.15)node{$7$};
\draw  (-0.5,-3)node{$9$};
\draw  (2.5,-1.65)node{$5$};
\draw  (2.5,-3)node{$11$};

\draw[->][line width=0.5pt] (-3.3,3.35) -- (-3.3,2.05)
node[pos=0.5,left] {$\alpha_{1}$};	
\draw[->][line width=0.5pt] (-2.9,3.35) -- (-2.9,2.05)
node[pos=0.5,right] {$\alpha_{17}$};
\draw[->][line width=0.5pt] (-3.1,1.55) -- (-3.1,0.25)
node[pos=0.5,left] {$\alpha_{2}$};	
\draw[->][line width=0.5pt] (-2.85,-0.15) -- (-1.7,-1.1)
node[pos=0.5,left] {$\alpha_{3}$};
\draw[->][line width=0.5pt] (-1.25,-1.15) -- (0.25,-1.15)
node[pos=0.5,below] {$\alpha_{4}$};	
\draw[->][line width=0.5pt] (0.75,-1.1)-- (1.8,-0.2)
node[pos=0.5,below] {$\alpha_{5}$};		
\draw[<-][line width=0.5pt] (1.95,0.25)-- (1.55,1.25)
node[pos=0.5,right] {$\alpha_{14}$};	
\draw[<-][line width=0.5pt]  (1.15,1.45)-- (-0.25,1.15)
node[pos=0.5,above] {$\alpha_{15}$};	
\draw[->][line width=0.5pt]   (-0.75,1.05)--  (-2.8,0)
node[pos=0.5,above] {$\alpha_{7}$};	
\draw[<-][line width=0.5pt]   (-1.4,-1.3)--  (-0.6,-2.75)
node[pos=0.5,left] {$\alpha_{9}$};	
\draw[->][line width=0.5pt]  (0.4,-1.35)--   (-0.3,-2.75)
node[pos=0.5,right] {$\alpha_{10}$};	
\draw[->][line width=0.5pt]  (2.05,-0.2)--   (2.45,-1.4)
node[pos=0.5,right] {$\alpha_{12}$};	
\draw[->][line width=0.5pt](2.25,-1.6)--(0.75,-1.25)
node[pos=0.5,below] {$\alpha_{11}$};	
\draw[->][line width=0.5pt] (2.5,-2.8)--(2.5,-1.85)
node[pos=0.5,right] {$\alpha_{13}$};	
\draw[->][line width=0.5pt] (-1.45,-0.95)--(-0.55,0.95)
node[pos=0.5,right] {$\alpha_{8}$};	
\draw[->][line width=0.5pt] (1.8,0.05)--(-0.4,1)
node[pos=0.5,right] {$\alpha_{6}$};	
\draw[->][line width=0.5pt] (-2.9,0.25) .. controls(-1.5,1) and (-1.5,2.5) ..   (-2.9,3.6)
node[pos=0.5,right] {$\alpha_{16}$};	
\end{tikzpicture}
\end{tabular}
\caption{Left: Surface with punctures with a triangulation $\widetilde{\tau}$. Right: quiver $Q(\widetilde{\tau})$.}\label{ejemplo3}
\end{figure}
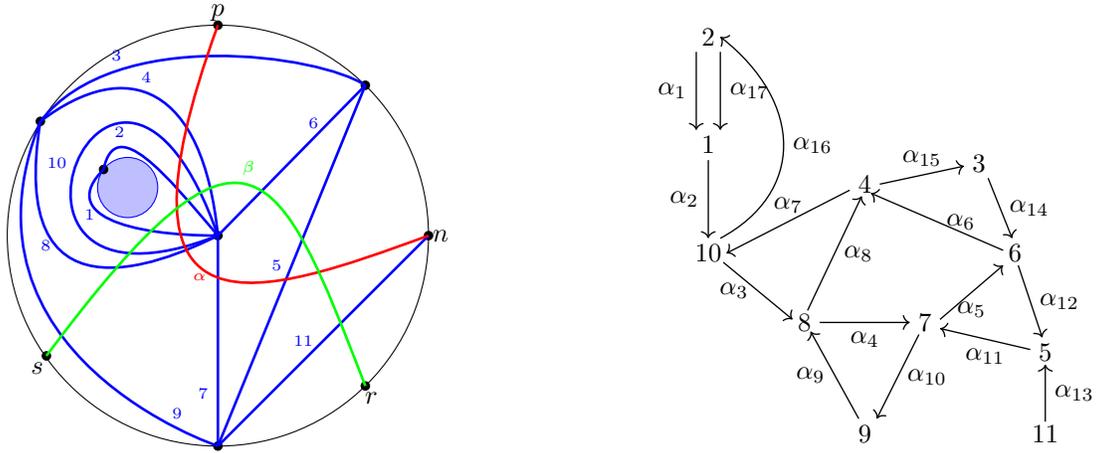

Let $\widetilde{\tau}$ be a triangulation of $(S,M)$ in Figure \ref{ejemplo3}. Let $(Q(\widetilde{\tau}), P(\widetilde{\tau}))$ the quiver with potential and $B = K Q(\widetilde{\tau}) / I$ the Jacobian algebra defined by $(Q(\widetilde{\tau}), P(\widetilde{\tau}))$. By Proposition \ref{prop:preserva-dimension}, for each multicurve $C \in \alpha^+ \beta$ such that the total dimension of $C$ and $\{\alpha, \beta \}$ coincide we can find a non-split sequence in $\mmod B$.

\begin{figure}[h!]
\scalebox{.85}{\centering
\def\svgwidth{5in}
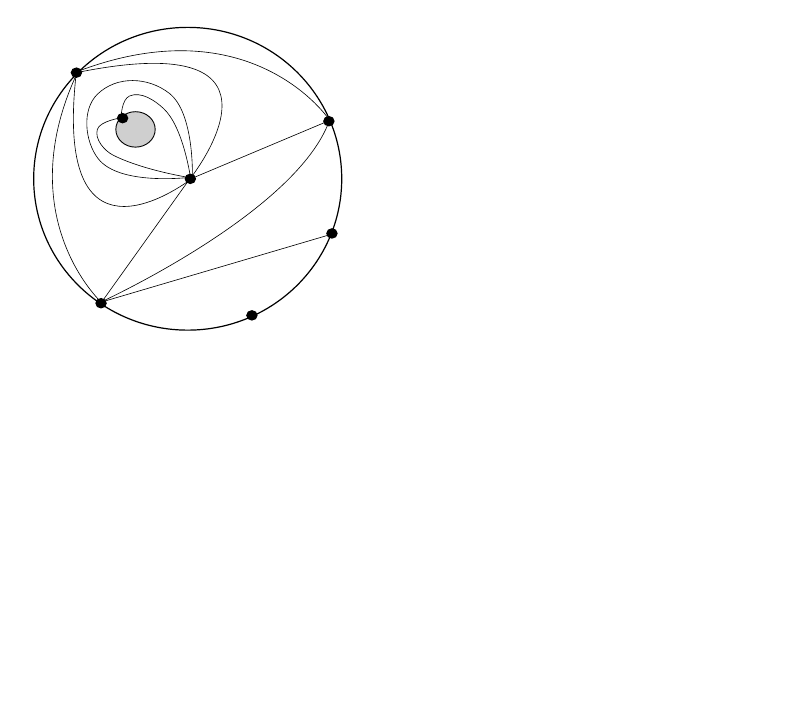}
\caption{Multicurves in $\alpha^+ \beta$. Label the multicurves (a) - (d) in lexicographic order. }\label{terminos medios 3}
\end{figure}

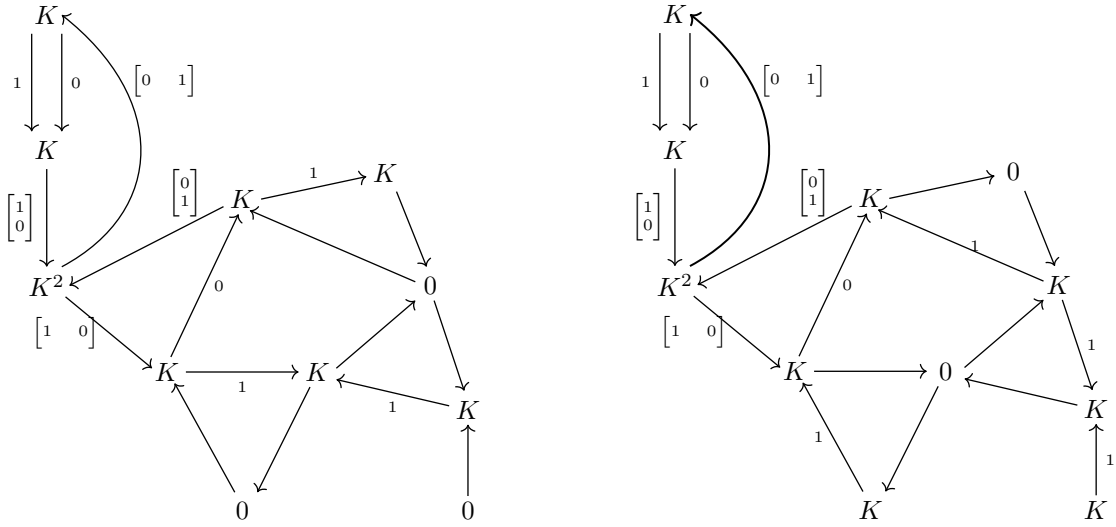
\begin{figure}[H]
\begin{tabular}{ p{80mm}  p{80mm}}
\begin{tikzpicture}[scale=.9 ]

\draw  (-3.1,3.6)node{$K$};
\draw  (-3.1,1.8)node{$K$};
\draw  (-3.1,0)node{$K^{2}$};
\draw  (-0.5,1.15)node{$K$};
\draw  (2,0)node{$0$};
\draw  (1.4,1.5)node{$K$};
\draw  (-1.5,-1.15)node{$K$};
\draw  (0.5,-1.15)node{$K$};
\draw  (-0.5,-3)node{$0$};
\draw  (2.5,-1.65)node{$K$};
\draw  (2.5,-3)node{$0$};

\draw[->][line width=0.5pt] (-3.3,3.35) -- (-3.3,2.05)
node[pos=0.5,left] {{\tiny $1$}};	
\draw[->][line width=0.5pt] (-2.9,3.35) -- (-2.9,2.05)
node[pos=0.5,right] {{\tiny $0$}};
\draw[->][line width=0.5pt] (-3.1,1.55) -- (-3.1,0.25)
node[pos=0.5,left] {{\tiny $ \begin{bmatrix}
1\\
0
\end{bmatrix} $}};	
\draw[->][line width=0.5pt] (-2.85,-0.15) -- (-1.7,-1.1)
node[pos=0.5,left] {{\tiny $ \begin{bmatrix}
1&0\\
\end{bmatrix} $}};	
\draw[->][line width=0.5pt] (-1.25,-1.15) -- (0.25,-1.15)
node[pos=0.5,below] {{\tiny $1$}};
\draw[->][line width=0.5pt] (0.75,-1.1)-- (1.8,-0.2);
\draw[<-][line width=0.5pt] (1.95,0.25)-- (1.55,1.25);
\draw[<-][line width=0.5pt]  (1.15,1.45)-- (-0.25,1.15)
node[pos=0.5,above] {{\tiny $1$}};	
\draw[->][line width=0.5pt]   (-0.75,1.05)--  (-2.8,0)
node[pos=0.25,above]  {{\tiny $ \begin{bmatrix}
0\\
1
\end{bmatrix} $}};	
\draw[<-][line width=0.5pt]   (-1.4,-1.3)--  (-0.6,-2.75);
\draw[->][line width=0.5pt]  (0.4,-1.35)--   (-0.3,-2.75);
\draw[->][line width=0.5pt]  (2.05,-0.2)--   (2.45,-1.4);
\draw[->][line width=0.5pt](2.25,-1.6)--(0.75,-1.25)
node[pos=0.5,below] {{\tiny $1$}};	
\draw[->][line width=0.5pt] (2.5,-2.8)--(2.5,-1.85);
\draw[->][line width=0.5pt] (-1.45,-0.95)--(-0.55,0.95)
node[pos=0.5,right] {{\tiny $0$}};	
\draw[->][line width=0.5pt] (1.8,0.05)--(-0.4,1);
\draw[->][line width=0.5pt] (-2.9,0.25) .. controls(-1.5,1) and (-1.5,2.5) ..   (-2.9,3.6)
node[pos=0.75,right] {{\tiny $ \begin{bmatrix}
0&1\\
\end{bmatrix} $}};	

 \end{tikzpicture}

&
\begin{tikzpicture}[scale=0.85]


\draw  (-3.1,3.6)node{$K$};
\draw  (-3.1,1.8)node{$K$};
\draw  (-3.1,0)node{$K^{2}$};
\draw  (-0.5,1.15)node{$K$};
\draw  (2,0)node{$K$};
\draw  (1.4,1.5)node{$0$};
\draw  (-1.5,-1.15)node{$K$};
\draw  (0.5,-1.15)node{$0$};
\draw  (-0.5,-3)node{$K$};
\draw  (2.5,-1.65)node{$K$};
\draw  (2.5,-3)node{$K$};

\draw[->][line width=0.5pt] (-3.3,3.35) -- (-3.3,2.05)
node[pos=0.5,left] {{\tiny $1$}};	
\draw[->][line width=0.5pt] (-2.9,3.35) -- (-2.9,2.05)
node[pos=0.5,right] {{\tiny $0$}};
\draw[->][line width=0.5pt] (-3.1,1.55) -- (-3.1,0.25)
node[pos=0.5,left] {{\tiny $ \begin{bmatrix}
1\\
0
\end{bmatrix} $}};		
\draw[->][line width=0.5pt] (-2.85,-0.15) -- (-1.7,-1.1)
node[pos=0.5,left] {{\tiny $ \begin{bmatrix}
1&0\\
\end{bmatrix} $}};	
\draw[->][line width=0.5pt] (-1.25,-1.15) -- (0.25,-1.15);	
\draw[->][line width=0.5pt] (0.75,-1.1)-- (1.8,-0.2);
\draw[<-][line width=0.5pt] (1.95,0.25)-- (1.55,1.25);
\draw[<-][line width=0.5pt]  (1.15,1.45)-- (-0.25,1.15);
\draw[->][line width=0.5pt]   (-0.75,1.05)--  (-2.8,0)
node[pos=0.25,above] {{\tiny $ \begin{bmatrix}
0\\
1
\end{bmatrix} $}};		
\draw[<-][line width=0.5pt]   (-1.4,-1.3)--  (-0.6,-2.75)
node[pos=0.5,left] {{\tiny $1$}};	
\draw[->][line width=0.5pt]  (0.4,-1.35)--   (-0.3,-2.75);
\draw[->][line width=0.5pt]  (2.05,-0.2)--   (2.45,-1.4)
node[pos=0.5,right] {{\tiny $1$}};
\draw[->][line width=0.5pt](2.25,-1.6)--(0.75,-1.25);
\draw[->][line width=0.5pt] (2.5,-2.8)--(2.5,-1.85)
node[pos=0.5,right] {{\tiny $1$}};	
\draw[->][line width=0.5pt] (-1.45,-0.95)--(-0.55,0.95)
node[pos=0.5,right] {{\tiny $0$}};	
\draw[->][line width=0.5pt] (1.8,0.05)--(-0.4,1)
node[pos=0.5,right] {{\tiny $1$}};	
\draw[->][line width=0.75pt] (-2.9,0.25) .. controls(-1.5,1) and (-1.5,2.5) ..   (-2.9,3.6)
node[pos=0.75,right] {{\tiny $ \begin{bmatrix}
0&1\\
\end{bmatrix} $}};	
 \end{tikzpicture}

\end{tabular}

\caption{Arc representations  $M_{\alpha}$ (left) and $M_{\beta}$ (right).}\label{rep alpha beta 3}
\end{figure}

The possible middle terms are given by the next multicurves in Figure \ref{terminos medios 3}. Observe that each multicurve $C$ satisfies $d(C) = d(\{\alpha, \beta \})$. To compute $M_{\gamma_{1}}\oplus M_{\gamma_{2}}\oplus M_{\gamma_{3}}$ in  item (a) of Figure \ref{terminos medios 3}, the arc representations $M_{\gamma_{1}}$ and $M_{\gamma_{3}}$ are computed easily, while to compute $M_{\gamma_{2}}$, the arc $\gamma_{2}$ is replaced by the curve $\gamma_{q_{0},q_{1}}\cup[q_{1},n]_{j^{o}}$ and the detour curve is drawn. The arc representation in item (b) can be found in a similar way.

\begin{figure}[b]
 \centering
\subfigure[$M_{\gamma_{1}}\oplus M_{\gamma_{2}}\oplus M_{\gamma_{3}}$ Figure \ref{terminos medios 3}. ]{
\begin{tikzpicture}[scale=0.9]

\draw  (-3.1,3.6)node{$K^{2}$};
\draw  (-3.1,1.8)node{$K^{2}$};
\draw  (-3.1,0)node{$K^{4}$};
\draw  (-0.5,1.15)node{$K^{2}$};
\draw  (2,0)node{$K$};
\draw  (1.4,1.5)node{$K$};
\draw  (-1.5,-1.15)node{$K^{2}$};
\draw  (0.5,-1.15)node{$K$};
\draw  (-0.5,-3)node{$K$};
\draw  (2.5,-1.65)node{$K^{2}$};
\draw  (2.5,-3)node{$K$};

\draw[->][line width=0.5pt] (-3.3,3.35) -- (-3.3,2.05)
node[pos=0.5,left] {$I_{2}$};	
\draw[->][line width=0.5pt] (-2.9,3.35) -- (-2.9,2.05)
node[pos=0.5,right] {{\tiny $0$}};
\draw[->][line width=0.5pt] (-3.1,1.55) -- (-3.1,0.25)
node[pos=0.5,left] {{\tiny $ \begin{bmatrix}
1&0\\
0&0\\
0&1\\
0&0
\end{bmatrix} $}};	
\draw[->][line width=0.5pt] (-2.85,-0.15) -- (-1.7,-1.1)
node[pos=0.35,left] {{\tiny $ \begin{bmatrix}
1&0&0&0\\
0&0&1&0
\end{bmatrix} $}};
\draw[->][line width=0.5pt] (-1.25,-1.15) -- (0.25,-1.15)
node[pos=0.5,below] {{\tiny $ \begin{bmatrix}
0&1\\
\end{bmatrix} $}};	
\draw[->][line width=0.5pt] (0.75,-1.1)-- (1.8,-0.2)
node[pos=0.5,below] {{\tiny $0$}};	
\draw[<-][line width=0.5pt] (1.95,0.25)-- (1.55,1.25)
node[pos=0.5,right] {{\tiny $0$}};
\draw[<-][line width=0.5pt]  (1.15,1.45)-- (-0.25,1.15)
node[pos=0.5,above] {{\tiny $ \begin{bmatrix}
1&0\\
\end{bmatrix} $}};		
\draw[->][line width=0.5pt]   (-0.75,1.05)--  (-2.8,0)
node[pos=0.25,above]{{\tiny $ \begin{bmatrix}
0&0\\
1&0\\
0&0\\
0&1
\end{bmatrix} $}};	
\draw[<-][line width=0.5pt]   (-1.4,-1.3)--  (-0.6,-2.75)
node[pos=0.5,left] {{\tiny $ \begin{bmatrix}
1\\
0
\end{bmatrix} $}};		
\draw[->][line width=0.5pt]  (0.4,-1.35)--   (-0.3,-2.75)
node[pos=0.5,right] {{\tiny $0$}};
\draw[->][line width=0.5pt]  (2.05,-0.2)--   (2.45,-1.4)
node[pos=0.5,right] {{\tiny $ \begin{bmatrix}
1\\
0
\end{bmatrix} $}};		
\draw[->][line width=0.5pt](2.25,-1.6)--(0.75,-1.25)
node[pos=0.5,below] {{\tiny $ \begin{bmatrix}
1&0\\
\end{bmatrix} $}};		
\draw[->][line width=0.5pt] (2.5,-2.8)--(2.5,-1.85)
node[pos=0.5,right] {{\tiny $ \begin{bmatrix}
0\\
1
\end{bmatrix} $}};		
\draw[->][line width=0.5pt] (-1.45,-0.95)--(-0.55,0.95)
node[pos=0.5,right] {{\tiny $0$}};
\draw[->][line width=0.5pt] (1.8,0.05)--(-0.4,1)
node[pos=0.5,below] {{\tiny $ \begin{bmatrix}
0\\
1
\end{bmatrix} $}};		
\draw[->][line width=0.5pt] (-2.9,0.25) .. controls(-1.5,1) and (-1.5,2.5) ..   (-2.8,3.6)
node[pos=0.85,right] {{\tiny $ \begin{bmatrix}
0&1&0&0\\
0&0&0&1
\end{bmatrix} $}};	

 \end{tikzpicture}

}
\end{figure}

\begin{figure}[H]
\begin{tabular}{ p{90mm}  p{70mm}}
\begin{tikzpicture}[scale=0.9]

\draw  (-3.1,3.6)node{$K^{2}$};
\draw  (-3.1,1.8)node{$K^{2}$};
\draw  (-3.1,0)node{$K^{4}$};
\draw  (-0.5,1.15)node{$K^{2}$};
\draw  (2,0)node{$K$};
\draw  (1.4,1.5)node{$K$};
\draw  (-1.5,-1.15)node{$K^{2}$};
\draw  (0.5,-1.15)node{$K$};
\draw  (-0.5,-3)node{$0$};
\draw  (2.5,-1.65)node{$K^{2}$};
\draw  (2.5,-3)node{$K$};

\draw[->][line width=0.5pt] (-3.3,3.35) -- (-3.3,2.05)
node[pos=0.5,left] {$I_{2}$};	
\draw[->][line width=0.5pt] (-2.9,3.35) -- (-2.9,2.05)
node[pos=0.5,right] {{\tiny $0$}};
\draw[->][line width=0.5pt] (-3.1,1.55) -- (-3.1,0.25)
node[pos=0.5,left] {{\tiny $ \begin{bmatrix}
1&0\\
0&0\\
0&1\\
0&0
\end{bmatrix} $}};		
\draw[->][line width=0.5pt] (-2.85,-0.15) -- (-1.7,-1.1)
node[pos=0.33,left] {{\tiny $ \begin{bmatrix}
1&0&0&0\\
0&0&1&0
\end{bmatrix} $}};
\draw[->][line width=0.5pt] (-1.25,-1.15) -- (0.25,-1.15)
node[pos=0.5,below] {{\tiny $ \begin{bmatrix}
0&1\\
\end{bmatrix} $}};
\draw[->][line width=0.5pt] (0.75,-1.1)-- (1.8,-0.2)
node[pos=0.5,below] {{\tiny $0$}};	
\draw[<-][line width=0.5pt] (1.95,0.25)-- (1.55,1.25)
node[pos=0.5,right] {{\tiny $0$}};	
\draw[<-][line width=0.5pt]  (1.15,1.45)-- (-0.25,1.15)
node[pos=0.5,above] {{\tiny $ \begin{bmatrix}
0&1\\
\end{bmatrix} $}};	
\draw[->][line width=0.5pt]   (-0.75,1.05)--  (-2.8,0)
node[pos=0.25,above] {{\tiny $ \begin{bmatrix}
0&0\\
1&0\\
0&0\\
0&1
\end{bmatrix} $}};	
\draw[<-][line width=0.5pt]   (-1.4,-1.3)--  (-0.6,-2.75);
\draw[->][line width=0.5pt]  (0.4,-1.35)--   (-0.3,-2.75);
\draw[->][line width=0.5pt]  (2.05,-0.2)--   (2.45,-1.4)
node[pos=0.5,right] {{\tiny $ \begin{bmatrix}
1\\
0
\end{bmatrix} $}};	
\draw[->][line width=0.5pt](2.25,-1.6)--(0.75,-1.25)
node[pos=0.5,below] {{\tiny $ \begin{bmatrix}
0&1\\
\end{bmatrix} $}};	
\draw[->][line width=0.5pt] (2.5,-2.8)--(2.5,-1.85)
node[pos=0.5,right] {{\tiny $ \begin{bmatrix}
0\\
1
\end{bmatrix} $}};
\draw[->][line width=0.5pt] (-1.45,-0.95)--(-0.55,0.95)
node[pos=0.5,right] {{\tiny $0$}};
\draw[->][line width=0.5pt] (1.8,0.05)--(-0.4,1)
node[pos=0.5,below] {{\tiny $ \begin{bmatrix}
1\\
0
\end{bmatrix} $}};	
\draw[->][line width=0.5pt] (-2.9,0.25) .. controls(-1.5,1) and (-1.5,2.5) ..   (-2.8,3.6)
node[pos=0.85,right]{{\tiny $ \begin{bmatrix}
0&1&0&0\\
0&0&0&1
\end{bmatrix} $}};	

 \end{tikzpicture}

&
\begin{tikzpicture}[scale=0.9]


\draw  (-3.1,3.6)node{$K^{2}$};
\draw  (-3.1,1.8)node{$K^{2}$};
\draw  (-3.1,0)node{$K^{4}$};
\draw  (-0.5,1.15)node{$K^{2}$};
\draw  (2,0)node{$K$};
\draw  (1.4,1.5)node{$K$};
\draw  (-1.5,-1.15)node{$K^{2}$};
\draw  (0.5,-1.15)node{$K$};
\draw  (-0.5,-3)node{$K$};
\draw  (2.5,-1.65)node{$K^{2}$};
\draw  (2.5,-3)node{$K$};

\draw[->][line width=0.5pt] (-3.3,3.35) -- (-3.3,2.05)
node[pos=0.5,left] {$I_{2}$};	
\draw[->][line width=0.5pt] (-2.9,3.35) -- (-2.9,2.05)
node[pos=0.5,right] {{\tiny $0$}};
\draw[->][line width=0.5pt] (-3.1,1.55) -- (-3.1,0.25)
node[pos=0.5,left] {{\tiny $ \begin{bmatrix}
1&0\\
0&0\\
0&1\\
0&0
\end{bmatrix} $}};		
\draw[->][line width=0.5pt] (-2.85,-0.15) -- (-1.7,-1.1)
node[pos=0.33,left] {{\tiny $ \begin{bmatrix}
1&0&0&0\\
0&0&1&0
\end{bmatrix} $}};
\draw[->][line width=0.5pt] (-1.25,-1.15) -- (0.25,-1.15)
node[pos=0.5,below] {{\tiny $ \begin{bmatrix}
0&1\\
\end{bmatrix} $}};
\draw[->][line width=0.5pt] (0.75,-1.1)-- (1.8,-0.2)
node[pos=0.5,below] {{\tiny $0$}};	
\draw[<-][line width=0.5pt] (1.95,0.25)-- (1.55,1.25)
node[pos=0.5,right] {{\tiny $0$}};	
\draw[<-][line width=0.5pt]  (1.15,1.45)-- (-0.25,1.15)
node[pos=0.5,above] {{\tiny $ \begin{bmatrix}
1&0\\
\end{bmatrix} $}};	
\draw[->][line width=0.5pt]   (-0.75,1.05)--  (-2.8,0)
node[pos=0.25,above] {{\tiny $ \begin{bmatrix}
0&0\\
1&0\\
0&0\\
0&1
\end{bmatrix} $}};	
\draw[<-][line width=0.5pt]   (-1.4,-1.3)--  (-0.6,-2.75)
node[pos=0.5,left]  {{\tiny $ \begin{bmatrix}
1\\
0
\end{bmatrix} $}};		
\draw[->][line width=0.5pt]  (0.4,-1.35)--   (-0.3,-2.75)
node[pos=0.5,right] {{\tiny $0$}};	
\draw[->][line width=0.5pt]  (2.05,-0.2)--   (2.45,-1.4)
node[pos=0.5,right] {{\tiny $ \begin{bmatrix}
1\\
1
\end{bmatrix} $}};	
\draw[->][line width=0.5pt](2.25,-1.6)--(0.75,-1.25)
node[pos=0.5,below] {{\tiny $ \begin{bmatrix}
0&1\\
\end{bmatrix} $}};	
\draw[->][line width=0.5pt] (2.5,-2.8)--(2.5,-1.85)
node[pos=0.5,right] {{\tiny $ \begin{bmatrix}
0\\
1
\end{bmatrix} $}};
\draw[->][line width=0.5pt] (-1.45,-0.95)--(-0.55,0.95)
node[pos=0.5,right] {{\tiny $0$}};
\draw[->][line width=0.5pt] (1.8,0.05)--(-0.4,1)
node[pos=0.5,below] {{\tiny $ \begin{bmatrix}
0\\
1
\end{bmatrix} $}};	
\draw[->][line width=0.5pt] (-2.9,0.25) .. controls(-1.5,1) and (-1.5,2.5) ..   (-2.8,3.6)
node[pos=0.85,right]{{\tiny $ \begin{bmatrix}
0&1&0&0\\
0&0&0&1
\end{bmatrix} $}};	

 \end{tikzpicture}

\\[2mm]

 \ \  \ \  \  \  \ \  \  \    (c) $M_{\gamma_{1}}\oplus M_{\gamma_{2}}$ Figure \ref{terminos medios 3}. & \  \  \ \  \  \  \ \  \  \  (d) $M_{\gamma_{1}}\oplus M_{\gamma_{2}}$ Figure \ref{terminos medios 3}.

\\[2mm]

\end{tabular}

\caption{Middle terms of non-split extensions.}\label{arc rep 3}
\end{figure}

The non-split extensions are the following;

\begin{figure}[H]
\centering				

\subfigure{\begin{tikzpicture}[scale=1]

\draw  (-3.5,0)node{$0$};
\draw  (-2.5,0)node{$M_{\alpha}$};
\draw  (-0.5,0)node{{\tiny $M_{\gamma_{1}}\oplus M_{\gamma_{2}}\oplus M_{\gamma_{3}}$}};
\draw  (1.75,0)node{$M_{\beta}$};
\draw  (2.75,0)node{$0$};

\draw[->][line width=0.5pt] (-3.25,0) -- (-2.85,0);
\draw[->][line width=0.5pt] (-2.2,0) -- (-1.75,0);
\draw[->][line width=0.5pt] (0.6,0) -- (1.45,0);
\draw[->][line width=0.5pt] (2.05,0) -- (2.6,0);

\draw  (-1.95,0.25)node{$f$};
\draw  (-1.95,0.75)node{$||$};
\draw  (-2.5,2.5)node{{\tiny $ \left\{ \begin{array}{lc}
             1 &     \mbox{if $i=3,7$}\\
             0  &  \mbox{if $i=6,9,11$}\\
              &     \\
	   \begin{bmatrix}
1 \\
1
\end{bmatrix}  &  \mbox{if $i=1,2,4,5,8$}\\
              &     \\
	    \begin{bmatrix}
1 &0\\
0&1\\
1&0\\
0&1
\end{bmatrix}  &  \mbox{if $i=10$}
             \end{array}
   \right.$}};

\draw  (1.25,0.25)node{$g$};
\draw  (1.25,0.75)node{$||$};
\draw  (1.75,2.5)node{{\tiny $ \left\{ \begin{array}{lc}
             1 &     \mbox{if $i=9,11$}\\
	 -1 &     \mbox{if $i=6$}\\
              &     \\
	   0  &  \mbox{if $i=7,3$}\\
              &     \\
	   \begin{bmatrix}
-1&1
\end{bmatrix}  &  \mbox{if $i=5$}\\
	 &     \\
	   \begin{bmatrix}
1&-1
\end{bmatrix}  &  \mbox{if $i=1,2,4,8$}\\
 &     \\
	   \begin{bmatrix}
1&0&-1&0\\
0&1&0&-1
\end{bmatrix}  &  \mbox{if $i=10$}
             \end{array}
   \right.$}};

	 \end{tikzpicture}
 		
}							


\end{figure}

\begin{figure}[H]
\centering				

\subfigure{\begin{tikzpicture}[scale=1]

\draw  (-3.5,0)node{$0$};
\draw  (-2.5,0)node{$M_{\alpha}$};
\draw  (-0.5,0)node{ $M_{\gamma_{1}}\oplus M_{\gamma_{2}}$};
\draw  (1.75,0)node{$M_{\beta}$};
\draw  (2.75,0)node{$0$};

\draw[->][line width=0.5pt] (-3.25,0) -- (-2.85,0);
\draw[->][line width=0.5pt] (-2.2,0) -- (-1.5,0);
\draw[->][line width=0.5pt] (0.4,0) -- (1.45,0);
\draw[->][line width=0.5pt] (2.05,0) -- (2.6,0);

\draw  (-1.95,0.25)node{$f$};
\draw  (-1.95,0.75)node{$||$};
\draw  (-2.5,2.5)node{{\tiny $ \left\{ \begin{array}{lc}
             1 &     \mbox{if $i=3,7$}\\
             0  &  \mbox{if $i=6,9,11$}\\
              &     \\
	   \begin{bmatrix}
1 \\
1
\end{bmatrix}  &  \mbox{if $i=1,2,4,5,8$}\\
              &     \\
	    \begin{bmatrix}
1 &0\\
0&1\\
1&0\\
0&1
\end{bmatrix}  &  \mbox{if $i=10$}
             \end{array}
   \right.$}};

\draw  (1.25,0.25)node{$g$};
\draw  (1.25,0.75)node{$||$};
\draw  (1.75,2.5)node{{\tiny $ \left\{ \begin{array}{lc}
             1 &     \mbox{if $i=6$}\\
	 -1 &     \mbox{if $i=11$}\\
              &     \\
	   0  &  \mbox{if $i=3,7,9$}\\
             	 &     \\
	   \begin{bmatrix}
1&-1
\end{bmatrix}  &  \mbox{if $i=1,2,4,5,8$}\\
 &     \\
	   \begin{bmatrix}
1&0&-1&0\\
0&1&0&-1
\end{bmatrix}  &  \mbox{if $i=10$}
             \end{array}
   \right.$}};

	 \end{tikzpicture}
 		
}							

\subfigure{\begin{tikzpicture}[scale=1]

\draw  (-3.5,0)node{$0$};
\draw  (-2.5,0)node{$M_{\alpha}$};
\draw  (-0.5,0)node{$M_{\gamma_{1}}\oplus M_{\gamma_{2}}$};
\draw  (1.75,0)node{$M_{\beta}$};
\draw  (2.75,0)node{$0$};

\draw[->][line width=0.5pt] (-3.25,0) -- (-2.85,0);
\draw[->][line width=0.5pt] (-2.2,0) -- (-1.5,0);
\draw[->][line width=0.5pt] (0.4,0) -- (1.45,0);
\draw[->][line width=0.5pt] (2.05,0) -- (2.6,0);

\draw  (-1.95,0.25)node{$f$};
\draw  (-1.95,0.75)node{$||$};
\draw  (-2.5,3)node{{\tiny $ \left\{ \begin{array}{lc}
             1 &     \mbox{if $i=3,7$}\\
             0  &  \mbox{if $i=6,9,11$}\\
              &     \\
	   \begin{bmatrix}
1 \\
1
\end{bmatrix}  &  \mbox{if $i=1,2,4,8$}\\
	&     \\
	   \begin{bmatrix}
2 \\
1
\end{bmatrix}  &  \mbox{if $i=5$}\\
              &     \\
	    \begin{bmatrix}
1 &0\\
0&1\\
1&0\\
0&1
\end{bmatrix}  &  \mbox{if $i=10$}
             \end{array}
   \right.$}};

\draw  (1.25,0.25)node{$g$};
\draw  (1.25,0.75)node{$||$};
\draw  (1.75,2.75)node{{\tiny $ \left\{ \begin{array}{lc}
             1 &     \mbox{if $i=9$}\\
	 -1 &     \mbox{if $i=6$}\\
	 -2&     \mbox{if $i=11$}\\
              &     \\
	   0  &  \mbox{if $i=3,7$}\\
              &     \\
	   \begin{bmatrix}
1&-1
\end{bmatrix}  &  \mbox{if $i=1,2,4,8$}\\
	 &     \\
	   \begin{bmatrix}
1&-2
\end{bmatrix}  &  \mbox{if $i=5$}\\
 &     \\
	   \begin{bmatrix}
1&0&-1&0\\
0&1&0&-1
\end{bmatrix}  &  \mbox{if $i=10$}
             \end{array}
   \right.$}};

	 \end{tikzpicture}
 		
}

\caption{Non-split extensions obtained via punctured skein relations.}\label{sucesiones exactas 3}

\end{figure}

\bibliographystyle{alpha}
\bibliography{biblio}

\end{document}

%% file: diamante2.pdf_tex
\begingroup%
  \makeatletter%
  \providecommand\color[2][]{%
    \errmessage{(Inkscape) Color is used for the text in Inkscape, but the package 'color.sty' is not loaded}%
    \renewcommand\color[2][]{}%
  }%
  \providecommand\transparent[1]{%
    \errmessage{(Inkscape) Transparency is used (non-zero) for the text in Inkscape, but the package 'transparent.sty' is not loaded}%
    \renewcommand\transparent[1]{}%
  }%
  \providecommand\rotatebox[2]{#2}%
  \newcommand*\fsize{\dimexpr\f@size pt\relax}%
  \newcommand*\lineheight[1]{\fontsize{\fsize}{#1\fsize}\selectfont}%
  \ifx\svgwidth\undefined%
    \setlength{\unitlength}{304.89752458bp}%
    \ifx\svgscale\undefined%
      \relax%
    \else%
      \setlength{\unitlength}{\unitlength * \real{\svgscale}}%
    \fi%
  \else%
    \setlength{\unitlength}{\svgwidth}%
  \fi%
  \global\let\svgwidth\undefined%
  \global\let\svgscale\undefined%
  \makeatother%
  \begin{picture}(1,0.32523727)%
    \lineheight{1}%
    \setlength\tabcolsep{0pt}%
    \put(0,0){\includegraphics[width=\unitlength,page=1]{diamante2.pdf}}%
    \put(0.00648397,0.17681319){\color[rgb]{0,0,0}\makebox(0,0)[lt]{\lineheight{1.25}\smash{\begin{tabular}[t]{l}$M_i$\end{tabular}}}}%
    \put(0.40659086,0.1829515){\color[rgb]{0,0,0}\makebox(0,0)[lt]{\lineheight{1.25}\smash{\begin{tabular}[t]{l}$M_{\beta}$\end{tabular}}}}%
    \put(0,0){\includegraphics[width=\unitlength,page=2]{diamante2.pdf}}%
    \put(0.07654193,0.12597736){\color[rgb]{0,0,0}\makebox(0,0)[lt]{\lineheight{1.25}\smash{\begin{tabular}[t]{l}$X_1$\end{tabular}}}}%
    \put(0.35079448,0.13786233){\color[rgb]{0,0,0}\makebox(0,0)[lt]{\lineheight{1.25}\smash{\begin{tabular}[t]{l}$X_2$\end{tabular}}}}%
    \put(0.18792524,0.03573027){\color[rgb]{0,0,0}\makebox(0,0)[lt]{\lineheight{1.25}\smash{\begin{tabular}[t]{l}$X_3$\end{tabular}}}}%
    \put(0.17136889,0.30963428){\color[rgb]{0,0,0}\makebox(0,0)[lt]{\lineheight{1.25}\smash{\begin{tabular}[t]{l}$A_1$\end{tabular}}}}%
    \put(0.17074772,0.2744031){\color[rgb]{0,0,0}\makebox(0,0)[lt]{\lineheight{1.25}\smash{\begin{tabular}[t]{l}$A_2$\end{tabular}}}}%
    \put(0.2936962,0.30922931){\color[rgb]{0,0,0}\makebox(0,0)[lt]{\lineheight{1.25}\smash{\begin{tabular}[t]{l}$B_1$\end{tabular}}}}%
    \put(0.29064011,0.27396275){\color[rgb]{0,0,0}\makebox(0,0)[lt]{\lineheight{1.25}\smash{\begin{tabular}[t]{l}$B_2$\end{tabular}}}}%
    \put(0,0){\includegraphics[width=\unitlength,page=3]{diamante2.pdf}}%
    \put(0.23257985,0.23661293){\color[rgb]{0,0,0}\makebox(0,0)[lt]{\lineheight{1.25}\smash{\begin{tabular}[t]{l}$X$\end{tabular}}}}%
    \put(0,0){\includegraphics[width=\unitlength,page=4]{diamante2.pdf}}%
    \put(0.73803528,0.03123855){\color[rgb]{0,0,0}\makebox(0,0)[lt]{\lineheight{1.25}\smash{\begin{tabular}[t]{l}$i$\end{tabular}}}}%
    \put(0.72146259,0.01104333){\color[rgb]{0,0,0}\makebox(0,0)[lt]{\lineheight{1.25}\smash{\begin{tabular}[t]{l}$s+1$\end{tabular}}}}%
    \put(0,0){\includegraphics[width=\unitlength,page=5]{diamante2.pdf}}%
    \put(0.59046583,0.12287503){\color[rgb]{0,0,0}\makebox(0,0)[lt]{\lineheight{1.25}\smash{\begin{tabular}[t]{l}$i$\end{tabular}}}}%
    \put(0.57871331,0.10213367){\color[rgb]{0,0,0}\makebox(0,0)[lt]{\lineheight{1.25}\smash{\begin{tabular}[t]{l}$r+1$\end{tabular}}}}%
    \put(0,0){\includegraphics[width=\unitlength,page=6]{diamante2.pdf}}%
    \put(0.86173968,0.10861963){\color[rgb]{0,0,0}\makebox(0,0)[lt]{\lineheight{1.25}\smash{\begin{tabular}[t]{l}$n-1,n$\end{tabular}}}}%
    \put(0.85868378,0.08842447){\color[rgb]{0,0,0}\makebox(0,0)[lt]{\lineheight{1.25}\smash{\begin{tabular}[t]{l}$1, s+1$\end{tabular}}}}%
    \put(-0.00169114,0.30411462){\color[rgb]{0,0,0}\makebox(0,0)[lt]{\lineheight{1.25}\smash{\begin{tabular}[t]{l}(a)\end{tabular}}}}%
    \put(0.5197953,0.30903431){\color[rgb]{0,0,0}\makebox(0,0)[lt]{\lineheight{1.25}\smash{\begin{tabular}[t]{l}(b)\end{tabular}}}}%
    \put(0,0){\includegraphics[width=\unitlength,page=7]{diamante2.pdf}}%
    \put(0.71955512,0.20689122){\color[rgb]{0,0,0}\makebox(0,0)[lt]{\lineheight{1.25}\smash{\begin{tabular}[t]{l}$n-1,n$\end{tabular}}}}%
    \put(0.71649893,0.18432014){\color[rgb]{0,0,0}\makebox(0,0)[lt]{\lineheight{1.25}\smash{\begin{tabular}[t]{l}$1,r+1$\end{tabular}}}}%
    \put(0,0){\includegraphics[width=\unitlength,page=8]{diamante2.pdf}}%
    \put(0.67980714,0.253869){\color[rgb]{0,0,0}\makebox(0,0)[lt]{\lineheight{1.25}\smash{\begin{tabular}[t]{l}$n$\end{tabular}}}}%
    \put(0.67985704,0.22951598){\color[rgb]{0,0,0}\makebox(0,0)[lt]{\lineheight{1.25}\smash{\begin{tabular}[t]{l}$1$\end{tabular}}}}%
    \put(0.54106329,0.17243466){\color[rgb]{0,0,0}\makebox(0,0)[lt]{\lineheight{1.25}\smash{\begin{tabular}[t]{l}$i$\end{tabular}}}}%
    \put(0.53924939,0.15397368){\color[rgb]{0,0,0}\makebox(0,0)[lt]{\lineheight{1.25}\smash{\begin{tabular}[t]{l}$1$\end{tabular}}}}%
    \put(0.66788234,0.31094137){\color[rgb]{0,0,0}\makebox(0,0)[lt]{\lineheight{1.25}\smash{\begin{tabular}[t]{l}$n-1$\end{tabular}}}}%
    \put(0.68207026,0.29193413){\color[rgb]{0,0,0}\makebox(0,0)[lt]{\lineheight{1.25}\smash{\begin{tabular}[t]{l}$1$\end{tabular}}}}%
    \put(0.80563357,0.31094137){\color[rgb]{0,0,0}\makebox(0,0)[lt]{\lineheight{1.25}\smash{\begin{tabular}[t]{l}$n-1$\end{tabular}}}}%
    \put(0.8050627,0.29193413){\color[rgb]{0,0,0}\makebox(0,0)[lt]{\lineheight{1.25}\smash{\begin{tabular}[t]{l}$r+1$\end{tabular}}}}%
    \put(0,0){\includegraphics[width=\unitlength,page=9]{diamante2.pdf}}%
    \put(0.81055327,0.25449274){\color[rgb]{0,0,0}\makebox(0,0)[lt]{\lineheight{1.25}\smash{\begin{tabular}[t]{l}$n$\end{tabular}}}}%
    \put(0.80014314,0.23548549){\color[rgb]{0,0,0}\makebox(0,0)[lt]{\lineheight{1.25}\smash{\begin{tabular}[t]{l}$r+1$\end{tabular}}}}%
    \put(0.92569546,0.16977225){\color[rgb]{0,0,0}\makebox(0,0)[lt]{\lineheight{1.25}\smash{\begin{tabular}[t]{l}$n-1,n$\end{tabular}}}}%
    \put(0.92263956,0.14957709){\color[rgb]{0,0,0}\makebox(0,0)[lt]{\lineheight{1.25}\smash{\begin{tabular}[t]{l}$r+1, s+1$\end{tabular}}}}%
  \end{picture}%
\endgroup%

%% file: categoria.pdf_tex
\begingroup%
  \makeatletter%
  \providecommand\color[2][]{%
    \errmessage{(Inkscape) Color is used for the text in Inkscape, but the package 'color.sty' is not loaded}%
    \renewcommand\color[2][]{}%
  }%
  \providecommand\transparent[1]{%
    \errmessage{(Inkscape) Transparency is used (non-zero) for the text in Inkscape, but the package 'transparent.sty' is not loaded}%
    \renewcommand\transparent[1]{}%
  }%
  \providecommand\rotatebox[2]{#2}%
  \newcommand*\fsize{\dimexpr\f@size pt\relax}%
  \newcommand*\lineheight[1]{\fontsize{\fsize}{#1\fsize}\selectfont}%
  \ifx\svgwidth\undefined%
    \setlength{\unitlength}{980.20625798bp}%
    \ifx\svgscale\undefined%
      \relax%
    \else%
      \setlength{\unitlength}{\unitlength * \real{\svgscale}}%
    \fi%
  \else%
    \setlength{\unitlength}{\svgwidth}%
  \fi%
  \global\let\svgwidth\undefined%
  \global\let\svgscale\undefined%
  \makeatother%
  \begin{picture}(1,0.59357878)%
    \lineheight{1}%
    \setlength\tabcolsep{0pt}%
    \put(0,0){\includegraphics[width=\unitlength,page=1]{categoria.pdf}}%
    \put(0.13186006,0.03432821){\color[rgb]{0,0,0}\makebox(0,0)[lt]{\lineheight{0}\smash{\begin{tabular}[t]{l}$\tau$\end{tabular}}}}%
    \put(0,0){\includegraphics[width=\unitlength,page=2]{categoria.pdf}}%
    \put(0.07058727,0.43468424){\color[rgb]{0,0,0}\makebox(0,0)[lt]{\lineheight{0}\smash{\begin{tabular}[t]{l}$5$\end{tabular}}}}%
    \put(0.07230571,0.46291145){\color[rgb]{0,0,0}\makebox(0,0)[lt]{\lineheight{0}\smash{\begin{tabular}[t]{l}$4$\end{tabular}}}}%
    \put(0.11818861,0.51280048){\color[rgb]{0,0,0}\makebox(0,0)[lt]{\lineheight{0}\smash{\begin{tabular}[t]{l}$3$\end{tabular}}}}%
    \put(0.06183474,0.54890867){\color[rgb]{0,0,0}\makebox(0,0)[lt]{\lineheight{0}\smash{\begin{tabular}[t]{l}$2$\end{tabular}}}}%
    \put(0.0136486,0.50457776){\color[rgb]{0,0,0}\makebox(0,0)[lt]{\lineheight{0}\smash{\begin{tabular}[t]{l}$1$\end{tabular}}}}%
    \put(0,0){\includegraphics[width=\unitlength,page=3]{categoria.pdf}}%
    \put(0.2789101,0.56414862){\color[rgb]{0,0,0}\makebox(0,0)[lt]{\lineheight{1.25}\smash{\begin{tabular}[t]{l}$4$\end{tabular}}}}%
    \put(0.2789101,0.50293701){\color[rgb]{0,0,0}\makebox(0,0)[lt]{\lineheight{1.25}\smash{\begin{tabular}[t]{l}$5$\end{tabular}}}}%
    \put(0.32481882,0.53354282){\color[rgb]{0,0,0}\makebox(0,0)[lt]{\lineheight{1.25}\smash{\begin{tabular}[t]{l}$3$\end{tabular}}}}%
    \put(0.37072752,0.53354282){\color[rgb]{0,0,0}\makebox(0,0)[lt]{\lineheight{1.25}\smash{\begin{tabular}[t]{l}$2$\end{tabular}}}}%
    \put(0.41663623,0.53354282){\color[rgb]{0,0,0}\makebox(0,0)[lt]{\lineheight{1.25}\smash{\begin{tabular}[t]{l}$1$\end{tabular}}}}%
    \put(0,0){\includegraphics[width=\unitlength,page=4]{categoria.pdf}}%
    \put(0.23442238,0.55135976){\color[rgb]{0,0,0}\makebox(0,0)[lt]{\lineheight{1.25}\smash{\begin{tabular}[t]{l}$Q:$\end{tabular}}}}%
    \put(0,0){\includegraphics[width=\unitlength,page=5]{categoria.pdf}}%
  \end{picture}%
\endgroup%

%% file: loop.pdf_tex
\begingroup%
  \makeatletter%
  \providecommand\color[2][]{%
    \errmessage{(Inkscape) Color is used for the text in Inkscape, but the package 'color.sty' is not loaded}%
    \renewcommand\color[2][]{}%
  }%
  \providecommand\transparent[1]{%
    \errmessage{(Inkscape) Transparency is used (non-zero) for the text in Inkscape, but the package 'transparent.sty' is not loaded}%
    \renewcommand\transparent[1]{}%
  }%
  \providecommand\rotatebox[2]{#2}%
  \ifx\svgwidth\undefined%
    \setlength{\unitlength}{431.25915847bp}%
    \ifx\svgscale\undefined%
      \relax%
    \else%
      \setlength{\unitlength}{\unitlength * \real{\svgscale}}%
    \fi%
  \else%
    \setlength{\unitlength}{\svgwidth}%
  \fi%
  \global\let\svgwidth\undefined%
  \global\let\svgscale\undefined%
  \makeatother%
  \begin{picture}(1,0.11303835)%
    \put(0,0){\includegraphics[width=\unitlength,page=1]{loop.pdf}}%
    \put(0.07334532,0.00675257){\color[rgb]{0,0,0}\makebox(0,0)[lb]{\smash{$\gamma^{\bullet}$}}}%
    \put(0.08074663,0.08837401){\color[rgb]{0,0,0}\makebox(0,0)[lb]{\smash{$\gamma^{\bowtie}$}}}%
    \put(0,0){\includegraphics[width=\unitlength,page=2]{loop.pdf}}%
    \put(0.56305523,0.07909885){\color[rgb]{0,0,0}\makebox(0,0)[lb]{\smash{$l$}}}%
    \put(0.16422312,0.04199819){\color[rgb]{0,0,0}\makebox(0,0)[lb]{\smash{$p$}}}%
    \put(0.57233039,0.04199819){\color[rgb]{0,0,0}\makebox(0,0)[lb]{\smash{$p$}}}%
    \put(0,0){\includegraphics[width=\unitlength,page=3]{loop.pdf}}%
  \end{picture}%
\endgroup%

%% file: skeinrel2.pdf_tex
\begingroup%
  \makeatletter%
  \providecommand\color[2][]{%
    \errmessage{(Inkscape) Color is used for the text in Inkscape, but the package 'color.sty' is not loaded}%
    \renewcommand\color[2][]{}%
  }%
  \providecommand\transparent[1]{%
    \errmessage{(Inkscape) Transparency is used (non-zero) for the text in Inkscape, but the package 'transparent.sty' is not loaded}%
    \renewcommand\transparent[1]{}%
  }%
  \providecommand\rotatebox[2]{#2}%
  \newcommand*\fsize{\dimexpr\f@size pt\relax}%
  \newcommand*\lineheight[1]{\fontsize{\fsize}{#1\fsize}\selectfont}%
  \ifx\svgwidth\undefined%
    \setlength{\unitlength}{571.16276298bp}%
    \ifx\svgscale\undefined%
      \relax%
    \else%
      \setlength{\unitlength}{\unitlength * \real{\svgscale}}%
    \fi%
  \else%
    \setlength{\unitlength}{\svgwidth}%
  \fi%
  \global\let\svgwidth\undefined%
  \global\let\svgscale\undefined%
  \makeatother%
  \begin{picture}(1,0.26462066)%
    \lineheight{1}%
    \setlength\tabcolsep{0pt}%
    \put(0,0){\includegraphics[width=\unitlength,page=1]{skeinrel2.pdf}}%
    \put(0.09291169,0.24822377){\color[rgb]{0,0,0}\makebox(0,0)[lt]{\lineheight{0}\smash{\begin{tabular}[t]{l}$\alpha$\end{tabular}}}}%
    \put(0.07202014,0.17313971){\color[rgb]{0,0,0}\makebox(0,0)[lt]{\lineheight{0}\smash{\begin{tabular}[t]{l}$\beta$\end{tabular}}}}%
    \put(0,0){\includegraphics[width=\unitlength,page=2]{skeinrel2.pdf}}%
    \put(0.29048597,0.12565409){\color[rgb]{0,0,0}\makebox(0,0)[lt]{\lineheight{0}\smash{\begin{tabular}[t]{l}$\alpha^+ \beta$\end{tabular}}}}%
    \put(0.43304584,0.12565409){\color[rgb]{0,0,0}\makebox(0,0)[lt]{\lineheight{0}\smash{\begin{tabular}[t]{l}$\alpha^- \beta$\end{tabular}}}}%
    \put(0.00337651,0.19140339){\color[rgb]{0,0,0}\makebox(0,0)[lt]{\lineheight{0}\smash{\begin{tabular}[t]{l}(1)\end{tabular}}}}%
    \put(0,0){\includegraphics[width=\unitlength,page=3]{skeinrel2.pdf}}%
    \put(0.52879873,0.04935243){\color[rgb]{0,0,0}\makebox(0,0)[lt]{\lineheight{0}\smash{\begin{tabular}[t]{l}(4)\end{tabular}}}}%
    \put(0,0){\includegraphics[width=\unitlength,page=4]{skeinrel2.pdf}}%
    \put(0.00340607,0.04935241){\color[rgb]{0,0,0}\makebox(0,0)[lt]{\lineheight{0}\smash{\begin{tabular}[t]{l}(2)\end{tabular}}}}%
    \put(0.12379425,0.04936117){\color[rgb]{0,0,0}\makebox(0,0)[lt]{\lineheight{0}\smash{\begin{tabular}[t]{l}$\delta$\end{tabular}}}}%
    \put(0,0){\includegraphics[width=\unitlength,page=5]{skeinrel2.pdf}}%
    \put(0.15200489,0.22366922){\color[rgb]{0,0,0}\makebox(0,0)[lt]{\lineheight{0}\smash{\begin{tabular}[t]{l}$x$\end{tabular}}}}%
    \put(0,0){\includegraphics[width=\unitlength,page=6]{skeinrel2.pdf}}%
    \put(0.07714355,0.04752603){\color[rgb]{0,0,0}\makebox(0,0)[lt]{\lineheight{0}\smash{\begin{tabular}[t]{l}$x$\end{tabular}}}}%
    \put(0,0){\includegraphics[width=\unitlength,page=7]{skeinrel2.pdf}}%
    \put(0.59391068,0.19140339){\color[rgb]{0,0,0}\makebox(0,0)[lt]{\lineheight{0}\smash{\begin{tabular}[t]{l}(3)\end{tabular}}}}%
    \put(0,0){\includegraphics[width=\unitlength,page=8]{skeinrel2.pdf}}%
    \put(0.7415442,0.19140339){\color[rgb]{0,0,0}\makebox(0,0)[lt]{\lineheight{0}\smash{\begin{tabular}[t]{l}$= +2$\end{tabular}}}}%
    \put(0,0){\includegraphics[width=\unitlength,page=9]{skeinrel2.pdf}}%
    \put(0.27770232,0.05746962){\color[rgb]{0,0,0}\makebox(0,0)[lt]{\lineheight{0}\smash{\begin{tabular}[t]{l}$\zeta$\end{tabular}}}}%
    \put(0.2428831,0.00378651){\color[rgb]{0,0,0}\makebox(0,0)[lt]{\lineheight{0}\smash{\begin{tabular}[t]{l}$\gamma$\end{tabular}}}}%
    \put(0.42394307,0.06457217){\color[rgb]{0,0,0}\makebox(0,0)[lt]{\lineheight{0}\smash{\begin{tabular}[t]{l}$\lambda$\end{tabular}}}}%
    \put(0.61196694,0.10008491){\color[rgb]{0,0,0}\makebox(0,0)[lt]{\lineheight{0}\smash{\begin{tabular}[t]{l}$\alpha$\end{tabular}}}}%
    \put(0.60500311,0.02195689){\color[rgb]{0,0,0}\makebox(0,0)[lt]{\lineheight{0}\smash{\begin{tabular}[t]{l}$\beta$\end{tabular}}}}%
    \put(0.70771978,0.22793078){\color[rgb]{0,0,0}\makebox(0,0)[lt]{\lineheight{0}\smash{\begin{tabular}[t]{l}$\sigma$\end{tabular}}}}%
    \put(0.78509661,0.0451701){\color[rgb]{0,0,0}\makebox(0,0)[lt]{\lineheight{0}\smash{\begin{tabular}[t]{l}$\alpha^- \beta$\end{tabular}}}}%
    \put(0,0){\includegraphics[width=\unitlength,page=10]{skeinrel2.pdf}}%
    \put(0.9070882,0.09130496){\color[rgb]{0,0,0}\makebox(0,0)[lt]{\lineheight{0}\smash{\begin{tabular}[t]{l}$\alpha^+ \beta$\end{tabular}}}}%
  \end{picture}%
\endgroup%

%% file: prop12.pdf_tex
\begingroup%
  \makeatletter%
  \providecommand\color[2][]{%
    \errmessage{(Inkscape) Color is used for the text in Inkscape, but the package 'color.sty' is not loaded}%
    \renewcommand\color[2][]{}%
  }%
  \providecommand\transparent[1]{%
    \errmessage{(Inkscape) Transparency is used (non-zero) for the text in Inkscape, but the package 'transparent.sty' is not loaded}%
    \renewcommand\transparent[1]{}%
  }%
  \providecommand\rotatebox[2]{#2}%
  \newcommand*\fsize{\dimexpr\f@size pt\relax}%
  \newcommand*\lineheight[1]{\fontsize{\fsize}{#1\fsize}\selectfont}%
  \ifx\svgwidth\undefined%
    \setlength{\unitlength}{1056.50182034bp}%
    \ifx\svgscale\undefined%
      \relax%
    \else%
      \setlength{\unitlength}{\unitlength * \real{\svgscale}}%
    \fi%
  \else%
    \setlength{\unitlength}{\svgwidth}%
  \fi%
  \global\let\svgwidth\undefined%
  \global\let\svgscale\undefined%
  \makeatother%
  \begin{picture}(1,0.21726725)%
    \lineheight{1}%
    \setlength\tabcolsep{0pt}%
    \put(0,0){\includegraphics[width=\unitlength,page=1]{prop12.pdf}}%
    \put(0.46737596,0.17311343){\color[rgb]{0,0,0}\makebox(0,0)[lt]{\lineheight{0}\smash{\begin{tabular}[t]{l}$=$\end{tabular}}}}%
    \put(0,0){\includegraphics[width=\unitlength,page=2]{prop12.pdf}}%
    \put(0.59913153,0.17311343){\color[rgb]{0,0,0}\makebox(0,0)[lt]{\lineheight{0}\smash{\begin{tabular}[t]{l}$+$\end{tabular}}}}%
    \put(0,0){\includegraphics[width=\unitlength,page=3]{prop12.pdf}}%
    \put(0.18812555,0.18024612){\color[rgb]{0,0,0}\makebox(0,0)[lt]{\lineheight{0}\smash{\begin{tabular}[t]{l}$(a)$\end{tabular}}}}%
    \put(0.00536607,0.11223192){\color[rgb]{0,0,0}\makebox(0,0)[lt]{\lineheight{0}\smash{\begin{tabular}[t]{l}$(b)$\end{tabular}}}}%
    \put(0,0){\includegraphics[width=\unitlength,page=4]{prop12.pdf}}%
    \put(0.34016368,0.17311343){\color[rgb]{0,0,0}\makebox(0,0)[lt]{\lineheight{0}\smash{\begin{tabular}[t]{l}$=$\end{tabular}}}}%
    \put(0,0){\includegraphics[width=\unitlength,page=5]{prop12.pdf}}%
    \put(0.11139766,0.05324449){\color[rgb]{0,0,0}\makebox(0,0)[lt]{\lineheight{0}\smash{\begin{tabular}[t]{l}$=$\end{tabular}}}}%
    \put(0,0){\includegraphics[width=\unitlength,page=6]{prop12.pdf}}%
    \put(0.22346562,0.05324449){\color[rgb]{0,0,0}\makebox(0,0)[lt]{\lineheight{0}\smash{\begin{tabular}[t]{l}$+$\end{tabular}}}}%
    \put(0.3355336,0.05324449){\color[rgb]{0,0,0}\makebox(0,0)[lt]{\lineheight{0}\smash{\begin{tabular}[t]{l}$+$\end{tabular}}}}%
    \put(0.44305823,0.05324449){\color[rgb]{0,0,0}\makebox(0,0)[lt]{\lineheight{0}\smash{\begin{tabular}[t]{l}$-$\end{tabular}}}}%
    \put(-0.00085963,0.05324449){\color[rgb]{0,0,0}\makebox(0,0)[lt]{\lineheight{0}\smash{\begin{tabular}[t]{l}$2$\end{tabular}}}}%
    \put(0,0){\includegraphics[width=\unitlength,page=7]{prop12.pdf}}%
    \put(0.88830119,0.05324449){\color[rgb]{0,0,0}\makebox(0,0)[lt]{\lineheight{0}\smash{\begin{tabular}[t]{l}$-$\end{tabular}}}}%
    \put(0,0){\includegraphics[width=\unitlength,page=8]{prop12.pdf}}%
    \put(0.66842316,0.05324449){\color[rgb]{0,0,0}\makebox(0,0)[lt]{\lineheight{0}\smash{\begin{tabular}[t]{l}$+$\end{tabular}}}}%
    \put(0.5563552,0.05324449){\color[rgb]{0,0,0}\makebox(0,0)[lt]{\lineheight{0}\smash{\begin{tabular}[t]{l}$+$\end{tabular}}}}%
    \put(0.77833193,0.05324449){\color[rgb]{0,0,0}\makebox(0,0)[lt]{\lineheight{0}\smash{\begin{tabular}[t]{l}$+$\end{tabular}}}}%
    \put(0,0){\includegraphics[width=\unitlength,page=9]{prop12.pdf}}%
  \end{picture}%
\endgroup%

%% file: new-simple.pdf_tex
\begingroup%
  \makeatletter%
  \providecommand\color[2][]{%
    \errmessage{(Inkscape) Color is used for the text in Inkscape, but the package 'color.sty' is not loaded}%
    \renewcommand\color[2][]{}%
  }%
  \providecommand\transparent[1]{%
    \errmessage{(Inkscape) Transparency is used (non-zero) for the text in Inkscape, but the package 'transparent.sty' is not loaded}%
    \renewcommand\transparent[1]{}%
  }%
  \providecommand\rotatebox[2]{#2}%
  \newcommand*\fsize{\dimexpr\f@size pt\relax}%
  \newcommand*\lineheight[1]{\fontsize{\fsize}{#1\fsize}\selectfont}%
  \ifx\svgwidth\undefined%
    \setlength{\unitlength}{624.52591109bp}%
    \ifx\svgscale\undefined%
      \relax%
    \else%
      \setlength{\unitlength}{\unitlength * \real{\svgscale}}%
    \fi%
  \else%
    \setlength{\unitlength}{\svgwidth}%
  \fi%
  \global\let\svgwidth\undefined%
  \global\let\svgscale\undefined%
  \makeatother%
  \begin{picture}(1,0.31280417)%
    \lineheight{1}%
    \setlength\tabcolsep{0pt}%
    \put(0,0){\includegraphics[width=\unitlength,page=1]{new-simple.pdf}}%
    \put(0.15843785,0.23817844){\color[rgb]{0,0,0}\makebox(0,0)[lt]{\lineheight{0}\smash{\begin{tabular}[t]{l}$=$\end{tabular}}}}%
    \put(0,0){\includegraphics[width=\unitlength,page=2]{new-simple.pdf}}%
    \put(0.37876501,0.23817844){\color[rgb]{0,0,0}\makebox(0,0)[lt]{\lineheight{0}\smash{\begin{tabular}[t]{l}$+$\end{tabular}}}}%
    \put(0,0){\includegraphics[width=\unitlength,page=3]{new-simple.pdf}}%
    \put(0.58372044,0.23817844){\color[rgb]{0,0,0}\makebox(0,0)[lt]{\lineheight{0}\smash{\begin{tabular}[t]{l}$= 2$\end{tabular}}}}%
    \put(0,0){\includegraphics[width=\unitlength,page=4]{new-simple.pdf}}%
    \put(0.81429553,0.23817844){\color[rgb]{0,0,0}\makebox(0,0)[lt]{\lineheight{0}\smash{\begin{tabular}[t]{l}$+$\end{tabular}}}}%
    \put(0,0){\includegraphics[width=\unitlength,page=5]{new-simple.pdf}}%
    \put(0.15843785,0.07005089){\color[rgb]{0,0,0}\makebox(0,0)[lt]{\lineheight{0}\smash{\begin{tabular}[t]{l}$=$\end{tabular}}}}%
    \put(0,0){\includegraphics[width=\unitlength,page=6]{new-simple.pdf}}%
    \put(0.37876501,0.07005089){\color[rgb]{0,0,0}\makebox(0,0)[lt]{\lineheight{0}\smash{\begin{tabular}[t]{l}$+$\end{tabular}}}}%
    \put(0,0){\includegraphics[width=\unitlength,page=7]{new-simple.pdf}}%
    \put(0.58836397,0.07005089){\color[rgb]{0,0,0}\makebox(0,0)[lt]{\lineheight{0}\smash{\begin{tabular}[t]{l}$=$\end{tabular}}}}%
  \end{picture}%
\endgroup%

%% file: new2-simple.pdf_tex
\begingroup%
  \makeatletter%
  \providecommand\color[2][]{%
    \errmessage{(Inkscape) Color is used for the text in Inkscape, but the package 'color.sty' is not loaded}%
    \renewcommand\color[2][]{}%
  }%
  \providecommand\transparent[1]{%
    \errmessage{(Inkscape) Transparency is used (non-zero) for the text in Inkscape, but the package 'transparent.sty' is not loaded}%
    \renewcommand\transparent[1]{}%
  }%
  \providecommand\rotatebox[2]{#2}%
  \newcommand*\fsize{\dimexpr\f@size pt\relax}%
  \newcommand*\lineheight[1]{\fontsize{\fsize}{#1\fsize}\selectfont}%
  \ifx\svgwidth\undefined%
    \setlength{\unitlength}{330.14482164bp}%
    \ifx\svgscale\undefined%
      \relax%
    \else%
      \setlength{\unitlength}{\unitlength * \real{\svgscale}}%
    \fi%
  \else%
    \setlength{\unitlength}{\svgwidth}%
  \fi%
  \global\let\svgwidth\undefined%
  \global\let\svgscale\undefined%
  \makeatother%
  \begin{picture}(1,0.19253041)%
    \lineheight{1}%
    \setlength\tabcolsep{0pt}%
    \put(0,0){\includegraphics[width=\unitlength,page=1]{new2-simple.pdf}}%
    \put(0.20368391,0.09322133){\color[rgb]{0,0,0}\makebox(0,0)[lt]{\lineheight{0}\smash{\begin{tabular}[t]{l}$=$\end{tabular}}}}%
    \put(0,0){\includegraphics[width=\unitlength,page=2]{new2-simple.pdf}}%
    \put(0.48102689,0.09322133){\color[rgb]{0,0,0}\makebox(0,0)[lt]{\lineheight{0}\smash{\begin{tabular}[t]{l}$+$\end{tabular}}}}%
    \put(-0.7137631,-0.08670921){\color[rgb]{0,0,0}\makebox(0,0)[lt]{\begin{minipage}{0.02993681\unitlength}\raggedright \end{minipage}}}%
    \put(0,0){\includegraphics[width=\unitlength,page=3]{new2-simple.pdf}}%
    \put(0.75363447,0.09322133){\color[rgb]{0,0,0}\makebox(0,0)[lt]{\lineheight{0}\smash{\begin{tabular}[t]{l}$-$\end{tabular}}}}%
    \put(0,0){\includegraphics[width=\unitlength,page=4]{new2-simple.pdf}}%
  \end{picture}%
\endgroup%

%% file: new42.pdf_tex
\begingroup%
  \makeatletter%
  \providecommand\color[2][]{%
    \errmessage{(Inkscape) Color is used for the text in Inkscape, but the package 'color.sty' is not loaded}%
    \renewcommand\color[2][]{}%
  }%
  \providecommand\transparent[1]{%
    \errmessage{(Inkscape) Transparency is used (non-zero) for the text in Inkscape, but the package 'transparent.sty' is not loaded}%
    \renewcommand\transparent[1]{}%
  }%
  \providecommand\rotatebox[2]{#2}%
  \newcommand*\fsize{\dimexpr\f@size pt\relax}%
  \newcommand*\lineheight[1]{\fontsize{\fsize}{#1\fsize}\selectfont}%
  \ifx\svgwidth\undefined%
    \setlength{\unitlength}{571.4257285bp}%
    \ifx\svgscale\undefined%
      \relax%
    \else%
      \setlength{\unitlength}{\unitlength * \real{\svgscale}}%
    \fi%
  \else%
    \setlength{\unitlength}{\svgwidth}%
  \fi%
  \global\let\svgwidth\undefined%
  \global\let\svgscale\undefined%
  \makeatother%
  \begin{picture}(1,0.34187173)%
    \lineheight{1}%
    \setlength\tabcolsep{0pt}%
    \put(0,0){\includegraphics[width=\unitlength,page=1]{new42.pdf}}%
    \put(0.1705358,0.26031133){\color[rgb]{0,0,0}\makebox(0,0)[lt]{\lineheight{0}\smash{\begin{tabular}[t]{l}$=$\end{tabular}}}}%
    \put(0,0){\includegraphics[width=\unitlength,page=2]{new42.pdf}}%
    \put(0.38893693,0.0765604){\color[rgb]{0,0,0}\makebox(0,0)[lt]{\lineheight{0}\smash{\begin{tabular}[t]{l}$+$\end{tabular}}}}%
    \put(0.59648796,0.0765604){\color[rgb]{0,0,0}\makebox(0,0)[lt]{\lineheight{0}\smash{\begin{tabular}[t]{l}$-$\end{tabular}}}}%
    \put(0,0){\includegraphics[width=\unitlength,page=3]{new42.pdf}}%
    \put(0.39103689,0.26031133){\color[rgb]{0,0,0}\makebox(0,0)[lt]{\lineheight{0}\smash{\begin{tabular}[t]{l}$+$\end{tabular}}}}%
    \put(0,0){\includegraphics[width=\unitlength,page=4]{new42.pdf}}%
    \put(0.1705358,0.0765604){\color[rgb]{0,0,0}\makebox(0,0)[lt]{\lineheight{0}\smash{\begin{tabular}[t]{l}$=$\end{tabular}}}}%
    \put(0,0){\includegraphics[width=\unitlength,page=5]{new42.pdf}}%
    \put(0.80316394,0.0765604){\color[rgb]{0,0,0}\makebox(0,0)[lt]{\lineheight{0}\smash{\begin{tabular}[t]{l}$+$\end{tabular}}}}%
    \put(0,0){\includegraphics[width=\unitlength,page=6]{new42.pdf}}%
  \end{picture}%
\endgroup%

%% file: big-disk.pdf_tex
\begingroup%
  \makeatletter%
  \providecommand\color[2][]{%
    \errmessage{(Inkscape) Color is used for the text in Inkscape, but the package 'color.sty' is not loaded}%
    \renewcommand\color[2][]{}%
  }%
  \providecommand\transparent[1]{%
    \errmessage{(Inkscape) Transparency is used (non-zero) for the text in Inkscape, but the package 'transparent.sty' is not loaded}%
    \renewcommand\transparent[1]{}%
  }%
  \providecommand\rotatebox[2]{#2}%
  \newcommand*\fsize{\dimexpr\f@size pt\relax}%
  \newcommand*\lineheight[1]{\fontsize{\fsize}{#1\fsize}\selectfont}%
  \ifx\svgwidth\undefined%
    \setlength{\unitlength}{315.18397774bp}%
    \ifx\svgscale\undefined%
      \relax%
    \else%
      \setlength{\unitlength}{\unitlength * \real{\svgscale}}%
    \fi%
  \else%
    \setlength{\unitlength}{\svgwidth}%
  \fi%
  \global\let\svgwidth\undefined%
  \global\let\svgscale\undefined%
  \makeatother%
  \begin{picture}(1,0.90612461)%
    \lineheight{1}%
    \setlength\tabcolsep{0pt}%
    \put(0,0){\includegraphics[width=\unitlength,page=1]{big-disk.pdf}}%
    \put(0.31051296,0.591977){\color[rgb]{0,0,0}\makebox(0,0)[lt]{\lineheight{1.25}\smash{\begin{tabular}[t]{l}$\ldots$\end{tabular}}}}%
    \put(0.10205654,0.1393693){\color[rgb]{0,0,0}\makebox(0,0)[lt]{\lineheight{1.25}\smash{\begin{tabular}[t]{l}$a_{n-1}$\end{tabular}}}}%
    \put(0.48504996,0.02022751){\color[rgb]{0,0,0}\makebox(0,0)[lt]{\lineheight{1.25}\smash{\begin{tabular}[t]{l}$a_{n-2}$\end{tabular}}}}%
    \put(0.01116562,0.43738799){\color[rgb]{0,0,0}\makebox(0,0)[lt]{\lineheight{1.25}\smash{\begin{tabular}[t]{l}$a_n$\end{tabular}}}}%
    \put(0.11325331,0.71210352){\color[rgb]{0,0,0}\makebox(0,0)[lt]{\lineheight{1.25}\smash{\begin{tabular}[t]{l}$a_1$\end{tabular}}}}%
    \put(0.26276221,0.85389216){\color[rgb]{0,0,0}\makebox(0,0)[lt]{\lineheight{1.25}\smash{\begin{tabular}[t]{l}$a_2$\end{tabular}}}}%
    \put(0,0){\includegraphics[width=\unitlength,page=2]{big-disk.pdf}}%
    \put(0.74948977,0.09407559){\color[rgb]{0,0,0}\makebox(0,0)[lt]{\lineheight{1.25}\smash{\begin{tabular}[t]{l}$a_{n-3}$\end{tabular}}}}%
    \put(0,0){\includegraphics[width=\unitlength,page=3]{big-disk.pdf}}%
    \put(0.1705541,0.5378217){\color[rgb]{0,0,0}\makebox(0,0)[lt]{\lineheight{1.25}\smash{\begin{tabular}[t]{l}$1$\end{tabular}}}}%
    \put(0.27922806,0.69044137){\color[rgb]{0,0,0}\makebox(0,0)[lt]{\lineheight{1.25}\smash{\begin{tabular}[t]{l}$2$\end{tabular}}}}%
    \put(0.74520873,0.44263943){\color[rgb]{0,0,0}\makebox(0,0)[lt]{\lineheight{1.25}\smash{\begin{tabular}[t]{l}$n-3$\end{tabular}}}}%
    \put(0.5673787,0.21453054){\color[rgb]{0,0,0}\makebox(0,0)[lt]{\lineheight{1.25}\smash{\begin{tabular}[t]{l}$n-2$\end{tabular}}}}%
    \put(0.39119512,0.35073932){\color[rgb]{0,0,0}\makebox(0,0)[lt]{\lineheight{1.25}\smash{\begin{tabular}[t]{l}$n-1$\end{tabular}}}}%
    \put(0.40272108,0.24242845){\color[rgb]{0,0,0}\makebox(0,0)[lt]{\lineheight{1.25}\smash{\begin{tabular}[t]{l}$n$\end{tabular}}}}%
  \end{picture}%
\endgroup%

%% file: sec1.pdf_tex
\begingroup%
  \makeatletter%
  \providecommand\color[2][]{%
    \errmessage{(Inkscape) Color is used for the text in Inkscape, but the package 'color.sty' is not loaded}%
    \renewcommand\color[2][]{}%
  }%
  \providecommand\transparent[1]{%
    \errmessage{(Inkscape) Transparency is used (non-zero) for the text in Inkscape, but the package 'transparent.sty' is not loaded}%
    \renewcommand\transparent[1]{}%
  }%
  \providecommand\rotatebox[2]{#2}%
  \newcommand*\fsize{\dimexpr\f@size pt\relax}%
  \newcommand*\lineheight[1]{\fontsize{\fsize}{#1\fsize}\selectfont}%
  \ifx\svgwidth\undefined%
    \setlength{\unitlength}{863.97323488bp}%
    \ifx\svgscale\undefined%
      \relax%
    \else%
      \setlength{\unitlength}{\unitlength * \real{\svgscale}}%
    \fi%
  \else%
    \setlength{\unitlength}{\svgwidth}%
  \fi%
  \global\let\svgwidth\undefined%
  \global\let\svgscale\undefined%
  \makeatother%
  \begin{picture}(1,0.38868924)%
    \lineheight{1}%
    \setlength\tabcolsep{0pt}%
    \put(0,0){\includegraphics[width=\unitlength,page=1]{sec1.pdf}}%
    \put(0.07058595,0.34003728){\color[rgb]{0,0,0}\makebox(0,0)[lt]{\lineheight{0}\smash{\begin{tabular}[t]{l}$a_n$\end{tabular}}}}%
    \put(0.09540238,0.26502491){\color[rgb]{0,0,0}\makebox(0,0)[lt]{\lineheight{0}\smash{\begin{tabular}[t]{l}$p$\end{tabular}}}}%
    \put(0.19099473,0.33906001){\color[rgb]{0,0,0}\makebox(0,0)[lt]{\lineheight{0}\smash{\begin{tabular}[t]{l}$r$\end{tabular}}}}%
    \put(0,0){\includegraphics[width=\unitlength,page=2]{sec1.pdf}}%
    \put(0.23519498,0.3408918){\color[rgb]{0,0,0}\makebox(0,0)[lt]{\lineheight{0}\smash{\begin{tabular}[t]{l}$a_n$\end{tabular}}}}%
    \put(0.26000725,0.26502491){\color[rgb]{0,0,0}\makebox(0,0)[lt]{\lineheight{0}\smash{\begin{tabular}[t]{l}$p$\end{tabular}}}}%
    \put(0.3504035,0.33948726){\color[rgb]{0,0,0}\makebox(0,0)[lt]{\lineheight{0}\smash{\begin{tabular}[t]{l}$r$\end{tabular}}}}%
    \put(0,0){\includegraphics[width=\unitlength,page=3]{sec1.pdf}}%
    \put(0.39893384,0.33790099){\color[rgb]{0,0,0}\makebox(0,0)[lt]{\lineheight{0}\smash{\begin{tabular}[t]{l}$a_n$\end{tabular}}}}%
    \put(0.42461215,0.26502491){\color[rgb]{0,0,0}\makebox(0,0)[lt]{\lineheight{0}\smash{\begin{tabular}[t]{l}$p$\end{tabular}}}}%
    \put(0.5145754,0.33863275){\color[rgb]{0,0,0}\makebox(0,0)[lt]{\lineheight{0}\smash{\begin{tabular}[t]{l}$r$\end{tabular}}}}%
    \put(0.20456615,0.30889417){\color[rgb]{0,0,0}\makebox(0,0)[lt]{\lineheight{0}\smash{\begin{tabular}[t]{l}$\rightarrow$\end{tabular}}}}%
    \put(0.36962589,0.30936825){\color[rgb]{0,0,0}\makebox(0,0)[lt]{\lineheight{0}\smash{\begin{tabular}[t]{l}$\rightarrow$\end{tabular}}}}%
    \put(0.00243417,0.30801933){\color[rgb]{0,0,0}\makebox(0,0)[lt]{\lineheight{0}\smash{\begin{tabular}[t]{l}$(1)$\end{tabular}}}}%
    \put(0,0){\includegraphics[width=\unitlength,page=4]{sec1.pdf}}%
    \put(0.06884978,0.20986709){\color[rgb]{0,0,0}\makebox(0,0)[lt]{\lineheight{0}\smash{\begin{tabular}[t]{l}$a_n$\end{tabular}}}}%
    \put(0.09540238,0.13870007){\color[rgb]{0,0,0}\makebox(0,0)[lt]{\lineheight{0}\smash{\begin{tabular}[t]{l}$p$\end{tabular}}}}%
    \put(0.18579859,0.21487144){\color[rgb]{0,0,0}\makebox(0,0)[lt]{\lineheight{0}\smash{\begin{tabular}[t]{l}$r$\end{tabular}}}}%
    \put(0,0){\includegraphics[width=\unitlength,page=5]{sec1.pdf}}%
    \put(0.2316371,0.2090126){\color[rgb]{0,0,0}\makebox(0,0)[lt]{\lineheight{0}\smash{\begin{tabular}[t]{l}$a_n$\end{tabular}}}}%
    \put(0.26000725,0.13870007){\color[rgb]{0,0,0}\makebox(0,0)[lt]{\lineheight{0}\smash{\begin{tabular}[t]{l}$p$\end{tabular}}}}%
    \put(0.35659832,0.2054718){\color[rgb]{0,0,0}\makebox(0,0)[lt]{\lineheight{0}\smash{\begin{tabular}[t]{l}$r$\end{tabular}}}}%
    \put(0,0){\includegraphics[width=\unitlength,page=6]{sec1.pdf}}%
    \put(0.39791363,0.21074875){\color[rgb]{0,0,0}\makebox(0,0)[lt]{\lineheight{0}\smash{\begin{tabular}[t]{l}$a_n$\end{tabular}}}}%
    \put(0.42461215,0.13870007){\color[rgb]{0,0,0}\makebox(0,0)[lt]{\lineheight{0}\smash{\begin{tabular}[t]{l}$p$\end{tabular}}}}%
    \put(0.52120322,0.2054718){\color[rgb]{0,0,0}\makebox(0,0)[lt]{\lineheight{0}\smash{\begin{tabular}[t]{l}$r$\end{tabular}}}}%
    \put(0.20456615,0.18256934){\color[rgb]{0,0,0}\makebox(0,0)[lt]{\lineheight{0}\smash{\begin{tabular}[t]{l}$\rightarrow$\end{tabular}}}}%
    \put(0.36962589,0.1830434){\color[rgb]{0,0,0}\makebox(0,0)[lt]{\lineheight{0}\smash{\begin{tabular}[t]{l}$\rightarrow$\end{tabular}}}}%
    \put(0.00243417,0.18169447){\color[rgb]{0,0,0}\makebox(0,0)[lt]{\lineheight{0}\smash{\begin{tabular}[t]{l}$(2)$\end{tabular}}}}%
    \put(0,0){\includegraphics[width=\unitlength,page=7]{sec1.pdf}}%
    \put(0.06538981,0.08314209){\color[rgb]{0,0,0}\makebox(0,0)[lt]{\lineheight{0}\smash{\begin{tabular}[t]{l}$a_n$\end{tabular}}}}%
    \put(0.09540238,0.01057055){\color[rgb]{0,0,0}\makebox(0,0)[lt]{\lineheight{0}\smash{\begin{tabular}[t]{l}$s$\end{tabular}}}}%
    \put(0.18589346,0.08608882){\color[rgb]{0,0,0}\makebox(0,0)[lt]{\lineheight{0}\smash{\begin{tabular}[t]{l}$r$\end{tabular}}}}%
    \put(0,0){\includegraphics[width=\unitlength,page=8]{sec1.pdf}}%
    \put(0.35511231,0.08095157){\color[rgb]{0,0,0}\makebox(0,0)[lt]{\lineheight{0}\smash{\begin{tabular}[t]{l}$r$\end{tabular}}}}%
    \put(0.26000725,0.01057055){\color[rgb]{0,0,0}\makebox(0,0)[lt]{\lineheight{0}\smash{\begin{tabular}[t]{l}$s$\end{tabular}}}}%
    \put(0.23136546,0.08431107){\color[rgb]{0,0,0}\makebox(0,0)[lt]{\lineheight{0}\smash{\begin{tabular}[t]{l}$a_n$\end{tabular}}}}%
    \put(0,0){\includegraphics[width=\unitlength,page=9]{sec1.pdf}}%
    \put(0.39849265,0.08658726){\color[rgb]{0,0,0}\makebox(0,0)[lt]{\lineheight{0}\smash{\begin{tabular}[t]{l}$a_n$\end{tabular}}}}%
    \put(0.42461215,0.01057055){\color[rgb]{0,0,0}\makebox(0,0)[lt]{\lineheight{0}\smash{\begin{tabular}[t]{l}$s$\end{tabular}}}}%
    \put(0.51630746,0.08479684){\color[rgb]{0,0,0}\makebox(0,0)[lt]{\lineheight{0}\smash{\begin{tabular}[t]{l}$r$\end{tabular}}}}%
    \put(0.20456615,0.05624446){\color[rgb]{0,0,0}\makebox(0,0)[lt]{\lineheight{0}\smash{\begin{tabular}[t]{l}$\rightarrow$\end{tabular}}}}%
    \put(0.36962589,0.05671853){\color[rgb]{0,0,0}\makebox(0,0)[lt]{\lineheight{0}\smash{\begin{tabular}[t]{l}$\rightarrow$\end{tabular}}}}%
    \put(0.00243417,0.05536959){\color[rgb]{0,0,0}\makebox(0,0)[lt]{\lineheight{0}\smash{\begin{tabular}[t]{l}$(3)$\end{tabular}}}}%
    \put(0,0){\includegraphics[width=\unitlength,page=10]{sec1.pdf}}%
    \put(0.18658771,0.03227792){\color[rgb]{0,0,0}\makebox(0,0)[lt]{\lineheight{0}\smash{\begin{tabular}[t]{l}$p$\end{tabular}}}}%
    \put(0.51762645,0.03227792){\color[rgb]{0,0,0}\makebox(0,0)[lt]{\lineheight{0}\smash{\begin{tabular}[t]{l}$p$\end{tabular}}}}%
    \put(0,0){\includegraphics[width=\unitlength,page=11]{sec1.pdf}}%
    \put(0.35302153,0.03227792){\color[rgb]{0,0,0}\makebox(0,0)[lt]{\lineheight{0}\smash{\begin{tabular}[t]{l}$p$\end{tabular}}}}%
    \put(0.13207581,0.37718859){\color[rgb]{0,0,0}\makebox(0,0)[lt]{\lineheight{1.25}\smash{\begin{tabular}[t]{l}$\alpha$\end{tabular}}}}%
    \put(0.45957097,0.37718859){\color[rgb]{0,0,0}\makebox(0,0)[lt]{\lineheight{1.25}\smash{\begin{tabular}[t]{l}$\beta$\end{tabular}}}}%
    \put(0,0){\includegraphics[width=\unitlength,page=12]{sec1.pdf}}%
    \put(0.71784641,0.31470766){\color[rgb]{0,0,0}\makebox(0,0)[lt]{\lineheight{1.25}\smash{\begin{tabular}[t]{l}$\rightsquigarrow$\end{tabular}}}}%
    \put(0,0){\includegraphics[width=\unitlength,page=13]{sec1.pdf}}%
    \put(0.78834695,0.3695248){\color[rgb]{0,0,0}\makebox(0,0)[lt]{\lineheight{1.25}\smash{\begin{tabular}[t]{l}$\alpha^+ \beta$\end{tabular}}}}%
    \put(0.86441885,0.31470766){\color[rgb]{0,0,0}\makebox(0,0)[lt]{\lineheight{1.25}\smash{\begin{tabular}[t]{l}$=$\end{tabular}}}}%
    \put(0,0){\includegraphics[width=\unitlength,page=14]{sec1.pdf}}%
    \put(0.71784641,0.18768387){\color[rgb]{0,0,0}\makebox(0,0)[lt]{\lineheight{1.25}\smash{\begin{tabular}[t]{l}$\rightsquigarrow$\end{tabular}}}}%
    \put(0,0){\includegraphics[width=\unitlength,page=15]{sec1.pdf}}%
    \put(0.71784641,0.0658707){\color[rgb]{0,0,0}\makebox(0,0)[lt]{\lineheight{1.25}\smash{\begin{tabular}[t]{l}$\rightsquigarrow$\end{tabular}}}}%
    \put(0,0){\includegraphics[width=\unitlength,page=16]{sec1.pdf}}%
  \end{picture}%
\endgroup%

%% file: sec2.pdf_tex
\begingroup%
  \makeatletter%
  \providecommand\color[2][]{%
    \errmessage{(Inkscape) Color is used for the text in Inkscape, but the package 'color.sty' is not loaded}%
    \renewcommand\color[2][]{}%
  }%
  \providecommand\transparent[1]{%
    \errmessage{(Inkscape) Transparency is used (non-zero) for the text in Inkscape, but the package 'transparent.sty' is not loaded}%
    \renewcommand\transparent[1]{}%
  }%
  \providecommand\rotatebox[2]{#2}%
  \newcommand*\fsize{\dimexpr\f@size pt\relax}%
  \newcommand*\lineheight[1]{\fontsize{\fsize}{#1\fsize}\selectfont}%
  \ifx\svgwidth\undefined%
    \setlength{\unitlength}{704.15455736bp}%
    \ifx\svgscale\undefined%
      \relax%
    \else%
      \setlength{\unitlength}{\unitlength * \real{\svgscale}}%
    \fi%
  \else%
    \setlength{\unitlength}{\svgwidth}%
  \fi%
  \global\let\svgwidth\undefined%
  \global\let\svgscale\undefined%
  \makeatother%
  \begin{picture}(1,0.29009236)%
    \lineheight{1}%
    \setlength\tabcolsep{0pt}%
    \put(0,0){\includegraphics[width=\unitlength,page=1]{sec2.pdf}}%
    \put(0.058029,0.26450126){\color[rgb]{0,0,0}\makebox(0,0)[lt]{\lineheight{0}\smash{\begin{tabular}[t]{l}$a_n$\end{tabular}}}}%
    \put(0.08467151,0.16248431){\color[rgb]{0,0,0}\makebox(0,0)[lt]{\lineheight{0}\smash{\begin{tabular}[t]{l}$p$\end{tabular}}}}%
    \put(0.20542385,0.2478653){\color[rgb]{0,0,0}\makebox(0,0)[lt]{\lineheight{0}\smash{\begin{tabular}[t]{l}$r$\end{tabular}}}}%
    \put(0,0){\includegraphics[width=\unitlength,page=2]{sec2.pdf}}%
    \put(0.25725403,0.27089189){\color[rgb]{0,0,0}\makebox(0,0)[lt]{\lineheight{0}\smash{\begin{tabular}[t]{l}$a_n$\end{tabular}}}}%
    \put(0.28003674,0.16248431){\color[rgb]{0,0,0}\makebox(0,0)[lt]{\lineheight{0}\smash{\begin{tabular}[t]{l}$p$\end{tabular}}}}%
    \put(0,0){\includegraphics[width=\unitlength,page=3]{sec2.pdf}}%
    \put(0.45187704,0.27697823){\color[rgb]{0,0,0}\makebox(0,0)[lt]{\lineheight{0}\smash{\begin{tabular}[t]{l}$a_n$\end{tabular}}}}%
    \put(0.475402,0.16248431){\color[rgb]{0,0,0}\makebox(0,0)[lt]{\lineheight{0}\smash{\begin{tabular}[t]{l}$p$\end{tabular}}}}%
    \put(0.21699179,0.21772017){\color[rgb]{0,0,0}\makebox(0,0)[lt]{\lineheight{0}\smash{\begin{tabular}[t]{l}$\rightarrow$\end{tabular}}}}%
    \put(0.41290839,0.21831708){\color[rgb]{0,0,0}\makebox(0,0)[lt]{\lineheight{0}\smash{\begin{tabular}[t]{l}$\rightarrow$\end{tabular}}}}%
    \put(0.00939269,0.21661866){\color[rgb]{0,0,0}\makebox(0,0)[lt]{\lineheight{0}\smash{\begin{tabular}[t]{l}$(4)$\end{tabular}}}}%
    \put(0,0){\includegraphics[width=\unitlength,page=4]{sec2.pdf}}%
    \put(0.05362011,0.09357695){\color[rgb]{0,0,0}\makebox(0,0)[lt]{\lineheight{0}\smash{\begin{tabular}[t]{l}$a_n$\end{tabular}}}}%
    \put(0.08467151,0.00342831){\color[rgb]{0,0,0}\makebox(0,0)[lt]{\lineheight{0}\smash{\begin{tabular}[t]{l}$p$\end{tabular}}}}%
    \put(0.19536859,0.09678241){\color[rgb]{0,0,0}\makebox(0,0)[lt]{\lineheight{0}\smash{\begin{tabular}[t]{l}$r$\end{tabular}}}}%
    \put(0,0){\includegraphics[width=\unitlength,page=5]{sec2.pdf}}%
    \put(0.25057361,0.09722874){\color[rgb]{0,0,0}\makebox(0,0)[lt]{\lineheight{0}\smash{\begin{tabular}[t]{l}$a_n$\end{tabular}}}}%
    \put(0.28003674,0.00342831){\color[rgb]{0,0,0}\makebox(0,0)[lt]{\lineheight{0}\smash{\begin{tabular}[t]{l}$p$\end{tabular}}}}%
    \put(0.38999155,0.09663026){\color[rgb]{0,0,0}\makebox(0,0)[lt]{\lineheight{0}\smash{\begin{tabular}[t]{l}$r$\end{tabular}}}}%
    \put(0,0){\includegraphics[width=\unitlength,page=6]{sec2.pdf}}%
    \put(0.44579039,0.09920679){\color[rgb]{0,0,0}\makebox(0,0)[lt]{\lineheight{0}\smash{\begin{tabular}[t]{l}$a_n$\end{tabular}}}}%
    \put(0.475402,0.00342831){\color[rgb]{0,0,0}\makebox(0,0)[lt]{\lineheight{0}\smash{\begin{tabular}[t]{l}$p$\end{tabular}}}}%
    \put(0.58906808,0.09434788){\color[rgb]{0,0,0}\makebox(0,0)[lt]{\lineheight{0}\smash{\begin{tabular}[t]{l}$r$\end{tabular}}}}%
    \put(0.21699179,0.0586642){\color[rgb]{0,0,0}\makebox(0,0)[lt]{\lineheight{0}\smash{\begin{tabular}[t]{l}$\rightarrow$\end{tabular}}}}%
    \put(0.41290839,0.0592611){\color[rgb]{0,0,0}\makebox(0,0)[lt]{\lineheight{0}\smash{\begin{tabular}[t]{l}$\rightarrow$\end{tabular}}}}%
    \put(0.00939269,0.05756267){\color[rgb]{0,0,0}\makebox(0,0)[lt]{\lineheight{0}\smash{\begin{tabular}[t]{l}$(5)$\end{tabular}}}}%
    \put(0,0){\includegraphics[width=\unitlength,page=7]{sec2.pdf}}%
    \put(0.39412843,0.25319084){\color[rgb]{0,0,0}\makebox(0,0)[lt]{\lineheight{0}\smash{\begin{tabular}[t]{l}$r$\end{tabular}}}}%
    \put(0,0){\includegraphics[width=\unitlength,page=8]{sec2.pdf}}%
    \put(0.58991925,0.25090847){\color[rgb]{0,0,0}\makebox(0,0)[lt]{\lineheight{0}\smash{\begin{tabular}[t]{l}$r$\end{tabular}}}}%
    \put(0,0){\includegraphics[width=\unitlength,page=9]{sec2.pdf}}%
    \put(0.81221673,0.22411446){\color[rgb]{0,0,0}\makebox(0,0)[lt]{\lineheight{1.25}\smash{\begin{tabular}[t]{l}$\rightsquigarrow$\end{tabular}}}}%
    \put(0.81221673,0.07286923){\color[rgb]{0,0,0}\makebox(0,0)[lt]{\lineheight{1.25}\smash{\begin{tabular}[t]{l}$\rightsquigarrow$\end{tabular}}}}%
    \put(0,0){\includegraphics[width=\unitlength,page=10]{sec2.pdf}}%
  \end{picture}%
\endgroup%

%% file: localsmoo.pdf_tex
\begingroup%
  \makeatletter%
  \providecommand\color[2][]{%
    \errmessage{(Inkscape) Color is used for the text in Inkscape, but the package 'color.sty' is not loaded}%
    \renewcommand\color[2][]{}%
  }%
  \providecommand\transparent[1]{%
    \errmessage{(Inkscape) Transparency is used (non-zero) for the text in Inkscape, but the package 'transparent.sty' is not loaded}%
    \renewcommand\transparent[1]{}%
  }%
  \providecommand\rotatebox[2]{#2}%
  \newcommand*\fsize{\dimexpr\f@size pt\relax}%
  \newcommand*\lineheight[1]{\fontsize{\fsize}{#1\fsize}\selectfont}%
  \ifx\svgwidth\undefined%
    \setlength{\unitlength}{673.47891259bp}%
    \ifx\svgscale\undefined%
      \relax%
    \else%
      \setlength{\unitlength}{\unitlength * \real{\svgscale}}%
    \fi%
  \else%
    \setlength{\unitlength}{\svgwidth}%
  \fi%
  \global\let\svgwidth\undefined%
  \global\let\svgscale\undefined%
  \makeatother%
  \begin{picture}(1,0.31926019)%
    \lineheight{1}%
    \setlength\tabcolsep{0pt}%
    \put(0,0){\includegraphics[width=\unitlength,page=1]{localsmoo.pdf}}%
    \put(0.10458926,0.20468507){\color[rgb]{0,0,0}\makebox(0,0)[lt]{\lineheight{0}\smash{\begin{tabular}[t]{l}$x$\end{tabular}}}}%
    \put(0,0){\includegraphics[width=\unitlength,page=2]{localsmoo.pdf}}%
    \put(0.23179886,0.23322061){\color[rgb]{0,0,0}\makebox(0,0)[lt]{\lineheight{0}\smash{\begin{tabular}[t]{l}$\rightsquigarrow$\end{tabular}}}}%
    \put(0.23713061,0.07556806){\color[rgb]{0,0,0}\makebox(0,0)[lt]{\lineheight{0}\smash{\begin{tabular}[t]{l}$\rightsquigarrow$\end{tabular}}}}%
    \put(0,0){\includegraphics[width=\unitlength,page=3]{localsmoo.pdf}}%
    \put(0.10041761,0.12310645){\color[rgb]{0,0,0}\makebox(0,0)[lt]{\lineheight{0}\smash{\begin{tabular}[t]{l}$y$\end{tabular}}}}%
    \put(0,0){\includegraphics[width=\unitlength,page=4]{localsmoo.pdf}}%
    \put(0.21642717,0.25456721){\color[rgb]{0,0,0}\makebox(0,0)[lt]{\lineheight{0}\smash{\begin{tabular}[t]{l}$\alpha^+ \beta, x$\end{tabular}}}}%
    \put(0.2207754,0.10027729){\color[rgb]{0,0,0}\makebox(0,0)[lt]{\lineheight{0}\smash{\begin{tabular}[t]{l}$\alpha^+ \beta, y$\end{tabular}}}}%
    \put(0.01184666,0.22813295){\color[rgb]{0,0,0}\makebox(0,0)[lt]{\lineheight{1.25}\smash{\begin{tabular}[t]{l}$a_n$\end{tabular}}}}%
    \put(0.01914545,0.09791893){\color[rgb]{0,0,0}\makebox(0,0)[lt]{\lineheight{1.25}\smash{\begin{tabular}[t]{l}$p$\end{tabular}}}}%
    \put(0.18313503,0.22802283){\color[rgb]{0,0,0}\makebox(0,0)[lt]{\lineheight{1.25}\smash{\begin{tabular}[t]{l}$r$\end{tabular}}}}%
    \put(0.18593032,0.10660582){\color[rgb]{0,0,0}\makebox(0,0)[lt]{\lineheight{1.25}\smash{\begin{tabular}[t]{l}$s$\end{tabular}}}}%
    \put(0,0){\includegraphics[width=\unitlength,page=5]{localsmoo.pdf}}%
    \put(0.46778747,0.08384992){\color[rgb]{0,0,0}\makebox(0,0)[lt]{\lineheight{1.25}\smash{\begin{tabular}[t]{l}$=$\end{tabular}}}}%
    \put(0.64823814,0.08384992){\color[rgb]{0,0,0}\makebox(0,0)[lt]{\lineheight{1.25}\smash{\begin{tabular}[t]{l}$+$\end{tabular}}}}%
    \put(0,0){\includegraphics[width=\unitlength,page=6]{localsmoo.pdf}}%
    \put(0.81564414,0.08424823){\color[rgb]{0,0,0}\makebox(0,0)[lt]{\lineheight{1.25}\smash{\begin{tabular}[t]{l}$-$\end{tabular}}}}%
    \put(0,0){\includegraphics[width=\unitlength,page=7]{localsmoo.pdf}}%
  \end{picture}%
\endgroup%

%% file: sec3.pdf_tex
\begingroup%
  \makeatletter%
  \providecommand\color[2][]{%
    \errmessage{(Inkscape) Color is used for the text in Inkscape, but the package 'color.sty' is not loaded}%
    \renewcommand\color[2][]{}%
  }%
  \providecommand\transparent[1]{%
    \errmessage{(Inkscape) Transparency is used (non-zero) for the text in Inkscape, but the package 'transparent.sty' is not loaded}%
    \renewcommand\transparent[1]{}%
  }%
  \providecommand\rotatebox[2]{#2}%
  \newcommand*\fsize{\dimexpr\f@size pt\relax}%
  \newcommand*\lineheight[1]{\fontsize{\fsize}{#1\fsize}\selectfont}%
  \ifx\svgwidth\undefined%
    \setlength{\unitlength}{440.15460638bp}%
    \ifx\svgscale\undefined%
      \relax%
    \else%
      \setlength{\unitlength}{\unitlength * \real{\svgscale}}%
    \fi%
  \else%
    \setlength{\unitlength}{\svgwidth}%
  \fi%
  \global\let\svgwidth\undefined%
  \global\let\svgscale\undefined%
  \makeatother%
  \begin{picture}(1,0.47855608)%
    \lineheight{1}%
    \setlength\tabcolsep{0pt}%
    \put(0,0){\includegraphics[width=\unitlength,page=1]{sec3.pdf}}%
    \put(0.13358376,0.44156511){\color[rgb]{0,0,0}\makebox(0,0)[lt]{\lineheight{0}\smash{\begin{tabular}[t]{l}$a_n$\end{tabular}}}}%
    \put(0.14872658,0.29285902){\color[rgb]{0,0,0}\makebox(0,0)[lt]{\lineheight{0}\smash{\begin{tabular}[t]{l}$p$\end{tabular}}}}%
    \put(0.33301188,0.44573429){\color[rgb]{0,0,0}\makebox(0,0)[lt]{\lineheight{0}\smash{\begin{tabular}[t]{l}$r$\end{tabular}}}}%
    \put(0,0){\includegraphics[width=\unitlength,page=2]{sec3.pdf}}%
    \put(0.38078826,0.36381762){\color[rgb]{0,0,0}\makebox(0,0)[lt]{\lineheight{0}\smash{\begin{tabular}[t]{l}$\rightarrow$\end{tabular}}}}%
    \put(0.69966689,0.3647434){\color[rgb]{0,0,0}\makebox(0,0)[lt]{\lineheight{0}\smash{\begin{tabular}[t]{l}$\rightarrow$\end{tabular}}}}%
    \put(0.01016501,0.36210924){\color[rgb]{0,0,0}\makebox(0,0)[lt]{\lineheight{0}\smash{\begin{tabular}[t]{l}$(6)$\end{tabular}}}}%
    \put(0.01016501,0.08722637){\color[rgb]{0,0,0}\makebox(0,0)[lt]{\lineheight{0}\smash{\begin{tabular}[t]{l}$(7)$\end{tabular}}}}%
    \put(0,0){\includegraphics[width=\unitlength,page=3]{sec3.pdf}}%
    \put(0.32977824,0.29643035){\color[rgb]{0,0,0}\makebox(0,0)[lt]{\lineheight{0}\smash{\begin{tabular}[t]{l}$s$\end{tabular}}}}%
    \put(0,0){\includegraphics[width=\unitlength,page=4]{sec3.pdf}}%
    \put(0.13358376,0.15743143){\color[rgb]{0,0,0}\makebox(0,0)[lt]{\lineheight{0}\smash{\begin{tabular}[t]{l}$a_n$\end{tabular}}}}%
    \put(0.14872658,0.00872538){\color[rgb]{0,0,0}\makebox(0,0)[lt]{\lineheight{0}\smash{\begin{tabular}[t]{l}$p$\end{tabular}}}}%
    \put(0.33301188,0.16160056){\color[rgb]{0,0,0}\makebox(0,0)[lt]{\lineheight{0}\smash{\begin{tabular}[t]{l}$r$\end{tabular}}}}%
    \put(0,0){\includegraphics[width=\unitlength,page=5]{sec3.pdf}}%
    \put(0.38078826,0.07968392){\color[rgb]{0,0,0}\makebox(0,0)[lt]{\lineheight{0}\smash{\begin{tabular}[t]{l}$\rightarrow$\end{tabular}}}}%
    \put(0.69966689,0.08060968){\color[rgb]{0,0,0}\makebox(0,0)[lt]{\lineheight{0}\smash{\begin{tabular}[t]{l}$\rightarrow$\end{tabular}}}}%
    \put(0,0){\includegraphics[width=\unitlength,page=6]{sec3.pdf}}%
    \put(0.32977824,0.01229663){\color[rgb]{0,0,0}\makebox(0,0)[lt]{\lineheight{0}\smash{\begin{tabular}[t]{l}$s$\end{tabular}}}}%
    \put(0,0){\includegraphics[width=\unitlength,page=7]{sec3.pdf}}%
  \end{picture}%
\endgroup%

%% file: sec4.pdf_tex
\begingroup%
  \makeatletter%
  \providecommand\color[2][]{%
    \errmessage{(Inkscape) Color is used for the text in Inkscape, but the package 'color.sty' is not loaded}%
    \renewcommand\color[2][]{}%
  }%
  \providecommand\transparent[1]{%
    \errmessage{(Inkscape) Transparency is used (non-zero) for the text in Inkscape, but the package 'transparent.sty' is not loaded}%
    \renewcommand\transparent[1]{}%
  }%
  \providecommand\rotatebox[2]{#2}%
  \newcommand*\fsize{\dimexpr\f@size pt\relax}%
  \newcommand*\lineheight[1]{\fontsize{\fsize}{#1\fsize}\selectfont}%
  \ifx\svgwidth\undefined%
    \setlength{\unitlength}{441.65458619bp}%
    \ifx\svgscale\undefined%
      \relax%
    \else%
      \setlength{\unitlength}{\unitlength * \real{\svgscale}}%
    \fi%
  \else%
    \setlength{\unitlength}{\svgwidth}%
  \fi%
  \global\let\svgwidth\undefined%
  \global\let\svgscale\undefined%
  \makeatother%
  \begin{picture}(1,0.4773087)%
    \lineheight{1}%
    \setlength\tabcolsep{0pt}%
    \put(0,0){\includegraphics[width=\unitlength,page=1]{sec4.pdf}}%
    \put(0.14233005,0.44777754){\color[rgb]{0,0,0}\makebox(0,0)[lt]{\lineheight{0}\smash{\begin{tabular}[t]{l}$a_n$\end{tabular}}}}%
    \put(0.15392159,0.28418442){\color[rgb]{0,0,0}\makebox(0,0)[lt]{\lineheight{0}\smash{\begin{tabular}[t]{l}$p$\end{tabular}}}}%
    \put(0.33120961,0.44935789){\color[rgb]{0,0,0}\makebox(0,0)[lt]{\lineheight{0}\smash{\begin{tabular}[t]{l}$r$\end{tabular}}}}%
    \put(0,0){\includegraphics[width=\unitlength,page=2]{sec4.pdf}}%
    \put(0.38469904,0.35854162){\color[rgb]{0,0,0}\makebox(0,0)[lt]{\lineheight{0}\smash{\begin{tabular}[t]{l}$\rightarrow$\end{tabular}}}}%
    \put(0.70807748,0.35947154){\color[rgb]{0,0,0}\makebox(0,0)[lt]{\lineheight{0}\smash{\begin{tabular}[t]{l}$\rightarrow$\end{tabular}}}}%
    \put(0.0054865,0.35682557){\color[rgb]{0,0,0}\makebox(0,0)[lt]{\lineheight{0}\smash{\begin{tabular}[t]{l}$(8)$\end{tabular}}}}%
    \put(0,0){\includegraphics[width=\unitlength,page=3]{sec4.pdf}}%
    \put(0.0054865,0.08071462){\color[rgb]{0,0,0}\makebox(0,0)[lt]{\lineheight{0}\smash{\begin{tabular}[t]{l}$(9)$\end{tabular}}}}%
    \put(0,0){\includegraphics[width=\unitlength,page=4]{sec4.pdf}}%
    \put(0.32457112,0.28611233){\color[rgb]{0,0,0}\makebox(0,0)[lt]{\lineheight{0}\smash{\begin{tabular}[t]{l}$s$\end{tabular}}}}%
    \put(0,0){\includegraphics[width=\unitlength,page=5]{sec4.pdf}}%
    \put(0.15392159,0.00873718){\color[rgb]{0,0,0}\makebox(0,0)[lt]{\lineheight{0}\smash{\begin{tabular}[t]{l}$p$\end{tabular}}}}%
    \put(0,0){\includegraphics[width=\unitlength,page=6]{sec4.pdf}}%
    \put(0.33120961,0.1739107){\color[rgb]{0,0,0}\makebox(0,0)[lt]{\lineheight{0}\smash{\begin{tabular}[t]{l}$r$\end{tabular}}}}%
    \put(0.32457112,0.01066515){\color[rgb]{0,0,0}\makebox(0,0)[lt]{\lineheight{0}\smash{\begin{tabular}[t]{l}$s$\end{tabular}}}}%
    \put(0,0){\includegraphics[width=\unitlength,page=7]{sec4.pdf}}%
    \put(0.14233005,0.1690117){\color[rgb]{0,0,0}\makebox(0,0)[lt]{\lineheight{0}\smash{\begin{tabular}[t]{l}$a_n$\end{tabular}}}}%
    \put(0,0){\includegraphics[width=\unitlength,page=8]{sec4.pdf}}%
    \put(0.38469904,0.07977579){\color[rgb]{0,0,0}\makebox(0,0)[lt]{\lineheight{0}\smash{\begin{tabular}[t]{l}$\rightarrow$\end{tabular}}}}%
    \put(0.70807748,0.08070571){\color[rgb]{0,0,0}\makebox(0,0)[lt]{\lineheight{0}\smash{\begin{tabular}[t]{l}$\rightarrow$\end{tabular}}}}%
    \put(0,0){\includegraphics[width=\unitlength,page=9]{sec4.pdf}}%
  \end{picture}%
\endgroup%

%% file: superficie-ejemplo3.pdf_tex
\begingroup%
  \makeatletter%
  \providecommand\color[2][]{%
    \errmessage{(Inkscape) Color is used for the text in Inkscape, but the package 'color.sty' is not loaded}%
    \renewcommand\color[2][]{}%
  }%
  \providecommand\transparent[1]{%
    \errmessage{(Inkscape) Transparency is used (non-zero) for the text in Inkscape, but the package 'transparent.sty' is not loaded}%
    \renewcommand\transparent[1]{}%
  }%
  \providecommand\rotatebox[2]{#2}%
  \newcommand*\fsize{\dimexpr\f@size pt\relax}%
  \newcommand*\lineheight[1]{\fontsize{\fsize}{#1\fsize}\selectfont}%
  \ifx\svgwidth\undefined%
    \setlength{\unitlength}{270.63589355bp}%
    \ifx\svgscale\undefined%
      \relax%
    \else%
      \setlength{\unitlength}{\unitlength * \real{\svgscale}}%
    \fi%
  \else%
    \setlength{\unitlength}{\svgwidth}%
  \fi%
  \global\let\svgwidth\undefined%
  \global\let\svgscale\undefined%
  \makeatother%
  \begin{picture}(1,0.9972284)%
    \lineheight{1}%
    \setlength\tabcolsep{0pt}%
    \put(0,0){\includegraphics[width=\unitlength,page=1]{superficie-ejemplo3.pdf}}%
    \put(0.14745728,0.41824789){\color[rgb]{0,0,0}\makebox(0,0)[lt]{\lineheight{1.25}\smash{\begin{tabular}[t]{l}$\gamma_0$\end{tabular}}}}%
    \put(0,0){\includegraphics[width=\unitlength,page=2]{superficie-ejemplo3.pdf}}%
    \put(0.51049111,0.65815905){\color[rgb]{0,0,0}\makebox(0,0)[lt]{\lineheight{1.25}\smash{\begin{tabular}[t]{l}$\alpha$\end{tabular}}}}%
    \put(0,0){\includegraphics[width=\unitlength,page=3]{superficie-ejemplo3.pdf}}%
    \put(0.72664866,0.22584419){\color[rgb]{0,0,0}\makebox(0,0)[lt]{\lineheight{1.25}\smash{\begin{tabular}[t]{l}$\beta$\end{tabular}}}}%
  \end{picture}%
\endgroup%

%% file: figura-grande.pdf_tex
\begingroup%
  \makeatletter%
  \providecommand\color[2][]{%
    \errmessage{(Inkscape) Color is used for the text in Inkscape, but the package 'color.sty' is not loaded}%
    \renewcommand\color[2][]{}%
  }%
  \providecommand\transparent[1]{%
    \errmessage{(Inkscape) Transparency is used (non-zero) for the text in Inkscape, but the package 'transparent.sty' is not loaded}%
    \renewcommand\transparent[1]{}%
  }%
  \providecommand\rotatebox[2]{#2}%
  \newcommand*\fsize{\dimexpr\f@size pt\relax}%
  \newcommand*\lineheight[1]{\fontsize{\fsize}{#1\fsize}\selectfont}%
  \ifx\svgwidth\undefined%
    \setlength{\unitlength}{389.05970296bp}%
    \ifx\svgscale\undefined%
      \relax%
    \else%
      \setlength{\unitlength}{\unitlength * \real{\svgscale}}%
    \fi%
  \else%
    \setlength{\unitlength}{\svgwidth}%
  \fi%
  \global\let\svgwidth\undefined%
  \global\let\svgscale\undefined%
  \makeatother%
  \begin{picture}(1,0.87499875)%
    \lineheight{1}%
    \setlength\tabcolsep{0pt}%
    \put(0,0){\includegraphics[width=\unitlength,page=1]{figura-grande.pdf}}%
    \put(0.01157924,0.60802805){\color[rgb]{0.01176471,0.01176471,0.01176471}\makebox(0,0)[lt]{\lineheight{1.25}\smash{\begin{tabular}[t]{l}$s$\end{tabular}}}}%
    \put(0.28236658,0.85154044){\color[rgb]{0.01176471,0.01176471,0.01176471}\makebox(0,0)[lt]{\lineheight{1.25}\smash{\begin{tabular}[t]{l}$p$\end{tabular}}}}%
    \put(0.42685489,0.58635959){\color[rgb]{0.01176471,0.01176471,0.01176471}\makebox(0,0)[lt]{\lineheight{1.25}\smash{\begin{tabular}[t]{l}$n$\end{tabular}}}}%
    \put(0.32110621,0.47182626){\color[rgb]{0.01176471,0.01176471,0.01176471}\makebox(0,0)[lt]{\lineheight{1.25}\smash{\begin{tabular}[t]{l}$r$\end{tabular}}}}%
    \put(0,0){\includegraphics[width=\unitlength,page=2]{figura-grande.pdf}}%
    \put(0.54161022,0.61318719){\color[rgb]{0.01176471,0.01176471,0.01176471}\makebox(0,0)[lt]{\lineheight{1.25}\smash{\begin{tabular}[t]{l}$s$\end{tabular}}}}%
    \put(0.81410163,0.85154044){\color[rgb]{0.01176471,0.01176471,0.01176471}\makebox(0,0)[lt]{\lineheight{1.25}\smash{\begin{tabular}[t]{l}$p$\end{tabular}}}}%
    \put(0.95859023,0.58635959){\color[rgb]{0.01176471,0.01176471,0.01176471}\makebox(0,0)[lt]{\lineheight{1.25}\smash{\begin{tabular}[t]{l}$n$\end{tabular}}}}%
    \put(0.85284138,0.47182626){\color[rgb]{0.01176471,0.01176471,0.01176471}\makebox(0,0)[lt]{\lineheight{1.25}\smash{\begin{tabular}[t]{l}$r$\end{tabular}}}}%
    \put(0,0){\includegraphics[width=\unitlength,page=3]{figura-grande.pdf}}%
    \put(0.00987516,0.16402412){\color[rgb]{0.01176471,0.01176471,0.01176471}\makebox(0,0)[lt]{\lineheight{1.25}\smash{\begin{tabular}[t]{l}$s$\end{tabular}}}}%
    \put(0.28236658,0.40237737){\color[rgb]{0.01176471,0.01176471,0.01176471}\makebox(0,0)[lt]{\lineheight{1.25}\smash{\begin{tabular}[t]{l}$p$\end{tabular}}}}%
    \put(0.42685489,0.1371966){\color[rgb]{0.01176471,0.01176471,0.01176471}\makebox(0,0)[lt]{\lineheight{1.25}\smash{\begin{tabular}[t]{l}$n$\end{tabular}}}}%
    \put(0.32110621,0.02266311){\color[rgb]{0.01176471,0.01176471,0.01176471}\makebox(0,0)[lt]{\lineheight{1.25}\smash{\begin{tabular}[t]{l}$r$\end{tabular}}}}%
    \put(0,0){\includegraphics[width=\unitlength,page=4]{figura-grande.pdf}}%
    \put(0.54161022,0.16402412){\color[rgb]{0.01176471,0.01176471,0.01176471}\makebox(0,0)[lt]{\lineheight{1.25}\smash{\begin{tabular}[t]{l}$s$\end{tabular}}}}%
    \put(0.81410163,0.40237737){\color[rgb]{0.01176471,0.01176471,0.01176471}\makebox(0,0)[lt]{\lineheight{1.25}\smash{\begin{tabular}[t]{l}$p$\end{tabular}}}}%
    \put(0.95859023,0.1371966){\color[rgb]{0.01176471,0.01176471,0.01176471}\makebox(0,0)[lt]{\lineheight{1.25}\smash{\begin{tabular}[t]{l}$n$\end{tabular}}}}%
    \put(0.85284138,0.02266311){\color[rgb]{0.01176471,0.01176471,0.01176471}\makebox(0,0)[lt]{\lineheight{1.25}\smash{\begin{tabular}[t]{l}$r$\end{tabular}}}}%
    \put(0,0){\includegraphics[width=\unitlength,page=5]{figura-grande.pdf}}%
  \end{picture}%
\endgroup%